\documentclass[a4paper, twoside, 11pt]{amsart}

\usepackage[margin=.95 in]{geometry}                 
\usepackage[foot]{amsaddr}
\usepackage{graphicx}
\usepackage{amssymb}
\usepackage{mathrsfs}
\usepackage{enumitem}
\usepackage{array,multirow}
\usepackage{mathtools}
\usepackage[usenames,dvipsnames]{xcolor}
\usepackage{dsfont}
\usepackage{calligra}
\DeclareMathAlphabet{\mathcalligra}{T1}{calligra}{m}{n}
\usepackage{tikz}
\usepackage{tikz-cd}
\usetikzlibrary{shapes.misc}
\usetikzlibrary{decorations.markings}
\usepackage{etoolbox}
\usepackage{float}
\usepackage{multicol}
\usepackage{rotating}
\mathtoolsset{showonlyrefs=true}
\usepackage{arydshln}
\usetikzlibrary{positioning}
\usepackage{caption}
\usepackage{subcaption}
\usepackage{graphics}
\usepackage{threeparttable}

\usepackage{lscape}

\usepackage[T2A]{fontenc}
\usepackage[utf8]{inputenc}
\usepackage[russian, main=english]{babel}
 
\usepackage{hyphenat}
\makeatletter

\newcommand{\Rmnum}[1]{\expandafter\@slowromancap\romannumeral #1@}
\makeatother

\renewcommand\arraystretch{1.2}  
\newcommand{\To}[1]{\rule{0pt}{#1ex}}       
\newcommand{\Bo}[1]{\rule[-#1ex]{0pt}{0pt}} 

\newcommand\TikCircle[1][2.5]{\tikz[baseline=-#1]{\draw[thick](0,0)circle (0.05 cm);}}
\newcommand\TikCircleFill[1][2.5]{\tikz[baseline=-#1]{\fill[black](0,0)circle (0.05 cm);}}

\renewcommand\arraystretch{1.2}  

\newtheorem{theorem}{Theorem}[section]
\newtheorem{lemma}[theorem]{Lemma}

\newtheorem{corollary}[theorem]{Corollary}

\newtheorem{remark}[theorem]{Remark}
\newtheorem{definition}[theorem]{Definition}

\title[Homogeneous 2-nondegenerate hypersurfaces]{Homogeneous 2-nondegenerate CR manifolds of hypersurface type in low dimensions}

\author{David Sykes}\email{sykes@math.muni.cz}\address{Department of Mathematics and Statistics, Masaryk University, Kotl\'{a}\v{r}sk\'{a} 2, Brno, {611 37}, {Czech Republic}}



\begin{document}

\maketitle


\begin{abstract}
In a recent paper, the author and I. Zelenko introduce the concept of modified CR symbols for organizing local invariants of $2$-nondegenerate CR structures. In this paper, we consider homogeneous hypersurfaces in $\mathbb{C}^4$, a natural frontier in the CR hypersurface Erlangen programs, and classify up to local equivalence the locally homogeneous $2$-nondegenerate hypersufaces in $\mathbb{C}^4$ whose symmetry group dimension is maximal among all such structures with the same local invariants encoded in their respective modified symbols. In the considered dimension, we show that among homogeneous structures with given modified CR symbols, the most symmetric structures (termed model structures) are unique. The classification is then achieved indirectly through classifying the modified symbols of homogeneous hypersurfaces in $\mathbb{C}^4$, obtaining (up to local equivalence) nine model structures. The methods used to obtain this classification are then applied to find homogeneous hypersurfaces in higher dimensional spaces. In total $20$ locally non-equivalent maximally symmetric homogeneous $2$-nondegenerate  hypersurfaces are described in $\mathbb{C}^5$, and $40$ such hypersurfaces are described in $\mathbb{C}^6$, of which some have been described in other works while many are new. Lastly, two new sequences, indexed by $n$, of homogeneous $2$-nondegenerate  hypersurfaces in $\mathbb{C}^{n+1}$ are described. Notably, all examples from one of these latter sequences can be realized as left-invariant structures on nilpotent Lie groups.
\end{abstract}

\maketitle
%
%
%
%

\section{Introduction}\label{introduction}

This article's main result (Theorems \ref{ch3 main 7d theorem a} and \ref{ch3 main 7d theorem b}) is a classification up to local equivalence of $2$-nondegenerate real hypersurfaces  in $\mathbb{C}^4$ that are locally equivalent to homogeneous CR manifolds whose symmetry groups have maximal dimension relative to the local invariants encoded in their \emph{modified CR symbols}, which comprise sets that are local invariants introduced in \cite{SYKES2023108850}. In the sequel we refer to these most symmetric structures as \emph{modified symbol models} (or shortly \emph{models}). There are many more homogeneous structures than models (e.g., uncountably many homogeneous non-model structures are described by the formulas in \cite[Theorem 1.3, with $n=3$]{kolar2023new}), but these models occupy an important role in the general study of homogeneous structures due to a fundamental relationship  given in Corollary \ref{ch3 main 7d corollary} between general homogeneous $2$-nondegenerate hypersurfaces in $\mathbb{C}^4$ and the model structures classified here. Broadly, homogeneous CR manifolds have been studied in several works, including \cite{atanov2019orbits,beloshapka2011classification,beloshapka2010homogeneous,doubrov2021homogeneous,fels2008classification,labovskii1997dimensions,loboda2020holomorphically,medori2001maximally,porter2021absolute,santi2020homogeneous,SYKES2023108850}, in part for their thematic role in classical differential geometric treatments of local equivalence problems, acting as prototypical structures of which general structures are described as generalizations or deformations, as in \cite{chern1974,kolar2019complete} for example.

In particular, in \cite{kolar2019complete}, Kol\'{a}\v{r}--Kossovskiy prescribe a complete normal form for $2$-nondegenerate real hypersurfaces  in $\mathbb{C}^3$, describing them as deformations of the thoroughly studied maximally symmetric $2$-nondegenerate hypersurface in $\mathbb{C}^3$ (i.e., the tube over a future light cone). Due to dimension constraints there is just one \emph{modified CR symbol} for hypersurfaces in $\mathbb{C}^3$, which corresponds to there being just one model structure in $\mathbb{C}^3$, but already in $\mathbb{C}^4$ there are many possible modified CR symbols with corresponding maximally symmetric structures. Seeking to generalize the Kol\'{a}\v{r}--Kossovskiy normal form to higher dimensional settings, one needs the appropriate notion of \emph{model structures} to deform, and this paper's classification indeed gives a natural choice for such structures. In total, there are nine CR structures in the present classification, enumerated as types \Rmnum{1}, \Rmnum{2}, \Rmnum{3}, \Rmnum{4}.A, \Rmnum{4}.B, \Rmnum{5}.A, \Rmnum{5}.B, \Rmnum{6}, \Rmnum{7} in Table \ref{main theorem table}, of which six (types \Rmnum{4} through \Rmnum{7}) have been  described in other works \cite{gregorovic2021equivalence,labovskii1997dimensions,porter2021absolute,santi2020homogeneous}, a seventh (type \Rmnum{1}) was given by the author in \cite[Example 8.1]{SYKES2023108850}, and an eighth (type \Rmnum{2}) has the structure of a tube over a previously discovered affinely homogeneous hypersurface in $\mathbb{R}^4$ \cite[formula (1), Theorem 2]{mozey2000}. For these structures that have been previously described in the aforementioned works, references to coordinate descriptions of the hypersurfaces (even given by defining equations) are in Table \ref{main theorem table}, and in the forthcoming text \cite{gregorovic2023defining}, we derive defining equation descriptions of the new hypersufaces in the list (i.e., types \Rmnum{1} and \Rmnum{3}), which ultimately required considerable calculation and new techniques.

Hypersurfaces in $\mathbb{C}^{n+1}$ that are locally equivalent to homogeneous CR manifolds have been classified up to local equivalence for $n=1$ in \cite{cartan1932geometrieII,cartan1933geometrie} and for $n=2$ in \cite{doubrov2021homogeneous,fels2008classification,loboda2020holomorphically}, which naturally leads to our present study of structures in $\mathbb{C}^4$. In more detail, for $n=1$, noting that the homogeneous CR hypersurfaces in $\mathbb{C}^2$ that are not equivalent to a hyperplane are Levi-nondegenerate, Cartan was able to complete the classification by establishing the \emph{gap phenomenon} that homogeneous Levi-nondegenerate hypersurfaces in $\mathbb{C}^2$ with non-maximal symmetry group dimension are simply transitive (therefore having $3$-dimensional symmetry groups), and Cartan then investigated the classification of these simply transitive structures using Bianchi's classification of $3$-dimensional Lie algebras. For $n=2$, homogeneous Levi-degenerate hypersurfaces in $\mathbb{C}^3$ were classified in the major work of Fels--Kaup in \cite{fels2008classification}, where they show that all such hypersurfaces are tubes over affinely homogeneous hypersurfaces in $\mathbb{R}^3$ -- a phenomenon that does not persist in higher dimensional settings; Fels--Kaup then obtain the CR classification through application of the earlier classifications \cite{doubrov1995homogeneous, eastwood1999affine} of affinely homogeneous hypersurfaces in $\mathbb{R}^3$, where, notably, showing CR inequivalence of tubes over different affinely homogeneous hypersurfaces is crucial step in \cite{fels2008classification}. The complete classification of Levi-nondegenerate CR hypersurfaces in $\mathbb{C}^3$ was obtained more recently in \cite{doubrov2021homogeneous,loboda2020holomorphically}, building upon contributions from many research groups over the preceding decades, and we refer readers to \cite{doubrov2021homogeneous} for a historic outline of these developments $\mathbb{C}^3$.   

In higher dimensions the classification is complicated by features of Levi degeneracy. By limiting considerations to hypersurfaces that are locally equivalent to a homogeneous CR manifold whose symmetry group's dimension is maximal relative to the signature of the structure's Levi form, this limited classification for Levi-nondegenerate hypersurfaces is obtained in \cite{medori2001maximally} for all $n$. The same classification problem is solved by the main results of \cite{porter2021absolute,sykes2021maximal} for the $2$-nondegenerate hypersurfaces whose Levi form has a $1$-dimensional kernel, that is, a classification of homogeneous structures whose symmetry algebras have maximal dimension relative to their Levi form for all $n$. For Levi-degenerate structures with arbitrary Levi forms it is furthermore interesting to consider those with maximal symmetry groups relative to the local invariants encoded in maps referred to by Freeman in \cite{freeman1977local} as \emph{generalized Levi forms} -- a term with several non-equivalent definitions, so we stress that the present usage refers to the definition in \cite{freeman1977local}. 

For $2$-nondegenerate structures, these latter local invariants are also encoded in the aforementioned modified CR symbols, so the classification that we obtain gives in particular all homogeneous CR structures on hypersurfaces in $\mathbb{C}^4$ with a maximal symmetry algebra relative to the structure's generalized Levi forms, of which there are $8$ in total. The ninth structure in the classification obtained here is submaximal with respect to its generalized Levi forms despite being maximal with respect to its modified symbols, illustrating the general fact that modified CR symbols encode more local invariants than generalized Levi forms. We show that this submaximal model has a $9$-dimensional symmetry algebra and the symmetry algebra of its associated dynamical Legendrian contact structure (introduced in \cite{SYKES2023108850}) is the 14-dimensional exceptional Lie algebra $\mathfrak{g}_2$, a case of special interest because this is the first known example for which a \emph{modified symbol model's} associated DLC structure has finite dimensional symmetry group and yet the CR structure's symmetry group dimension is strictly less than that of its DLC structure. Notably, in $\mathbb{C}^4$ the non-planar homogeneous Levi-degenerate hypersurfaces consist of $2$-nondegenerate and $3$-nondegenerate structures, and  the recent \cite{kruglikov20233} classifies the $3$-nondegenerate homogeneous structures in $\mathbb{C}^4$, which further motivates our present study of the $2$-nondegenerate class.

The secondary purpose of this text is to introduce the methods that were used to obtain the classification, which can, in principle, be applied in higher dimensional settings. Fundamentally, the approach consists of analyzing the algebraic properties satisfied by a homogeneous hypersurface's modified CR symbols, which reduces to a problem of assessing consistency of a certain overdetermined algebraic system (given in \eqref{system}). Modified symbols for which this system can be solved admit reductions used to generate homogeneous CR manifolds that we refer to as flat CR manifolds (see Definition \ref{flat structure definition}). These methods are effective for finding new examples of homogeneous $2$-nondegenerate CR structures on hypersurfaces in higher dimensional spaces, namely flat structures. To illustrate this, in Sections \ref{Extending and linking  ARMS} and \ref{9-dimensional examples with rank-$1$ Levi kernel} we describe $20$ locally non-equivalent examples of $9$-dimensional homogeneous hypersurfaces obtained via these methods, of which some have been studied previously while many are new. We emphasize the terms \emph{flat} and \emph{modified symbol model} have formally different definitions -- the former being a structure uniquely defined by algebraic data, and the latter being a structure whose symmetry group dimension attains some upper bound -- although they turn out to be equivalent properties in $\mathbb{C}^4$; we also caution that while our usage of \emph{flat} is consistent with \cite{SYKES2023108850} and the context of the Tanaka theory therein, the term appears with several inequivalent meanings across related literature.

Of these $20$ $9$-dimensional examples, several are obtained from the nine $7$-dimensional modified symbol models (which are flat) via two constructions that we introduce in Section \ref{Extending and linking  ARMS} called \emph{extending} and \emph{linking} abstract reduced modified symbols. By linking and extending the flat structures in $\mathbb{C}^4$, we obtain $14$ locally homogeneous non-equivalent structures in $\mathbb{C}^5$  and $38$ such structures in $\mathbb{C}^6$. For every $p,q,\in \mathbb{N}$ and $2$-nondegenerate flat hypersurface $M$ in $\mathbb{C}^{n+1}$, we describe associated $2$-nondegenerate hypersurface-type CR structures on $M\times \mathbb{C}^{(p+q)}$, that we call the $2(p+q)$-dimensional signature $(p,q)$ extensions of $M$. These extensions are themselves flat CR structures. Thus from every $2$-nondegenerate flat hypersurface, we generate sequences of higher-dimensional homogeneous examples.  Maximally symmetric homogeneous $2$-nondegenerate hypersurfaces in $\mathbb{C}^{n+1}$ were found in \cite{porter2021absolute} for arbitrary $n$, and, for $n>4$, these structures are all  extensions of the unique maximally symmetric model in $\mathbb{C}^4$ also described in \cite[Example 1]{labovskii1997dimensions}. While all of these structures are described by Lie-theoretic means, we have developed techniques for deriving their hypersurface realizations described in terms of defining equations, which will appear in the forthcoming text \cite{gregorovic2023defining}.

Given the current paucity of known high-dimensional $2$-nondegenerate homogeneous CR hypersurface examples, it is  natural to ask if all sufficiently high-dimensional flat structures  can be constructed from low-dimensional structures via combinations of extensions and linkings -- that is, if their reduced modified symbols are \emph{indecomposable}  (in the sense of Definition \ref{indecomposable ARMS}). We address this in section \ref{Sequences of flat structures in higher dimensions}, where, by applying the same methods used to obtain the aforementioned low-dimensional examples, we obtain two new sequences of homogeneous hypersurfaces in $\mathbb{C}^{n+1}$, indexed by their CR dimension $n$, that are different from any sequence generated by extensions and linkings. One of these two sequences has the property that each example in the sequence can be described as a left-invariant structure on a nilpotent Lie group.
The other sequence is interesting in contrast to the first one because examples in both sequences share the same generalized Levi forms despite being locally non-equivalent.

\section{Preliminaries}\label{preliminaries}
In this section we introduce definitions and precursory theorems necessary to derive this paper's main results. We also introduce in Section \ref{Matrix representations of ARMS} matrix representations of certain Lie algebras that will be of fundamental importance in the subsequent analysis. Introducing such representations is necessary because there is no established structure theory for the considered class of Lie algebras.

\subsection{Definitions and precursory theorems}

So that this text is self contained and because the concepts are rather new, we introduce here minimal working definitions of \emph{CR symbols}, \emph{modified CR symbols}, and \emph{reduced modified CR symbols} of homogeneous $2$-nondegenerate CR structures, along with some of their basic properties. For a more detailed exposition and study of these objects also defined in more general settings where homogeneity is not assumed, we refer the reader to \cite{porter2021absolute,SYKES2023108850}.

Throughout the sequel, let $M$ be a real $2$-nondegenerate hypersurface in $\mathbb{C}^{n+1}$ that is locally equivalent to a homogeneous CR manifold. Let $H$ denote the tangential Cauchy--Riemann bundle of $M$ (i.e., the holomorphic part of the complexified maximal complex subbundle in $TM$), let $K\subset H$ be the Levi kernel, and let $r$ be the rank of $K$. The Levi form $\mathcal{L}_p:H_p\times H_p\to \mathbb{C} T_pM/(H_p\oplus \overline{H_p})$ descends to a nondegenerate Hermitian form $\ell_p:H_p/K_p\times H_p/K_p\to \mathbb{C} T_pM/ (H_p\oplus \overline{H_p})$ at every point $p\in M$, and for each vector $v\in K_p$ there is an $\ell_p$-self-adjoint antilinear operator $\mathrm{ad}_v:H/K\to H/K$ defined by 
\begin{align}\label{Levi kernel adjoint map}
\mathrm{ad}_v(X_p+K_p):=\left[V,\overline{X}\right]_p\pmod{\overline{H}_p\oplus K_p}
\end{align}
for all $ p\in M$, $X\in \Gamma(X)$, and $V\in \Gamma(K)$ with $V_p=v$.
One can show that this definition of $\mathrm{ad}_v$ does not depend on the choices of $V$ and depends only on the value of $X$ at $p$. Note that uniform $2$-nondegeneracy is equivalent to the map $v\mapsto \mathrm{ad}_v$ being injective on $K_p$ for all $p\in M$. A consequence of uniform $2$-nondegeneracy is that
\begin{equation}
\label{estim1}
\dbinom{n-r+1}{2}\geq r.
\end{equation}

The space 
\begin{align}\label{heisenberg decomposition}
\mathfrak{g}_{-}(p):= \mathfrak{g}_{-2,0}(p)\oplus \mathfrak{g}_{-1,-1}(p)\oplus \mathfrak{g}_{-1,1}(p)
\end{align}
with $\mathfrak{g}_{-2,0}(p):=\mathbb{C} T_pM/ (H_p\oplus \overline{H_p})$, $ \mathfrak{g}_{-1,-1}(p):=\overline{H}_p/\overline{K}_p$, and $ \mathfrak{g}_{-1,1}(p):= {H}_p/ {K}_p$, inherits the structure of a $(2n+1-2r)$-dimensional complex Hiesenberg algebra from the Levi form by defining its nontrivial brackets via the formula
\[
[v,\overline{w}]=2i\ell_p(v,w)
\quad\quad\forall\, v,w\in \mathfrak{g}_{-1,1}(p).
\]

Here, and throughout the sequel, we use the notation $\mathfrak{g}_{j}(p):=\bigoplus_{k} \mathfrak{g}_{j,k}(p)$,
where the summation is over all $k$ for which $ \mathfrak{g}_{j,k}$  has been defined. In a standard way, the Heisenberg algebra's structure confers a conformal symplectic structure on $\mathfrak{g}_{-1}(p)$, and the conformal symplectic algebra $\mathfrak{csp}\big(\mathfrak{g}_{-1}(p)\big)$ can be regarded both as an algebra of endomorphisms of $\mathfrak{g}_{-1}(p)$ and of derivations of $\mathfrak{g}_{-}(p)$.  We will switch freely between both interpretations of $\mathfrak{csp}\big(\mathfrak{g}_{-1}(p)\big)$. 

Letting $\iota$ denote the involution on $\mathfrak{g}_{-}(p)$ induced by the usual conjugation on $\mathbb{C} T_pM$, the map $\iota$ induces an involution on $\mathfrak{csp}\big(\mathfrak{g}_{-1}(p)\big)$ given by the formula
\begin{align}\label{extending the involution}
\iota(\varphi) (v):=\iota\circ \varphi\circ \iota(v) \quad\quad\forall\, \varphi\in\mathfrak{csp}\big(\mathfrak{g}_{-1}(p)\big),\, v\in \mathfrak{g}_{-}(p),
\end{align}
and hence $\iota$ extends uniquely to an antilinear involution on the Lie algebra $\mathfrak{g}_{-}(p)\rtimes \mathfrak{csp}\big(\mathfrak{g}_{-1}(p)\big)$ satisfying \eqref{extending the involution}. One can show that,
$\iota([v,w])=[\iota(v),\iota(w)]$ for all $v,w\in \mathfrak{g}_{-}(p)\rtimes \mathfrak{csp}\big(\mathfrak{g}_{-1}(p)\big)$.
For $v\in K_p$, let $\widetilde{\mathrm{ad}}_v\in \mathfrak{csp}\big(\mathfrak{g}_{-1}(p)\big)$ be the endomorphism of $\mathfrak{g}_{-1}$ given by
\[
\widetilde{\mathrm{ad}}_v(w):=
\begin{cases}
0 & \mbox{ if }w\in \mathfrak{g}_{-1,1}(p) \\
\mathrm{ad}_v(\overline{w}) & \mbox{ if }w\in \mathfrak{g}_{-1,-1}(p),
\end{cases}
\]
and define
\[
\mathfrak{g}_{0,2}(p):=\left\{\left.\widetilde{\mathrm{ad}}_v\,\right|\,v\in K_p\right\}
\quad\mbox{ and }\quad 
\mathfrak{g}_{0,-2}(p):=\iota\left(\mathfrak{g}_{0,2}(p)\right).
\]
Also define
\[
\mathfrak{g}_{0,0}(p):=\left\{\left. v\in \mathfrak{csp}\big(\mathfrak{g}_{-1}(p)\big)\,\right|\, \left[v,\mathfrak{g}_{j,k}\right]\subset\mathfrak{g}_{j,k}\, \forall\,(j,k)\in\{(-2,0),(-1,\pm1),(0,\pm2)\} \right\}.
\]

\begin{definition}[introduced in \cite{porter2021absolute}]\label{CR symbol}
The \emph{CR symbol} of the structure on $M$ at a point $p\in M$ is the bi-graded subspace $\mathfrak{g}^0(p)$ of $\mathfrak{g}_{-}(p)\rtimes \mathfrak{csp}\big(\mathfrak{g}_{-1}(p)\big)$ given by
\begin{align}\label{symbol decomposition}
\mathfrak{g}^0(p):=\mathfrak{g}_{-}(p)\oplus\mathfrak{g}_{0,-2}(p)\oplus\mathfrak{g}_{0,0}(p)\oplus\mathfrak{g}_{0,2}(p).
\end{align}
The CR symbol is \emph{regular} if it is a subalgebra of $\mathfrak{g}_{-}(p)\rtimes \mathfrak{csp}\big(\mathfrak{g}_{-1}(p)\big)$.
\end{definition}
CR symbols are basic local invariants of $2$-nondegenerate CR structures in one-to-one correspondence with generalized Levi forms. The symbol's definition organizes data from the generalized Levi forms into an algebraic structure useful for applications of Tanaka prolongation and descriptions of the structures' symmetry algebras.
\begin{remark}\label{CR symbol matrix representation}
After fixing a basis of $H_p/K_p$ and $\mathbb{C}T_pM/(H_p\oplus \overline{H})$, we can represent $\ell_p$ with respect to those bases by a matrix $H_{\ell_p}$, and similarly we can represent each operator in $\{\mathrm{ad}_v\,|\, v\in K_p\}$ by a matrix with respect to that same basis. If we then fix a basis $(v_1,\ldots, v_r)$ of $K_p$ and let $A_j$ be the matrix representing $\mathrm{ad}_{v_j}$, then the CR symbol at $p$ is given by the matrix representation
\[
(\mathrm{span}_\mathbb{R}\{H_{\ell_p}\}\,,\,\mathrm{span}_{\mathbb{C}}\{A_1,\ldots, A_r\}).
\]
\end{remark}

Definition \ref{CR symbol} remains well posed if, instead of assuming that $(M,H)$ is locally homogeneous, we assume only that $(M,H)$ is uniformly $2$-nondegenerate. Our assumption of homogeneity, however, implies that all CR symbols on $M$ are equivalent under the equivalence relation that, for $p,q\in M$, $\mathfrak{g}^0(p)\cong \mathfrak{g}^0(q)$ if there is a Lie algebra isomorphism between $ \mathfrak{g}_{-}(p)\rtimes \mathfrak{csp}\big(\mathfrak{g}_{-1}(p)\big)$ and $\mathfrak{g}_{-}(q)\rtimes \mathfrak{csp}\big(\mathfrak{g}_{-1}(q)\big)$ that restricts to an isomorphism between the symbols $\mathfrak{g}^0(p)$ and $\mathfrak{g}^0(q)$ and commutes with the involution defined on  $\mathfrak{g}^0(p)$ and $\mathfrak{g}^0(q)$. When this latter property is satisfied, we say that $M$ has a constant symbol of type $\mathfrak{g}^0$, where $\mathfrak{g}^0$ is any CR symbol such that $\mathfrak{g}^0\cong \mathfrak{g}^0(p)$ for all $p\in M$.

Going  forward let us fix a CR symbol  $\mathfrak{g}^0$ such that  $(M,H)$ has a constant symbol of type $\mathfrak{g}^0$. Such a symbol indeed exists because we are  assuming that $M$ is locally homogeneous (e.g., one can take $\mathfrak{g}^0=\mathfrak{g}^0(p)$ for some point $p\in M$). By definition, $\mathfrak{g}^0$ has an involution defined on it and has the same decomposition
\[
\mathfrak{g}^0=\mathfrak{g}_{-2}\oplus \mathfrak{g}_{-1}\oplus \mathfrak{g}_{0}=\mathfrak{g}_{-2,0}\oplus \mathfrak{g}_{-1,-1}\oplus \mathfrak{g}_{-1,1}\oplus \mathfrak{g}_{0,-2}\oplus \mathfrak{g}_{0,0}\oplus\mathfrak{g}_{0,2}
\] 
as in \eqref{heisenberg decomposition} and \eqref{symbol decomposition}, with the Heisenberg component $\mathfrak{g}_{-}:=\mathfrak{g}_{-2,0}\oplus\mathfrak{g}_{-1,-1}\oplus \mathfrak{g}_{-1,1}$.

To define the modified CR symbols of $(M,H)$, we first introduce the adapted (partial) frame bundle $\mathrm{pr}:P^0\to M$, which is the fiber bundle over $M$ whose fiber $P^0_p$ at a point $p\in M$ is given by
\[
P^0_p:=
\left\{\left.\varphi\right|_{\mathfrak{g}_{-1}}:\mathfrak{g}_{-1}\to \mathfrak{g}_{-1}(p)\,\left|\, 
\parbox{6.5cm}{ $\varphi $ is a Lie algebra isomorphism from $\mathfrak{g}_{-}$ to  $\mathfrak{g}_{-}(p)$, $\varphi(\mathfrak{g}_{i,j})\subset\mathfrak{g}_{i,j}(p)\, \forall\, i,j $, and\\ $\left(\varphi|_{\mathfrak{g}_{-1}}\right)^{-1}\circ \mathfrak{g}_{0,\pm 2}(p)\circ \left(\varphi|_{\mathfrak{g}_{-1}}\right)= \mathfrak{g}_{0,\pm 2} $}
\right.
\right\}.
\]
The distribution $(K\oplus\overline{K})\cap TM$ is involutive, and thus generates a foliation of $M$. We let $\mathcal{N}$ denote the space consisting of leaves of this foliation, and let $\pi:M\to \mathcal{N}$ denote the natural projection. Let us assume that $\mathcal{N}$ with its quotient topology has the structure of a smooth manifold, which can be achieved by shrinking $M$ (i.e., replacing $M$ by a  sufficiently small neighborhood in $M$). Notice that the differential $\pi_*$ of $\pi$ naturally identifies $\mathfrak{g}_{-1}(p)$ with a subspace in $T_{\pi(p)}\mathcal{N}$ for all $p\in M$. For a fiber $P^0_{\pi(p)}$ of $\pi\circ\mathrm{pr}:P^0\to\mathcal{N}$ any fixed $\psi\in P^0_{\pi(p)}$  we have a local embedding of $\Phi_\psi$ into the conformal symplectic group $CSp(\mathfrak{g}_{-1})$ given by
\[
\Phi_\psi(\varphi):=(\pi_*\circ \psi)^{-1}\circ \pi_*\circ \varphi\in CSp(\mathfrak{g}_{-1}),
\]
and if we apply the (left) Maurer--Cartan of form of $CSp(\mathfrak{g}_{-1})$ to the tangent space of the image of this embedding at the point $\Phi_\psi(\psi)$, the Maurer--Cartan form maps that tangent space to a subspace in $ \mathfrak{csp}(\mathfrak{g}_{-1})$ that we label as $\mathfrak{g}_0^{\mathrm{mod}}(\psi)$. In terms of these subspaces  $\mathfrak{g}_0^{\mathrm{mod}}(\psi)$, we can now define the modified symbols.
\begin{definition}[introduced in \cite{SYKES2023108850}]\label{modified CR symbol}
The \emph{modified CR symbol} of the structure on $M$ at a point $\psi\in P^0$ is the subspace $\mathfrak{g}^{0,\mathrm{mod}}(\psi)$ of $\mathfrak{g}_{-}\rtimes \mathfrak{csp}\big(\mathfrak{g}_{-1}\big)$ of the form 
\[
\mathfrak{g}^{0,\mathrm{mod}}(\psi)=\mathfrak{g}_{-}\oplus\mathfrak{g}_0^{\mathrm{mod}}(\psi)
\] 
where $\mathfrak{g}_0^{\mathrm{mod}}(\psi)$ is given through the above Maurer--Cartan form mediated construction.
\end{definition}
The modified CR symbols $\mathfrak{g}^{0,\mathrm{mod}}(\psi)$ can coincide with the CR symbol $\mathfrak{g}^{0}$, but in general the two symbols differ. There is, however, a weaker relationship between the two symbols that always holds, namely
$\mathfrak{g}_{0,0}\subset \mathfrak{g}^{0,\mathrm{mod}}(\psi)$ for all $\psi\in P^0$.
The subgroup $G_{0,0}$ in $CSp(\mathfrak{g}_{-1})$ generated by $\mathfrak{g}_{0,0}$ acts on $P^0$ giving $P^0$ the structure of a principle bundle with structure group $G_{0,0}$. Modified symbols on the orbits of $G_{0,0}$ are related by the adjoint action of $G_{0,0}$ on $\mathfrak{csp}(\mathfrak{g}_{-1})$. Specifically, for $\psi\in P^0$ and $g\in G_{0,0}$, letting $g.\psi$ denote image of $\psi$ under the structure group action of $g$, we have
\begin{align}\label{orbit of symbols}
\mathfrak{g}_0^{\mathrm{mod}}(g.\psi)=g^{-1}\circ \mathfrak{g}_0^{\mathrm{mod}}(\psi)\circ g.
\end{align} 
For a point $p\in M$, the set 
\begin{align}\label{orbit of symbols invariant}
\left\{\left.
\mathfrak{g}^{0,\mathrm{mod}}(g.\psi)
\, \right|\,
g\in G_{0,0}
\right\}=\left\{\left.
\mathfrak{g}_{-}\oplus \left( g^{-1}\circ \mathfrak{g}_0^{\mathrm{mod}}(\psi)\circ g\right)
\, \right|\,
g\in G_{0,0}
\right\},
\end{align}
given by choosing any $\psi\in P^0_p$, is a local invariant of the CR structure $H$ at $p$. 
%

We are now going to describe a reduction procedure that produces subbundles of $P^0$, called \emph{reductions}, from which we will obtain the aforementioned \emph{reduced modified CR symbols}. Among these reductions are the level sets in $P^0$ of the mapping $\psi\mapsto \mathfrak{g}^{0,\mathrm{mod}}(\psi)$.  Homogeneity of $(M,H)$  implies each such level set will project surjectively onto $M$. 
and $P^{0,\mathrm{red}}$ is a reduction of the  $G_{0,0}$-principal bundle $P^0$, whose structure group we label $G_{0,0}^{\mathrm{red}}$. We call the connected components of $P^{0,\mathrm{red}}$ reductions of $P^0$.  Note that this reduction depends on $\psi$ and that $P^0$ has many such reductions. Since we work only with connected components of the level sets, in the sequel we use $P^{0,\mathrm{red}}$ to label any such connected component and let $G_{0,0}^{\mathrm{red}}$ denote its connected structure group. We will now introduce analagous reductions of $P^{0,\mathrm{red}}$ that we also simply call reductions of $P^0$.
\begin{definition}[introduced in \cite{SYKES2023108850}]\label{reduced modified CR symbol}
Let $P^{0,\mathrm{red}}\subset P^0$ be a reduction of $P^0$. The \emph{reduced modified CR symbol} of the structure on $M$ at a point $\psi\in P^{0,\mathrm{red}}$ associated with the reduction $P^{0,\mathrm{red}}$ is the subspace $\mathfrak{g}^{0,\mathrm{red}}(\psi)$ of $\mathfrak{g}_{-}\rtimes \mathfrak{csp}\big(\mathfrak{g}_{-1}\big)$ defined by exactly the same definition given for modified symbols but with $P^{0,\mathrm{red}}$ used in place of $P^0$.
\end{definition}

Just as the reduction $P^{0,\mathrm{red}}$ of $P^0$ defined above was given as a level set of the mapping $\psi\mapsto \mathfrak{g}^{0,\mathrm{mod}}(\psi)$, we can now define reductions of $P^{0,\mathrm{red}}$ to be level sets of the mapping $\psi\mapsto \mathfrak{g}^{0,\mathrm{red}}(\psi)$. As mentioned, given any reduction $P^{0,\mathrm{red}}$ of $P^0$, we will call these subsequent reductions of $P^{0,\mathrm{red}}$ reductions of $P^0$ as well. For any such reduction, we define its associated reduced modified symbols using Definition \ref{reduced modified CR symbol}. Thus we can apply an iterative process, where we first take a connected component of a level set of the mapping $\psi\mapsto \mathfrak{g}^{0,\mathrm{mod}}(\psi)$ to obtain a reduction of $P^0$, then consider this reduction's associated reduced modified symbols as defined in Definition \ref{reduced modified CR symbol}, then take a connected component of a level set within this last reduction of the mapping $\psi\mapsto \mathfrak{g}^{0,\mathrm{red}}(\psi)$ to obtain a new reduction of $P^0$, and finally repeat these last two steps any number of times.  We call this iterative process the \emph{(geometric) reduction procedure} for $P^0$, which we apply to obtain reductions of $P^0$ whose associated reduced modified symbols are in some ways easier to study than the original modified symbols, while, nevertheless encoding all of the information about $(M,H)$ that is encoded in the original modified symbols. 
These reductions are especially useful  for the study of homogeneous hypersurfaces due to the following lemma.

\begin{lemma}\label{constant symbol reduction lemma}
If $(M,H)$ is a homogeneous $2$-nondegenerate CR manifold then there exists a reduction $P^{0,\mathrm{red}}$ of $P^0$ whose associated reduced modified symbols are all the same and invariant under the previously defined involution $\iota$ of $\mathfrak{csp}(\mathfrak{g}_{-1})$, that is, the map $\psi\mapsto \mathfrak{g}^{0,\mathrm{red}}(\psi)$ is constant on $P^{0,\mathrm{red}}$ and $\iota\left( \mathfrak{g}^{0,\mathrm{red}}(\psi)\right)= \mathfrak{g}^{0,\mathrm{red}}(\psi)$. This reduction's associated reduced modified symbol will, moreover, be a subalgebra of $\mathfrak{g}_{-}\rtimes \mathfrak{csp}(\mathfrak{g}_{-1})$.
\end{lemma}
\begin{proof}
From the homogeneity of $(M,H)$, at every step of the reduction procedure we obtain a reduction $P^{0,\mathrm{red}}$ of $P^0$ for which the set $\{\mathfrak{g}^{0,\mathrm{red}}(\psi)\,|\, \psi\in P^{0,\mathrm{red}}_p\}$ of associated reduced modified symbols on the fiber $P^{0,\mathrm{red}}_p$ over a point $p\in M$ does not depend on $M$. Therefore any level set $\widetilde{P}^{0,\mathrm{red}}$ in $P^{0,\mathrm{red}}$ of the mapping  $\psi\mapsto \mathfrak{g}^{0,\mathrm{red}}(\psi)$ will project surjectively onto $M$, and
\begin{align}\label{reduction dimensions}
\dim(M)\leq \dim\left(\widetilde{P}^{0,\mathrm{red}}\right) \leq \dim\left(P^{0,\mathrm{red}}\right).
\end{align}
Since fibers of $\widetilde{P}^{0,\mathrm{red}}$ are orbits of a closed subgroup in the structure group $G_{0,0}^{\mathrm{red}}$ of $P^{0,\mathrm{red}}$  and $G_{0,0}^{\mathrm{red}}$ is connected, the upper bound in \eqref{reduction dimensions} is a strict inequality if this closed subroup does not equal $G_{0,0}$. Therefore,  the upper bound in \eqref{reduction dimensions} is a strict inequality if $P^{0,\mathrm{red}}$ does not have a constant reduced modified symbol, and hence the reduction procedure reduces the dimension of the reduction obtained in every step unless it yields a reduction with constant reduced modified symbol in some step. The reductions obtained cannot, however, have dimension less than the lower bound in  \eqref{reduction dimensions}, so the procedure must at some step produce a reduction whose dimension is not less than the dimension of the reduction in the previous step, which implies that the reduction at that step has constant reduced modified symbol.

To ensure that the reduction procedure finally yields a reduction with a constant reduced modified symbol that is invariant under the involution $\iota$, it suffices simply at every step of the procedure to take only level sets of reduced modified symbols that are themselves invariant under $\iota$. This is indeed possible because it is shown in \cite{SYKES2023108850} that $\iota$ induces an involution on $P^0$, whose fixed point set $\Re P^0$ is manifold with real dimension equal to half the real dimension of $P^0$, and modified symbols at these points are invariant under the involution on $\mathfrak{csp}(\mathfrak{g}_{-1})$. 
Lastly, it is shown in \cite[Proposition 5.1 and Section 6]{SYKES2023108850} that if a reduction has constant reduced modified symbol then that reduced modified symbol is a subalgebra in $\mathfrak{g}_{-}\rtimes \mathfrak{csp}(\mathfrak{g}_{-1})$.
\end{proof}

Dual to the geometric reduction procedure for $P^0$ described above, there is an algebraic reduction procedure that can be applied to any modified symbol $\mathfrak{g}^{0,\mathrm{mod}}(\psi)$ associated to a point $\psi\in P^0$ that yields all of the possible modified symbols $\mathfrak{g}^{0,\mathrm{red}}(\psi)$ that the geometric reduction procedure can produce. For our purposes, since we will ultimately care only about reductions of the type in Lemma \ref{constant symbol reduction lemma}, it will suffice to describe this procedure only for $\psi\in \Re P^0$, which ensures that $\mathfrak{g}^{0,\mathrm{red}}(\psi)$ is invariant under $\iota$, so let us assume $\psi\in \Re P^0$. Consider the filtration
\begin{align}\label{algebraic red}
\mathfrak{g}^{0,\mathrm{mod}}(\psi)=V_0\supset V_1\supset\cdots\supset V_s
\end{align}
for some number $s$, where $V_{j+1}$ is the maximal subspace in $V_j$ such that $[V_{j+1},V_{j}]\subset V_{j}$,
\[
V_{j+1}:=V_{j}\cap N_{\mathfrak{csp}(\mathfrak{g}_{-1})}(V_{j})  
\quad\quad\forall \,j\in \mathbb{N}\cup\{0\}.
\]
where $N_{\mathfrak{csp}(\mathfrak{g}_{-1})}(V_{j})$ denotes the normalizer of $V_j$ in $\mathfrak{csp}(\mathfrak{g}_{-1})$.
For homogeneous structures, spaces of the form $V_s$ in \eqref{algebraic red} obtained for different choices of $\psi$ are exactly the possible reduced modified symbols associated with reductions of $P^0$. Furthermore, if $s$ is sufficiently large such that $V_{s}=V_{s+1}$ then $V_s$ will be a subalgebra of $\mathfrak{csp}(\mathfrak{g}_{-1})$ corresponding to a reduction of the type in Lemma \ref{constant symbol reduction lemma}.

As subspaces of $\mathfrak{g}_{-}\rtimes \mathfrak{csp}(\mathfrak{g}_{-1})$, reduced modified symbols share several properties, and we refer to all subspaces of $\mathfrak{g}_{-}\rtimes \mathfrak{csp}(\mathfrak{g}_{-1})$ having these properties as \emph{abstract reduced modified symbols} (or ARMS), defined as follows.
\begin{definition}\label{ARMS}
An \emph{abstract reduced modified symbol} (or ARMS) for $2$-nondegenerate hypersurface-type CR structures is a fixed choice of antilinear involution $\iota$ and decomposition of the Heisenberg algebra $\mathfrak{g}_{-}$ as in \eqref{heisenberg decomposition}, together with a subspace $\mathfrak{g}^{0,\mathrm{red}}$ of $\mathfrak{g}_{-}\rtimes \mathfrak{csp}(\mathfrak{g}_{-1})$ with the decomposition 
\begin{align}\label{modified symbol decomposition}
\mathfrak{g}^{0,\mathrm{red}}=\mathfrak{g}_{-}\oplus \mathfrak{g}_{0}^{\mathrm{red}},
\end{align}
where $\mathfrak{g}_{0}^{\mathrm{red}}$ is a subspace in $ \mathfrak{csp}(\mathfrak{g}_{-1})$ with a further decomposition
\begin{align}\label{modified symbol decomposition a}
\mathfrak{g}_{0}^{\mathrm{red}}=\mathfrak{g}_{0,-}^{\mathrm{red}}\oplus \mathfrak{g}_{0,0}^{\mathrm{red}}\oplus \mathfrak{g}_{0,+}^{\mathrm{red}}
\end{align}
satisfying (1) $\iota(\mathfrak{g}_{0}^{\mathrm{red}})=\mathfrak{g}_{0}^{\mathrm{red}}$, (2) $\iota(\mathfrak{g}_{0,-}^{\mathrm{red}})=\mathfrak{g}_{0,+}^{\mathrm{red}}$, (3) $[v,\mathfrak{g}_{-1,1}]\not\subset \mathfrak{g}_{-1,1}$ for all $ v\in \mathfrak{g}_{0,-}^{\mathrm{red}}$, (4) $[v,\mathfrak{g}_{-1,-1}]\subset \mathfrak{g}_{-1,-1} $ for all $ v\in \mathfrak{g}_{0,-}^{\mathrm{red}}$, and (5) $[v,\mathfrak{g}_{-1,i}] \subset \mathfrak{g}_{-1,i} $ for all $v\in \mathfrak{g}_{0,0}^{\mathrm{red}},\, i\in\{-1,1\}$.
\end{definition}

We stress that the decomposition in \eqref{modified symbol decomposition a} is not canonical, whereas the decomposition in  \eqref{modified symbol decomposition} is. The dimension $\dim(\mathfrak{g}_{0,+})$ is, however, independent of the splitting chosen in \eqref{modified symbol decomposition a}. Ultimately we will only care about ARMS that can be equivalent to the reduced modified symbols of the reduction in Lemma \ref{constant symbol reduction lemma}, so in the sequel we will only consider ARMS that are subalgebras of $\mathfrak{g}_{-}\rtimes \mathfrak{csp}(\mathfrak{g}_{-1})$, motivating the following terminology.\\

\begin{definition}\label{subalgebra property}We will say that an ARMS satisfies \emph{the subalgebra property} if it is a subalgebra of $\mathfrak{g}_{-}\rtimes \mathfrak{csp}(\mathfrak{g}_{-1})$.
\end{definition}

An ARMS as in Definition \ref{ARMS} satisfying the subalgebra property can be used to construct a $2$-nondegenerate CR manifold whose Levi kernel has rank $\dim(\mathfrak{g}_{0,+})$, which motivates the following definition.
\begin{definition}\label{ARMS type}
An abstract reduced modified symbol with a decomposition as in Definition \ref{ARMS} has \emph{Levi kernel dimension $r$} if $\dim(\mathfrak{g}_{0,+})=r$.
\end{definition}

We will now describe homogeneous CR manifolds generated by ARMS satisfying the subalgebra property. Let $\mathfrak{g}^{0,\mathrm{red}}$  be an ARMS satisfying the subalgebra property, and let us fix a decomposition of $\mathfrak{g}^{0,\mathrm{red}}$ as in \eqref{modified symbol decomposition} and \eqref{modified symbol decomposition a}. Define
\[
H:=\mathfrak{g}_{-1,1}\oplus \mathfrak{g}_{0,+}^{\mathrm{red}}\oplus \mathfrak{g}_{0,0}^{\mathrm{red}},
\]
let $G$ be the connected simply-connected Lie group of $\mathfrak{g}^{0,\mathrm{red}}$, let $G_{0,0}\subset M^{\mathbb{C}}$ be the subgroup generated by $\mathfrak{g}_{0,0}^{\mathrm{red}}$, let $\Re G$ be the subgroup generated by
\[
\Re(\mathfrak{g}^{0,\mathrm{red}}):=\{v\in \mathfrak{g}^{0,\mathrm{red}}\,|\, \iota(v)=v\},
\]
and let $\pi:G\to G/G_{0,0}$ be the canonical projection to the left-coset space. Letting $\vec{H}$ denote the left invariant distribution on $G$ generated by $H$, the distribution $\pi_*(\vec{H})$ defines a tangential Cauchy--Riemann bundle on $\pi(\Re G)$, defining a homogeneous CR manifold
\begin{align}\label{flat structure}
\left(\pi(\Re G), \pi_*(\vec{H})\right).
\end{align}
\begin{definition}\label{flat structure definition}
The \emph{flat structure (flat CR manifold or flat CR structure)} generated by an ARMS $\mathfrak{g}^{0,\mathrm{red}}$ satisfying the subalgebra property (Definition \ref{subalgebra property}) is the CR manifold (respectively CR manifold or CR structure) in \eqref{flat structure}.
\end{definition}
\begin{remark}\label{regular flat structures}
If a CR symbol is regular as in Definition \ref{CR symbol} then it satisfies the axioms of an ARMS given in Definition \ref{ARMS} as well as the subalgebra property (Definition \ref{subalgebra property}), and thus generates a flat structure as described above.
\end{remark}

The next two definitions address the circumstance that algebraically non-equivalent ARMS can generate the same CR structure, and we would like to sort the ARMS into equivalence classes based on the CR structures that they  generate.

\begin{definition}[completion and base]\label{ARMS completion}
Let $\mathfrak{g}^{0,\mathrm{red}}$  be an abstract reduced modified symbol with Levi kernel dimension $r$ and suppose furthermore that $\mathfrak{g}^{0,\mathrm{red}}$  is a subalgebra of  $\mathfrak{g}_{-}\rtimes \mathfrak{csp}(\mathfrak{g}_{-1})$. A \emph{completion}  (respectively \emph{base}) of $\mathfrak{g}^{0,\mathrm{red}}$  is a maximal (respectively \emph{minimal}) subalgebra of  $\mathfrak{g}_{-}\rtimes \mathfrak{csp}(\mathfrak{g}_{-1})$ containing $\mathfrak{g}^{0,\mathrm{red}}$ that is also an ARMS with Levi kernel dimension $r$.

We say that $\mathfrak{g}^{0,\mathrm{red}}$ is \emph{complete} if it equals its own completion.
\end{definition}

\begin{definition}\label{equivalence of ARMS}
Two ARMS satisfying the subalgebra property (Definition \ref{subalgebra property}) are \emph{equivalent} if they have algebraicaly equivalent completions (see Definition \ref{ARMS completion}), that is they have completions equipped with some decomposition as in \eqref{modified symbol decomposition} and \eqref{modified symbol decomposition a} between which there is a Lie algebra isomorphism preserving the decompositions and the indexing of their respective components (including the bi-grading of the Heisenberg part) and commuting with the involution $\iota$. 
\end{definition}
It is easily seen that an ARMS completion and an ARMS base both generate the same flat structure as the ARMS does, so two ARMS are equivalent if and only if they generate equivalent flat CR structures.

The classification that we derive here will consist of flat structures generated by ARMS, and it will rely heavily on theorems of \cite{porter2021absolute,SYKES2023108850}. Specifically, we will use the following theorems.
\begin{theorem}[{follows from \cite[Theorem 5.2 and section 6]{SYKES2023108850}}]\label{maximally symmetric regular models modified symbols}
If $(M,H)$  is a uniformly $2$-nondegenerate hypersurface-type CR manifold with a constant CR symbol $\mathfrak{g}^0$ and constant modified symbol on $P^0$ then its modified CR symbol coincides with $\mathfrak{g}^0$ and the CR symbol is regular. Conversely, if a CR symbol $\mathfrak{g}^0$ is regular then the flat structure generated by $\mathfrak{g}^0$ referred to in Remark \ref{regular flat structures} will have constant CR symbol and constant modified symbol on $P^0$ equivalent to $\mathfrak{g}^0$ itself.
\end{theorem}
In \cite{porter2021absolute}, a variation of Tanaka prolongation called a bi-graded prolongation is introduced and applied to obtain sharp upper bounds for the symmetry group dimension of uniformly $2$-nondegenerate hypersurface-type CR manifolds with regular CR symbols, which are referred to in the next theorem.
\begin{theorem}[corollary of {\cite[Theorem 3.2]{porter2021absolute}}]\label{maximally symmetric regular models}
If $(M,H)$ is a homogeneous $2$-nondegenerate hypersu\-rface\--type CR manifold having a regular CR symbol $\mathfrak{g}^0$ whose symmetry group's dimension attains the upper bound given in \cite[Theorem 3.1]{porter2021absolute}, then $(M,H)$ is locally equivalent to the flat structure generated by $\mathfrak{g}^0$ referred to in Remark \ref{regular flat structures}.
\end{theorem}
In \cite{SYKES2023108850}, a correspondence between CR structures and dynamical Legendrian contact structures is used to obtain sharp upper bounds for the symmetry group dimension of the uniformly $2$-nondegenerate hypersurface-type CR manifolds referred to in \cite{SYKES2023108850} as being \emph{recoverable}, a property meaning that the CR structure is uniquely determined by its corresponding dynamical Legendrian contact structure. We refer the reader  to \cite{SYKES2023108850} for a thorough description and study of this property, but for the purposes of this text we need only to know the following lemma.
\begin{lemma}\label{recoverability}
If $(M,H)$ is a homogeneous $2$-nondegenerate hypersurface-type CR manifold with a rank $1$ Levi kernel and the $\mathfrak{g}_{0,2}(p)$ component of the CR symbol (at an arbitrary point $p\in M$) is spanned by an operator in $\mathfrak{csp}\big(\mathfrak{g}_{-1}(p)\big)$ whose rank is greater than $1$, then $(M,H)$ is recoverable in the sense of \cite{SYKES2023108850}. In particular if, for such $(M,H)$, the CR symbol is non-regular then $(M,H)$ is recoverable.
\end{lemma}
These recoverable structures are the topic of the next theorem.
\begin{theorem}[corollary of {\cite[Theorem 6.2]{SYKES2023108850}}]\label{maximally symmetric recoverable models}
Let $\mathfrak{g}^{0,\mathrm{red}}$ be a fixed ARMS satisfying the subalgebra property (Definition \ref{subalgebra property}). If $(M,H)$ is a homogeneous $2$-nondegenerate hypersurface-type recoverable CR manifold whose adapted (partial) frame bundle has a reduction with a constant reduced modified symbol equivalent to $\mathfrak{g}^{0,\mathrm{red}}$ and whose symmetry group dimension attains the upper bound given in \cite[Theorem 6.2]{SYKES2023108850}, then $(M,H)$ is locally equivalent to the flat structure generated by $\mathfrak{g}^{0,\mathrm{red}}$. Explicitly, the given upper bound is the complex dimension of the universal  Tanaka prolongation of $\mathfrak{g}^{0,\mathrm{red}}$, and this bound is sharp.
\end{theorem}

From Theorems \ref{maximally symmetric regular models} and \ref{maximally symmetric recoverable models}, we see that ARMS generate the maximally symmetric homogeneous models for the structures that are either recoverable or have a constant modified symbol on $P^0$. We will see that all but one of the maximally symmetric $2$-nondegenerate hypersurfaces in $\mathbb{C}^4$ satisfy at least one of these last two properties, so it is from these theorems and analysis of ARMS that we obtain most of the classification in this paper's main result Theorem \ref{ch3 main 7d theorem a}.

\subsection{Matrix representations of ARMS}\label{Matrix representations of ARMS}

An efficient way to describe ARMS is through the following matrix representations. These are described in \cite{SYKES2023108850}, but we give here a minimal review of the topic so that this paper is self-contained. Let $\mathfrak{g}^{0,\mathrm{red}}$  be an abstract reduced modified symbol with Levi kernel dimension $r$ satisfying the subalgebra property (Definition \ref{subalgebra property}), and fix a basis
\begin{align}\label{Heisenberg basis}
(e_0,e_1,\ldots, e_{2n-2r})
\end{align}
of $\mathfrak{g}_{-}$ satisfying
\begin{align}\label{Heisenberg component bases}
\mathfrak{g}_{-2}=\mathrm{span}\{e_0\},\quad \mathfrak{g}_{-1,-1}=\mathrm{span}\{e_1,\ldots, e_{n-r}\}, \quad \iota(e_0)=e_0,
\end{align}
and $\iota(e_j)=e_{j+n-r}$ for all $j\in\{1,\ldots, n-r\}$.
Let $H_\ell$ be the Hermitian matrix satisfying $[e_j,e_{n-r+k}]=i(H_\ell)_{j,k} e_0$ for all $ j,k\in\{1,\ldots, n-r\}$, and fix a decomposition of $\mathfrak{g}^{0,\mathrm{red}}$  as in \eqref{modified symbol decomposition a}. 
There exists a set of $(n-r)\times(n- r)$ matrices $A_{1},\ldots, A_{r},\Omega_{1},\ldots, \Omega_{r}$ such that, regarding $\mathfrak{g}_{0}^{\mathrm{red}}$ as a space of endomorphisms of $\mathfrak{g}_{-1}$ represented by matrices with respect to the basis $(e_1,\ldots, e_{2n-2r})$, we can identify
\begin{align}\label{g0plus}
\mathfrak{g}_{0,+}^{\mathrm{red}}=
\mathrm{span}_{\mathbb{C}}\left\{
\left.
\left(\begin{array}{cc}
\Omega_i & A_i \\
0 & -{H_\ell}^{-1}\Omega_i^{T} H_\ell
\end{array}\right)
\,\right|\,
i\in\{1,\ldots, n-r\}
\right\}.
\end{align}
Accordingly, applying \eqref{extending the involution}, one obtains
\begin{align}\label{g0minus}
\mathfrak{g}_{0,-}^{\mathrm{red}}=
\mathrm{span}_{\mathbb{C}}\left\{
\left.
\left(\begin{array}{cc}
 -\overline{H_\ell}^{-1}\Omega_i^{*} \overline{H_\ell}& 0 \\
\overline{A_i} & \overline{\Omega_i} 
\end{array}\right)
\,\right|\,
i\in\{1,\ldots, n-r\}
\right\}.
\end{align}
To describe $\mathfrak{g}_{0,0}^{\mathrm{red}}$, there is some space $\mathscr{A}_0$ of $(n-r)\times(n-r)$ matrices such that 
\begin{align}\label{g00}
X
\subset \mathfrak{g}_{0,0}^{\mathrm{red}} \subset
X\cup \mathrm{span}_{\mathbb{C}}\{I\}
\quad\mbox{ where }\quad
X=\left\{
\left.
\left(\begin{array}{cc}
\alpha& 0 \\
0& -{H_\ell}^{-1}\alpha^{T} H_\ell
\end{array}\right)
\right|
\alpha\in\mathscr{A}_0
\right\}.
\end{align}
The space $\mathscr{A}_0$ is indeed uniquely determined by this last property. To describe it more explicitly, consider the Lie algebras of $(n-r)\times(n-r)$ matrices $\alpha$ satisfying
\begin{align}\label{firstalgebra}
\alpha A_iH_\ell ^{-1} +   A_iH_\ell ^{-1}\alpha^T \in \mathrm{span}\{A_jH_\ell ^{-1}\}_{j=1}^{r} \quad\forall\,i\in\{1,\ldots,\mathrm r\}
\end{align}
and
\begin{align}\label{secondalgebra}
  \alpha^TH_\ell \overline{A_i}+H_\ell \overline{A_i} \alpha\in \mathrm{span}\{H_\ell \overline{A_j}\}_{j=1}^{r}\quad\forall\,i\in\{1,\ldots,r\},
\end{align}
respectively, and define the algebra $\mathscr{A}$ to be their intersection, that is,
\begin{align}\label{intersection algebra}
\mathscr{A}:=\left\{\alpha\,\left| \,\mbox{$\alpha$ satisfies \eqref{firstalgebra} and \eqref{secondalgebra}}\right.\right\}.
\end{align}
The space $\mathscr{A}_{0}$ is a subspace of  $\mathscr{A}$. For our purposes, one can always assume that $I\in  \mathfrak{g}_{0,0}^{\mathrm{red}}$ because it does not change the flat CR structure generated by  $\mathfrak{g}^{0,\mathrm{red}}$. We summarize this representation with the following lemma.
\begin{lemma}\label{matrix representation lemma}
Each ARMS (even without imposing the subalgebra property of Definition \ref{subalgebra property}) is determined by a tuple $(H_\ell,A_{1},\ldots, A_{r},\Omega_{1},\ldots, \Omega_{r},\mathscr{A}_0)$ consisting of $2r+1$ $(n-r)\times (n-r)$ matrices and a subspace of the vector space $\mathscr{A}$ defined in \eqref{intersection algebra}.
\end{lemma}
This matrix representation of ARMS describes a general ARMS satisfying the axioms in Definition \ref{ARMS}. A matrix representation of an ARMS satisfying the subalgebra property (Definition \ref{subalgebra property}) has the additional properties described in the following lemma.

\begin{lemma}[compare to {\cite[Proposition 5.4]{SYKES2023108850}}]\label{system of conditions for homogeneity lemma}
The ARMS $\mathfrak g^{0, \mathrm{red}}$ is a Lie subalgebra of $\mathfrak g_-\rtimes \mathfrak{csp}(\mathfrak{g}_{-1})$ if and only if $\mathscr{A}_0$ is a subalgebra of $\mathscr{A}$ and
there exist coefficients $\eta_{\alpha,i}^s\in\mathbb{C} $ and $\mu_{i,j}^s\in\mathbb{C}$ indexed by $\alpha\in\mathscr{A}_0$ and $i,j,s\in\{1,\ldots,\mathrm{rank}\,K\}$ such that the  system of relations
\begin{align}\label{system}
\left\{\mbox{
\begin{minipage}{.8\textwidth} 
(i)\quad\quad$\displaystyle \,\,  \alpha A_iH_\ell ^{-1} +   A_iH_\ell ^{-1}\alpha^T  = \sum_{s=1}^{r}\eta_{\alpha,i}^s A_sH_\ell^{-1}$\\
(ii)\quad\quad$\displaystyle  [\alpha,\Omega_i]-\sum_{s=1}^{r}\eta_{\alpha,i}^s\Omega_s\in\mathscr{A}_0 $\\
(iii)\quad\quad$\displaystyle \Omega_j^TH_\ell\overline{A_i}+H_\ell\overline{A_i} \Omega_j=\sum_{s=1}^{r} \mu_{i,j}^s H_\ell\overline{A_s} $\\
(iv)\quad\quad$\displaystyle \left[\overline{H_\ell^{-1}\Omega_i^T H_\ell},\Omega_j \right]+A_j\overline{A_i}-\sum_{s=1}^{r}\left(\overline{\mu_{i,j}^s}\Omega_s +\mu_{j,i}^s \overline{H_\ell^{-1}\Omega_s^T H_\ell}\right) \in\mathscr{A}_0 $\\
\end{minipage}
}
\right.
\end{align}
holds for all $ \alpha\in\mathscr{A}_0$ and $i,j\in\{1,\ldots,\mathrm{rank}\,K\}$. 
\end{lemma}

\begin{remark}\label{symbol_as_tuple}
The matrices $A_{1},\ldots, A_{r}$, $H_\ell$ determine the CR symbol of the flat structure generated by the ARMS  $\mathfrak{g}^{0,\mathrm{red}}$. We therefore say that   $\mathfrak{g}^{0,\mathrm{red}}$ corresponds to the CR symbol of its flat structure.
\end{remark}

We have the following useful characterizations of regular symbols.
\begin{lemma}[{\cite[Remark 5.3]{SYKES2023108850}} and part of Lemma {\cite[Lemma 4.3]{sykes2021maximal}}]\label{regular CR symbol}
The flat structure generated by the ARMS given by 
\[
(H_\ell,A_{1},\ldots, A_{r},\Omega_{1},\ldots, \Omega_{r},\mathscr{A}_0)
\] 
has a regular CR symbol if and only if
\[
A_i\overline{A_j}A_k+A_k\overline{A_j}A_i\in\mathrm{span}\{A_1,\ldots, A_{n-r}\}\quad\quad\forall\, i,j,k.
\]
Furthermore, if an ARMS satisfying the subalgebra property (Definition \ref{subalgebra property}) having Levi kernel dimension $1$  is encoded by $(H_\ell,A_1,\Omega_1, \mathscr{A}_0)$ with $\Omega_1\in \mathscr{A}$ then this ARMS corresponds to a regular CR symbol (in the sense of Remark \ref{symbol_as_tuple}).
\end{lemma}

\section{The classification of modified symbol models in $\mathbb{C}^4$}\label{The classification of flat structures in C4 generated by ARMS}
In this section we classify all ARMS that are reductions of modified CR symbols associated with homogeneous models on $7$-dimensional manifolds and show that the modified symbol models in $\mathbb{C}^4$ (i.e., homogeneous $2$-nondegenerate hypersurfaces in $\mathbb{C}^4$ that are maximally symmetric relative to their modified symbols) are locally equivalent to the flat structures in $\mathbb{C}^4$  generated by such ARMS. Hence the classification of modified symbol models follows from the classification of ARMS generating flat structures. Since we set $\dim(M)=2n+1$ and $\mathrm{rank}\, K=r$, in this section we have $n=3$.  Note that by \eqref{estim1} with $n=3$, we get that $r=1$, which corresponds to the fact that all uniformly $2$-nondegenerate 7-dimensional CR manifolds have a rank $1$ Levi kernel. Accordingly, each of the reduced modified symbols that we are classifying in this section is determined by a tuple $\{H_\ell,A_1,\Omega_1, \mathscr{A}_0\}$ satisfying the system in \eqref{system} with $r=1$. Here we have written $A_1$ and $\Omega_1$ to match the notation of Section \ref{Matrix representations of ARMS}, but for convenience let us omit the subscripts because they are unnecessary in this case with $r=1$.

Since $n-r=2$, $H_\ell$ and $A$ are $2\times 2$ matrices representing a nondegenerate Hermitian form $\ell$ and an $\ell$-selfadjoint antilinear operator, and, by changing a basis to bring such a pair to the canonical form of \cite[Theorem 2.2]{sykes2020canonical}, we can assume (after possibly rescaling $H_\ell$ and $A$ by different real coefficients) that they have one of the forms
\bgroup
\renewcommand\arraystretch{1}  
\begin{align}\label{ch3 7d case H and A formula 1}
H_\ell=
\left(
\begin{array}{cc}
1&0\\
0&\epsilon
\end{array}
\right)
\quad\mbox{ and }\quad
A=
\left(
\begin{array}{cc}
1&0\\
0&0
\end{array}
\right),
\quad\mbox{ for some } \epsilon=\pm1,
\end{align}
\begin{align}\label{ch3 7d case H and A formula 2}
H_\ell=
\left(
\begin{array}{cc}
1&0\\
0&\epsilon
\end{array}
\right)
\quad\mbox{ and }\quad
A= I\quad\mbox{ for some }  \epsilon=\pm1,
\end{align}
\begin{align}\label{ch3 7d case H and A formula 3}
H_\ell=
\left(
\begin{array}{cc}
0&1\\
1&0
\end{array}
\right)
\quad\mbox{ and }\quad
A=
\left(
\begin{array}{cc}
0&1\\
0&0
\end{array}
\right),
\end{align}
\begin{align}\label{ch3 7d case H and A formula 4}
H_\ell=
\left(
\begin{array}{cc}
0&1\\
1&0
\end{array}
\right)
\quad\mbox{ and }\quad
A=
\left(
\begin{array}{cc}
0&-1\\
1&0
\end{array}
\right),
\end{align}
\begin{align}\label{ch3 7d case H and A formula 5}
H_\ell=
\left(
\begin{array}{cc}
0&1\\
1&0
\end{array}
\right)
\quad\mbox{ and }\quad
A=
\left(
\begin{array}{cc}
0&i\\
1&0
\end{array}
\right),
\end{align}
\begin{align}\label{ch3 7d case H and A formula 6}
H_\ell=
\left(
\begin{array}{cc}
0&1\\
1&0
\end{array}
\right)
\quad\mbox{ and }\quad
A=
\left(
\begin{array}{cc}
0&e^{i\theta}\\
1&0
\end{array}
\right)
 \mbox{ for some } \theta\in\left(0,\tfrac{\pi}{2}\right)\cup\left(\tfrac{\pi}{2},\pi\right),
\end{align}
\begin{align}\label{ch3 7d case H and A formula 7}
H_\ell=
\left(
\begin{array}{cc}
1&0\\
0&\epsilon
\end{array}
\right)
\quad\mbox{ and }\quad
A=
\left(
\begin{array}{cc}
1&0\\
0&\lambda
\end{array}
\right)
\quad\mbox{ for some } \epsilon=\pm1,\, \lambda>1,
\end{align}
or
\begin{align}\label{ch3 7d case H and A formula 8}
H_\ell=
\left(
\begin{array}{cc}
0&1\\
1&0
\end{array}
\right)
\quad\mbox{ and }\quad
A=
\left(
\begin{array}{cc}
1&1\\
0&1
\end{array}
\right)
.
\end{align}
\egroup
These possible forms for the pair $(H_\ell, A)$ are ordered above to highlight a few key patterns. By Lemma \ref{regular CR symbol}, the CR symbols corresponding to  \eqref{ch3 7d case H and A formula 1}, \eqref{ch3 7d case H and A formula 2}, \eqref{ch3 7d case H and A formula 3}, and \eqref{ch3 7d case H and A formula 4} are regular and thus of the type classified in \cite{porter2021absolute}. Therefore, as noted in Remark \ref{regular flat structures},  \eqref{ch3 7d case H and A formula 1}, \eqref{ch3 7d case H and A formula 2}, \eqref{ch3 7d case H and A formula 3}, and \eqref{ch3 7d case H and A formula 4} are all associated with homogeneous models, and it remains for us to determine which ARMS satisfying the subalgebra property (Definition \ref{subalgebra property}) exist corresponding to these four cases. We furthermore need to determine which ARMS if any satisfy the subalgebra property and correspond to \eqref{ch3 7d case H and A formula 5}, \eqref{ch3 7d case H and A formula 6}, \eqref{ch3 7d case H and A formula 7}, or \eqref{ch3 7d case H and A formula 8}. Each such ARMS will generate a flat structure, and we will show that this flat structure is the unique modified symbol model having the same modified symbols. The conclusions of this analysis are summarized below in Theorems \ref{ch3 main 7d theorem a} and  \ref{ch3 main 7d theorem b}.

\begin{theorem}\label{ch3 main 7d theorem a}   Up to local equivalence, there are nine $7$-dimensional $2$-nondegenerate flat structures generated by ARMS (as described in Section \ref{preliminaries}). They are respectively generated by each of the nine ARMS described in Table  \ref{main theorem table}. In particular:
\begin{enumerate}
\item There exist three equivalence classes of ARMS satisfying the subalgebra property (Definition \ref{subalgebra property}) corresponding to \eqref{ch3 7d case H and A formula 1}, one for $\epsilon=1$ and two for $\epsilon=-1$. These are represented by types \Rmnum{3}, \Rmnum{4}.A, and \Rmnum{4}.B in Table \ref{main theorem table}.
\item There exist two equivalence classes of ARMS satisfying the subalgebra property (Definition \ref{subalgebra property}) corresponding to \eqref{ch3 7d case H and A formula 2}, one for each parameter setting of $\epsilon$. These are represented by types  \Rmnum{5}.A and \Rmnum{5}.B in Table \ref{main theorem table}.
\item There exists one equivalence class of ARMS satisfying the subalgebra property (Definition \ref{subalgebra property}) corresponding to each of the four cases \eqref{ch3 7d case H and A formula 3}, \eqref{ch3 7d case H and A formula 4},  \eqref{ch3 7d case H and A formula 5}, and \eqref{ch3 7d case H and A formula 8}. These are represented by types \Rmnum{1}, \Rmnum{2}, \Rmnum{6}, and \Rmnum{7} in Table \ref{main theorem table}.
\item No ARMS satisfying the subalgebra property (Definition \ref{subalgebra property}) correspond to any of the cases in \eqref{ch3 7d case H and A formula 6} and \eqref{ch3 7d case H and A formula 7}.
\end{enumerate}
The symmetry groups of these nine flat structures have the respective dimensions indicated in Table \ref{main theorem table}.
\end{theorem}

\begin{corollary}\label{ch3 main 7d corollary}
The adapted (partial) frame bundle $P^0$ of a $7$-dimensional homogeneous $2$\--nondegener\-ate hypersurface-type CR manifold admits a reduction having the structure of a principal bundle  over the complexified Levi leaf space (defined locally in \cite{SYKES2023108850}) whose structure group has as its Lie algebra the degree zero component of one of the nine ARMS in Table \ref{main theorem table}. Thus every such CR manifold is canonically assigned one of the nine types labeled in Table \eqref{main theorem table}, and, moreover, this type is determined by any one of its modified CR symbols.
\end{corollary}
\begin{proof}
This is an immediate corollary of Lemma \ref{constant symbol reduction lemma} and Theorem \ref{ch3 main 7d theorem a}.
\end{proof}
\begin{lemma}\label{g2 structure theorem}
If a $7$-dimensional homogeneous $2$-nondegenerate hypersurface-type CR manifold $(M,H)$ whose $P^0$ bundle admits a reduction $P^{0,\mathrm{red}}$ with constant reduced modified symbol $\mathfrak{g}^{0,\mathrm{red}}$ of type \Rmnum{3} (as enumerated in Table \ref{main theorem table}) has a $9$-dimensional symmetry group then $(M,H)$ is locally equivalent to the flat structure generated by $\mathfrak{g}^{0,\mathrm{red}}$.
\end{lemma}
We defer the proof of Lemma \ref{g2 structure theorem} to Section \ref{main theorem proof subsection d}.
\begin{theorem}\label{ch3 main 7d theorem b}
The symmetry group $\mathrm{Aut}(M)$ of a $7$-dimensional homogeneous $2$-nondegenerate hyper\-surface-type CR manifold $M$ has dimension bounded by that of the flat structure having the same type as $M$ (as described in Corollary \ref{ch3 main 7d corollary}). This bound is given in Table \ref{main theorem table}, and if $\dim\mathrm{Aut}(M)$ attains its bound then the CR structure on $M$ is flat (as defined in Definition \ref{flat structure definition}).  In particular the flat structures are in one-to-one correspondence with modified symbol models (as defined in Section \ref{introduction}). 
\end{theorem}
\begin{proof}
This theorem is a corollary of Lemma \ref{g2 structure theorem}, Theorem \ref{maximally symmetric recoverable models}, and the main results of \cite{porter2021absolute}. Indeed, for types \Rmnum{1}, \Rmnum{2}, \Rmnum{5}, and \Rmnum{6} it follows immediately from Theorem \ref{maximally symmetric recoverable models}, whereas for types \Rmnum{4}, \Rmnum{5}, \Rmnum{6}, and \Rmnum{7} it is a special case of the results in \cite{porter2021absolute}. See Section \ref{main theorem proof subsection c} for more detail on how these previous theorems are applied.
 
The remaining case, type \Rmnum{3}, is addressed by Lemma \ref{g2 structure theorem}.
\end{proof}

We prove Theorem \ref{ch3 main 7d theorem a} and Lemma \ref{g2 structure theorem} and expand on the proof of Theorem \ref{ch3 main 7d theorem b} in Sections \ref{main theorem proof subsection a} through \ref{main theorem proof subsection d}, presenting the proofs with the following outline, partitioned into two steps.

\textbf{Step 1 (}sections \ref{main theorem proof subsection a}\textbf{-}\ref{main theorem proof subsection b}\textbf{):} We give a constructive proof of the existence statements in Theorem \ref{ch3 main 7d theorem a},  explicitly describing the equivalence classes of reduced modified CR  symbols referred to in Theorem \ref{ch3 main 7d theorem a}. Since our classification goal reduces to describing the tuples $(H_\ell, A,\Omega ,\mathscr{A}_0)$ for which the system \eqref{system} is consistent, we will suppose that $\Omega$  and $\mathscr{A}_0$ are fixed such that \eqref{system} is satisfied. We also let $\mathfrak{g}^0$ be the CR symbol encoded by the pair $(H_\ell,A)$, as described in Remark \ref{symbol_as_tuple}. Depending on the value of  $(H_\ell, A)$ we will either describe this pair $(\Omega ,\mathscr{A}_0)$ in more detail, deriving the corresponding formulas in Table \ref{main theorem table}, or derive a contradiction from the assumption that such a pair exists. Doing this for all $(H_\ell, A)$ in the normal forms of \eqref{ch3 7d case H and A formula 1} through \eqref{ch3 7d case H and A formula 8} completes the proof of the existence statements in Theorem \ref{ch3 main 7d theorem a}.

\textbf{Step 2 (}sections \ref{main theorem proof subsection c}\textbf{-}\ref{main theorem proof subsection d}\textbf{):} Establishing the symmetry group bounds of Theorems \ref{ch3 main 7d theorem a} and  \ref{ch3 main 7d theorem b} and the local uniqueness statements of Theorem \ref{ch3 main 7d theorem b} requires different arguments for the different types in Table \ref{main theorem table}, which we present in Sections \ref{main theorem proof subsection c} and \ref{main theorem proof subsection d}.  As noted in the proof of Theorem \ref{ch3 main 7d theorem b}, types  \Rmnum{1}, \Rmnum{2}, \Rmnum{5}.A,  \Rmnum{5}.B, and \Rmnum{6} can be treated as immediate applications  of \cite[Theorem 6.2]{SYKES2023108850}, which we explain further in Section \ref{main theorem proof subsection c}, whereas Theorem \ref{ch3 main 7d theorem b} is already proven in \cite{porter2021absolute} for types  \Rmnum{4}.A, \Rmnum{4}.B, \Rmnum{5}.A, \Rmnum{5}.B, \Rmnum{6}, and \Rmnum{7}.

\subsection{{Symbols corresponding to formulas (\ref{ch3 7d case H and A formula 1}) through (\ref{ch3 7d case H and A formula 4})}}\label{main theorem proof subsection a}

Suppose that $H_\ell$ and $A$ are as in \eqref{ch3 7d case H and A formula 1}, \eqref{ch3 7d case H and A formula 2}, \eqref{ch3 7d case H and A formula 3}, or \eqref{ch3 7d case H and A formula 4}.  Since, by Lemma \ref{regular CR symbol},   $\mathfrak{g}^0$ is regular, the algebra $\mathfrak{g}^0$ is itself a reduced modified CR symbol corresponding to the pair $(H_\ell, A)$, as noted in Remark \ref{regular flat structures}. This reduced modified symbol is described by taking $\mathscr{A}_0=\mathscr{A}$ and taking $\Omega $ to be any matrix in $\mathscr{A}$. Hence, all that remains for us to do is determine whether or not there exist  reduced modified symbols for which $\Omega $ is not in $\mathscr{A}$.

If $(H_\ell, A)$ is as in  \eqref{ch3 7d case H and A formula 1}, then there turns out to be exactly one equivalence class of solutions with $\Omega $ not in  $\mathscr{A}$ provided that $\epsilon=-1$, whereas there is no such solution if $\epsilon=1$. So we record this as a lemma.

\begin{lemma}\label{ch3 7d submaximal regular lemma}
Suppose $(H_\ell, A)$ is as in \eqref{ch3 7d case H and A formula 1}. If $\epsilon=-1$ then (up to the equivalence in Definition \ref{equivalence of ARMS}) there exists exactly one ARMS $\mathfrak{g}^{0,\mathrm{red}}$ satisfying  the subalgebra property (Definition \ref{subalgebra property})  such that the matrix $\Omega $ is not in $\mathscr{A}$, and if $\epsilon=1$ then there is no such ARMS. In the former case, this equivalence class of reduced modified symbols is represented by any one of the ARMS described by \eqref{ch3 7d case H and A formula 1} and 
\begin{align}\label{ch3 7d case submaximal solution}
\Omega =
e^{i\theta}\left(
\begin{array}{cc}
0 & 0\\
\sqrt{\tfrac{3}{4}} &0
\end{array}
\right)
\, \mbox{ and }\, 
\mathscr{A}_0=
\mathrm{span}
\left\{
\left(
\begin{array}{cc}
1& 0\\
0 & 3
\end{array}
\right)
\right\}
\, \mbox{for some }\theta\in\mathbb{R}.
\end{align}
\end{lemma}

\vspace{5pt}
\setlength\tabcolsep{2pt}

\noindent
\scalebox{.87}{
\begin{threeparttable}
\begin{tabular}{|c|c|c|c|c|c|c|}
\hline
label & $H_\ell$&$A$&$\Omega$&\parbox{2cm}{\centering matrices spanning $\mathscr{A}$} & \To{5}\Bo{4}\parbox{1.6cm}{\centering symmetry group dimension} & \To{5}\Bo{4}\parbox{3.4cm}{\centering coordinate descriptions} \\\hline\hline
\To{5}\Bo{3.6} Type \Rmnum{1} & 
\bgroup
\def\arraystretch{1}
$ \left(
\begin{array}{cc}
0&1\\
1&0
\end{array}
\right)$
\egroup
&
\bgroup
\def\arraystretch{1}
$ \left(
\begin{array}{cc}
0&i\\
1&0
\end{array}
\right)$
\egroup
&
\bgroup
\def\arraystretch{1}
$\left(
\begin{array}{cc}
0&\frac{i}{\sqrt{2}}\\
\frac{1}{\sqrt{2}}&0
\end{array}
\right)$
\egroup
&
$0$
&
$8$
&
\To{3}\Bo{3}\parbox{3.4cm}{\flushleft defining equation to appear in \cite{gregorovic2023defining}}
\\\hline
\To{5}\Bo{3.6} Type \Rmnum{2} &
\bgroup
\def\arraystretch{1}
$ \left(
\begin{array}{cc}
0&1\\
1&0
\end{array}
\right)$
\egroup
&
\bgroup
\def\arraystretch{1}
$  
\left(
\begin{array}{cc}
1&1\\
0&1
\end{array}
\right)
$
\egroup
&
\bgroup
\def\arraystretch{1}
$
\left(
\begin{array}{cc}
1&1/2\\
0&0
\end{array}
\right)
$
\egroup
&
$0$
&
$8$
&
\To{1}\Bo{1}\parbox{3.4cm}{\flushleft \cite[Theorem 2, eqn. 1]{mozey2000}}
\\\hline
\To{5}\Bo{3.6} Type \Rmnum{3} &
\bgroup
\def\arraystretch{1}
$ \left(
\begin{array}{cc}
1&0\\
0&-1
\end{array}
\right)$
\egroup
&
\bgroup
\def\arraystretch{1}
$ \left(
\begin{array}{cc}
1&0\\
0&0
\end{array}
\right)$
\egroup
&
\bgroup
\def\arraystretch{1}
$\left(
\begin{array}{cc}
0&0\\
\frac{\sqrt{3}}{2}&0
\end{array}
\right)$
\egroup
&
\bgroup
\def\arraystretch{1}
$
\left(
\begin{array}{cc}
a&0\\
0&3a
\end{array}
\right)
$
\egroup
&
$9$
&
\To{3}\Bo{3}\parbox{3.4cm}{\flushleft defining equation to appear in \cite{gregorovic2023defining}}
\\\hline
\To{3}Type \Rmnum{4}.A ($\epsilon=1$) &
\multirow{2}*{
\bgroup
\def\arraystretch{1}
$ \left(
\begin{array}{cc}
1&0\\
0&\epsilon
\end{array}
\right)$
\egroup
}
&
\multirow{2}*{
\bgroup
\def\arraystretch{1}
$  \left(
\begin{array}{cc}
1&0\\
0&0
\end{array}
\right)$
\egroup
}
&
\multirow{2}*{
$0$
}
&
\multirow{2}*{
\bgroup
\def\arraystretch{1}
$
\left(
\begin{array}{cc}
a&0\\
0&b
\end{array}
\right)
$
\egroup
}
& \multirow{2}*{
10
}
& \multirow{2}*{
\To{1}\Bo{1}\parbox{3cm}{\flushleft \cite[Theorem 2, eqn. 8]{mozey2000}}
}
\\\cdashline{1-1}
\Bo{2}Type \Rmnum{4}.B ($\epsilon=-1$) &&&&&&
\\\hline
\To{3}Type \Rmnum{5}.A ($\epsilon=1$) &
\multirow{2}*{
\bgroup
\def\arraystretch{1}
$ \left(
\begin{array}{cc}
1&0\\
0&\epsilon
\end{array}
\right)$
\egroup
}
&
\multirow{2}*{
\bgroup
\def\arraystretch{1}
$  \left(
\begin{array}{cc}
1&0\\
0&1
\end{array}
\right)$
\egroup
}
&
\multirow{2}*{
$0$
}
&
\multirow{2}*{
\bgroup
\def\arraystretch{1}
$
\left(
\begin{array}{cc}
a&b\\
-\epsilon b&a
\end{array}
\right)
$
\egroup
}
& \multirow{2}*{
15
}
&
 \multirow{2}*{
\parbox{3.2cm}{\flushleft tubes over null cones of symmetric forms in $\mathbb{R}^4$}
}
\\\cdashline{1-1}
\Bo{2.1}Type \Rmnum{5}.B ($\epsilon=-1$) &&&&&&
\\\hline
\To{5}\Bo{3.6} Type \Rmnum{6} &\bgroup
\def\arraystretch{1}
$ \left(
\begin{array}{cc}
0&1\\
1&0
\end{array}
\right)$
\egroup
&
\bgroup
\def\arraystretch{1}
$ \left(
\begin{array}{cc}
0&-1\\
1&0
\end{array}
\right)$
\egroup
&
$0$
&
$
\left(
\begin{array}{cc}
a&b\\
 b&a
\end{array}
\right)
$
&
$15$
&
\parbox{3.2cm}{\flushleft 
\cite[Section 5.4]{gregorovic2021equivalence}
}
\\\hline
\To{5}\Bo{3.6} Type \Rmnum{7} &\bgroup
\def\arraystretch{1}
$ \left(
\begin{array}{cc}
0&1\\
1&0
\end{array}
\right)$
\egroup
&
\bgroup
\def\arraystretch{1}
$ \left(
\begin{array}{cc}
0&1\\
0&0
\end{array}
\right)$
\egroup
&
$0$
&
$
\left(
\begin{array}{cc}
a&b\\
 0&c
\end{array}
\right)
$
&
$16$
&
\To{4.2}\Bo{3.5}\parbox{3.4cm}{\flushleft \cite[Theorem 2, eqn. 8]{mozey2000},  \cite[Section 5, e.g. 1 with $n=3$]{labovskii1997dimensions}}
\\\hline
\end{tabular}
\caption{Flat structures of Theorem \ref{ch3 main 7d theorem a} described in the notation of Lemma \ref{matrix representation lemma}. The letters $a$, $b$, and $c$ denote complex variables. The last column gives references to known coordinate descriptions of the respective flat structures. Some of these references describe real hypersurfaces in $\mathbb{R}^4$, and it is rather the tube over this real hypersurface that has the relevant CR structure.} \label{main theorem table}
\end{threeparttable}
}

\begin{proof}
Notice that $\mathscr{A}$ is the space of all $2\times 2$ diagonal matrices, and item (iii) in \eqref {system} implies $\Omega_{1,2}=0$. With $\Omega_{1,2}=0$, we get
\begin{align}\label{ch3 7d case regular Omega Gramian commutator}
\left[\overline{H_\ell^{-1}\Omega^T H_\ell},\Omega \right]=
\left(
\begin{array}{cc}
\epsilon|\Omega_{2,1}|^2& \epsilon\overline{\Omega_{2,1}}(\Omega_{2,2}-\Omega_{1,1})\\ 
 \Omega_{2,1}(\overline{\Omega_{2,2}-\Omega_{1,1}})&-\epsilon|\Omega_{2,1}|^2
\end{array}
\right).
\end{align}
The coefficient $\mu_{1,1}^1$ in item (iii) of \eqref {system} is equal to $2\Omega_{1,1}$, and hence, by \eqref{ch3 7d case regular Omega Gramian commutator}, labeling the matrix $\left[\overline{H_\ell^{-1}\Omega ^T H_\ell},\Omega  \right]+A\overline{A}-\left(\overline{\mu_{1,1}^1 }\Omega  +\mu_{1,1}^1  \overline{H_\ell^{-1}\Omega ^T H_\ell}\right) $ in item  (iv) of \eqref {system}  $\alpha$, we have 
\begin{align}\label{ch3 7d case regular item 4 matrix}
\alpha =
\left(
\begin{array}{cc}
\epsilon|\Omega_{2,1}|^2-4|\Omega_{1,1}|^2+1& \epsilon\overline{\Omega_{2,1}}(\Omega_{2,2}-3\Omega_{1,1})\\
 \Omega_{2,1}(\overline{\Omega_{2,2}-3\Omega_{1,1}})&a
\end{array}
\right),
\end{align}
where
\[
a=-\epsilon|\Omega_{2,1}|^2-2\left(\overline{\Omega_{1,1}}\Omega_{2,2}+\Omega_{1,1}\overline{\Omega_{2,2}}\right).
\]
By  item  (iv) of \eqref {system}, $\alpha$ belongs to  $\mathscr{A}$, and is therefore diagonal. Since we are searching for a solution with $\Omega $ not in $\mathscr{A}$, we can assume that $\Omega_{2,1}\neq 0$, and hence setting the off-diagonal entries in \eqref{ch3 7d case regular item 4 matrix} equal to zero yields
\begin{align}\label{ch3 7d case regular Omega simplification}
\Omega_{2,2}=3\Omega_{1,1}. 
\end{align}
Accordingly
\begin{align}\label{ch3 7d case regular alpha simplified}
\alpha
=
\left(
\begin{array}{cc}
\epsilon|\Omega_{2,1}|^2-4|\Omega_{1,1}|^2+1& 0\\
 0&-\epsilon|\Omega_{2,1}|^2-12|\Omega_{1,1}|^2
\end{array}
\right).
\end{align}
Evaluating item (i) in \eqref {system} with $\alpha$ given by \eqref{ch3 7d case regular alpha simplified}, we obtain $\eta_{\alpha,1}^1=2\epsilon|\Omega_{2,1}|^2-8|\Omega_{1,1}|^2+2$, and since, noting $\Omega_{1,2}=0$, 
\[
[\alpha,\Omega ]=
\left(
\begin{array}{cc}
0& 0\\
\left(-2\epsilon|\Omega_{2,1}|^2-8|\Omega_{1,1}|^2-1\right) \Omega_{2,1}&0
\end{array}
\right), 
\]
the $(2,1)$ entry of $[\alpha,\Omega ]-\eta_{\alpha,1}^1 \Omega $ is equal to
\begin{align}\label{ch3 7d case regular item 2}
\left([\alpha,\Omega ]-\eta_{\alpha,1}^1 \Omega \right)_{2,1}=-\left(4\epsilon|\Omega_{2,1}|^2+3\right) \Omega_{2,1}.
\end{align}
By item (ii) in \eqref {system}, $[\alpha,\Omega ]-\eta_{\alpha,1}^1 \Omega $ belongs to $\mathscr{A}$, and hence $\left([\alpha,\Omega ]-\eta_{\alpha,1}^1 \Omega \right)_{2,1}=0$. If $\epsilon=1$ then we have obtained a contradiction because then the value in \eqref{ch3 7d case regular item 2} is nonzero. Accordingly, if $\Omega \not\in\mathscr{A}$ then $\epsilon=-1$.  Setting \eqref{ch3 7d case regular item 2} equal to zero with $\epsilon=-1$, we get
\begin{align}\label{ch3 7d case regular simplified omega final}
|\Omega_{2,1}|^2=\frac{3}{4}.
\end{align}
By \eqref{ch3 7d case regular alpha simplified} and \eqref{ch3 7d case regular simplified omega final}
\[
\alpha
=
\left(
\begin{array}{cc}
-4|\Omega_{1,1}|^2-\tfrac{1}{4}& 0\\
 0&3\left(-4|\Omega_{1,1}|^2-\tfrac{1}{4}\right)
\end{array}
\right),
\]
and hence
\begin{align}\label{ch3 7d case regular intersection}
\mathrm{span}\left\{
\frac{\alpha}{-4|\Omega_{1,1}|^2-\tfrac{1}{4}}
\right\}
=
\mathrm{span}\left\{
\left(
\begin{array}{cc}
1& 0\\
 0&3
\end{array}
\right)
\right\}
\subset
\mathscr{A}_0.
\end{align}
Notice that $I$ is not in $\mathscr{A}_0$ because then items (i) and (ii) of \eqref {system} would imply that $\Omega $ is in $\mathscr{A}_0$, so equality actually holds in \eqref{ch3 7d case regular intersection}, which together with \eqref{ch3 7d case regular simplified omega final} implies that \eqref{ch3 7d case H and A formula 1} and \eqref{ch3 7d case submaximal solution} indeed give a solution to the system \eqref {system}.

Lastly, we need to show that changing the parameter $\theta$ in \eqref{ch3 7d case submaximal solution} does not change the equivalence class represented by the corresponding reduced modified CR symbol. To see this last observation, consider the $1$-parameter subgroup
\begin{align}\label{ch3 7d scaling group}
\left\{
\left.
\left(
\begin{array}{cc}
e^{it}I &0\\
0&e^{-it}I
\end{array}
\right)
\,\right|\,
t\in\mathbb{R}
\right\}
\end{align}
of $CSp(\mathfrak{g}_{-1})$. This subgroup belongs to the group $\Re G_{0,0}=G_{0,0}\cap \Re G$ (where $G$ and $G_{0,0}$ are as in Section \ref{preliminaries}) and it acts transitively (via the natural adjoint action) on the set of reduced modified symbols parameterized by $\theta$ described by \eqref{ch3 7d case H and A formula 1} and \eqref{ch3 7d case submaximal solution}, giving isomorphisms between these ARMS establishing their equivalence in the sense of Definition \ref{equivalence of ARMS}.
\end{proof}

\begin{lemma}
If $(H_\ell, A)$ is as in  \eqref{ch3 7d case H and A formula 2}, \eqref{ch3 7d case H and A formula 3}  or  \eqref{ch3 7d case H and A formula 4} then (up to the equivalence in Definition \ref{equivalence of ARMS}) the only corresponding ARMS satisfying the subalgebra property (Definition \ref{subalgebra property})  is the one described by taking $\mathscr{A}_0=\mathscr{A}$ and taking $\Omega $ to be any matrix in $\mathscr{A}$.
\end{lemma}
\begin{proof}
If $(H_\ell, A)$ is as in  \eqref{ch3 7d case H and A formula 2}, it is easily checked that  item (iii) in \eqref {system} implies that both of the set inclusion conditions in \eqref{intersection algebra} are satisfied by setting $\alpha=\Omega $. In other words, if $(H_\ell, A)$ is as in  \eqref{ch3 7d case H and A formula 2} then $\Omega $ is in $\mathscr{A}$.

This also happens if $(H_\ell, A)$ is as in  \eqref{ch3 7d case H and A formula 4} instead by exactly the same calculation, which is clear because if $(H_\ell, A)$ is as in  \eqref{ch3 7d case H and A formula 4}  then $AH_\ell^{-1}$ and $H_\ell \overline{A}$ are the same in this case as they are in the case where \eqref{ch3 7d case H and A formula 2} holds.

Similarly, if $(H_\ell, A)$ is as in  \eqref{ch3 7d case H and A formula 3} then $\mathscr{A}$ is the space of $2\times 2$ upper-triangular matrices, and item (iii) in \eqref {system} implies $\Omega_{2,1}=0$. In other words, if $(H_\ell, A)$ is as in  \eqref{ch3 7d case H and A formula 3} then again we get that $\Omega $ is in $\mathscr{A}$.
\end{proof}

\subsection{Symbols corresponding to formulas (\ref{ch3 7d case H and A formula 5}) through (\ref{ch3 7d case H and A formula 8})}\label{main theorem proof subsection b}

In each of these cases (i.e., in \eqref{ch3 7d case H and A formula 5} through \eqref{ch3 7d case H and A formula 8}), $\mathscr{A}$ is spanned by the identity matrix $I$. If $I$ is in $\mathscr{A}_0$ then items (i) and (ii) of \eqref {system}  imply that $\Omega $ is in $\mathscr{A}_0$, which contradicts Lemma \ref{regular CR symbol}, so 
\begin{align}\label{ch3 7d case non-reggular intersection}
\mathscr{A}_0=0.
\end{align}

Proceeding, suppose first that $H_\ell$ and $A$ are as in \eqref{ch3 7d case H and A formula 5} and \eqref{ch3 7d case H and A formula 6}. To treat both cases with common formulas, for the case where $H_\ell$ and $A$ are as in \eqref{ch3 7d case H and A formula 5}, we set $\theta=\tfrac{\pi}{2}$ so that $A$ is described by the same formula as in \eqref{ch3 7d case H and A formula 6}. Item (iii) in \eqref {system} implies
\begin{align}\label{ch3 7d case non-reggular omega simplified a}
\Omega_{1,1}=\Omega_{2,2},
\quad 
\Omega_{1,2}=-e^{-i\theta}\Omega_{2,1},
\quad\mbox{ and }\quad
\mu_{1,1}^1 =2\Omega_{1,1}.
\end{align}

Labeling the matrix $\left[\overline{H_\ell^{-1}\Omega ^T H_\ell},\Omega  \right]+A\overline{A}-\left(\overline{\mu_{1,1}^1 }\Omega  +\mu_{1,1}^1  \overline{H_\ell^{-1}\Omega ^T H_\ell}\right) $ in item  (iv) of \eqref {system}  $\alpha$ and applying \eqref{ch3 7d case non-reggular omega simplified a} to simplify $\alpha$, we obtain
\begin{align}\label{ch3 7d case non-reggular item 4 aa}
\alpha_{1,2}=4\Re\left(e^{i\theta}\Omega_{1,1}\overline{\Omega_{2,1}}\right),
\quad\quad
\alpha_{2,1}=-4\Re\left(\Omega_{1,1}\overline{\Omega_{2,1}}\right),
\end{align}
and
\begin{align}\label{ch3 7d case non-reggular item 4 a}
\alpha_{1,1}=e^{i\theta}-4|\Omega_{1,1}|^2+(e^{-i\theta}-e^{i\theta})|\Omega_{2,1}|^2.
\end{align}

Since item (iv) of \eqref {system} gives that $\alpha$ belongs to $\mathscr{A}_0$, by \eqref{ch3 7d case non-reggular intersection}, the values in \eqref{ch3 7d case non-reggular item 4 aa} and \eqref{ch3 7d case non-reggular item 4 a} are  equal to zero. Since $0<\theta<\pi$, setting the values in \eqref{ch3 7d case non-reggular item 4 aa} equal to zero implies $\Omega_{1,1}\overline{\Omega_{2,1}}$, whereas setting   \eqref{ch3 7d case non-reggular item 4 a} equal to zero implies $\Omega_{2,1}\neq 0$. Therefore, $\Omega_{1,1}=0$, and, by \eqref{ch3 7d case non-reggular item 4 a}, the equation $\alpha_{1,1}=0$ simplifies to 
\begin{align}\label{ch3 7d case non-reggular omega simplified b}
\Omega_{1,1}=0
\quad\mbox{ and }\quad
|\Omega_{2,1}|^2=\frac{e^{i\theta}}{e^{i\theta}-e^{-i\theta}}.
\end{align}
Since $0<\theta<\pi$ and $0\leq |\Omega_{2,1}|$, \eqref{ch3 7d case non-reggular omega simplified b} implies that $\theta=\tfrac{\pi}{2}$, and hence the system \eqref {system} is inconsistent if $(H_\ell, A)$ is as in \eqref{ch3 7d case H and A formula 6}, which yields the following result.

\begin{lemma}\label{ch3 7d nonregular nondiagonal lemma}
There are no ARMS satisfying the subalgebra property (Definition \ref{subalgebra property})  corresponding to any of the cases in \eqref{ch3 7d case H and A formula 6}.
\end{lemma}

\begin{lemma}\label{ch3 7d maximal non-reggular lemma}
There exists exactly one equivalence class of ARMS $\mathfrak{g}^{0,\mathrm{red}}$ satisfying the subalgebra property  (in the sense of Definition \ref{equivalence of ARMS}) corresponding to the case where $(H_\ell, A)$ is as in \eqref{ch3 7d case H and A formula 5}. This equivalence class of ARMS is represented by any one of the symbols described by \eqref{ch3 7d case H and A formula 5} and 
\begin{align}\label{ch3 7d case maximal non-reggular solution}
\Omega =
e^{i\theta}\left(
\begin{array}{cc}
0 & i\sqrt{\tfrac{1}{2}}\\
\sqrt{\tfrac{1}{2}} &0
\end{array}
\right)
\quad\mbox{ and }\quad
\mathscr{A}_0=
0
\quad\mbox{for some }\theta\in\mathbb{R}.
\end{align}
\end{lemma}
\begin{proof}

By \eqref{ch3 7d case non-reggular intersection}, \eqref{ch3 7d case non-reggular omega simplified a}, and \eqref{ch3 7d case non-reggular omega simplified b}, if $(H_\ell, A, \Omega , \mathscr{A}_0)$ satisfies \eqref {system} with $(H_\ell, A)$ as in \eqref{ch3 7d case H and A formula 5} then, indeed \eqref{ch3 7d case maximal non-reggular solution} holds. Conversely, if $(H_\ell, A, \Omega , \mathscr{A}_0)$ is as in \eqref{ch3 7d case H and A formula 5} and \eqref{ch3 7d case maximal non-reggular solution} then it is straightforward to check that the system \eqref {system} is consistent.

We finish this proof using the same conclusion as in the proof of Lemma \ref{ch3 7d submaximal regular lemma}. That is,  the $1$-parameter subgroup of $CSp(\mathfrak{g}_{-1})$ given in \eqref{ch3 7d scaling group} acts transitively on the set of reduced modified symbols parameterized by $\theta$ described by \eqref{ch3 7d case H and A formula 5} and \eqref{ch3 7d case maximal non-reggular solution}, providing the Lie algebra isomorphisms that show as $\theta$ varies in \eqref{ch3 7d case maximal non-reggular solution} the corresponding reduced modified CR symbols belong to the same equivalence class (in the sense of Definition \ref{equivalence of ARMS}).
\end{proof}

The following lemmas address the cases in \eqref{ch3 7d case H and A formula 7} and \eqref{ch3 7d case H and A formula 8}.

\begin{lemma}\label{ch3 7d nonregular diagonal lemma}
There are no ARMS satisfying the subalgebra property (Definition \ref{subalgebra property}) corresponding to either of the cases (i.e., $\epsilon=1$ and $\epsilon=-1$) in  \eqref{ch3 7d case H and A formula 7}.
\end{lemma}

\begin{proof}
Item (iii) in \eqref {system} implies
\begin{align}\label{ch3 7d case non-reggular 2 omega simplified a}
\Omega_{1,1}=\Omega_{2,2},
\quad 
\Omega_{1,2}=-\epsilon\lambda \Omega_{2,1},
\quad\mbox{ and }\quad
\mu_{1,1}^1 =2\Omega_{1,1}.
\end{align}
Labeling the matrix $\left[\overline{H_\ell^{-1}\Omega ^T H_\ell},\Omega  \right]+A\overline{A}-\left(\overline{\mu_{1,1}^1 }\Omega  +\mu_{1,1}^1  \overline{H_\ell^{-1}\Omega ^T H_\ell}\right) $ in item  (iv) of \eqref {system}  $\alpha$ and applying \eqref{ch3 7d case non-reggular 2 omega simplified a} to simplify $\alpha$, we obtain
\begin{align}\label{ch3 7d case non-reggular 2 item 4 a}
\alpha_{1,2}=2\epsilon(\lambda\overline{\Omega_{1,1}}\Omega_{2,1}-\Omega_{1,1}\overline{\Omega_{2,1}})
\quad\mbox{ and }\quad
\alpha_{2,1}=-2(\overline{\Omega_{1,1}}\Omega_{2,1}-\lambda\Omega_{1,1}\overline{\Omega_{2,1}})
\end{align}
and
\begin{align}\label{ch3 7d case non-reggular 2 item 4 ab}
\alpha_{1,1}=1-4|\Omega_{1,1}|^2+\epsilon(1-\lambda^2)|\Omega_{2,1}|^2
\,\mbox{ and }\,  
\alpha_{2,2}=\lambda^2-4|\Omega_{1,1}|^2-\epsilon(1-\lambda^2)|\Omega_{2,1}|^2.
\end{align}
Since item (iv) of \eqref {system} gives that $\alpha$ belongs to $\mathscr{A}_0$, by \eqref{ch3 7d case non-reggular intersection}, the values in \eqref{ch3 7d case non-reggular 2 item 4 a} and \eqref{ch3 7d case non-reggular 2 item 4 ab} are equal to zero. Accordingly,
\[
\epsilon(\lambda^2-1)\Omega_{1,1}\overline{\Omega_{2,1}}=\frac{\alpha_{1,2}+\epsilon\lambda \alpha_{2,1}}{2}=0,
\]
which implies that either $\Omega_{1,1}=0$ or $\Omega_{2,1}=0$ because $\lambda^2\neq 1$. If $\Omega_{2,1}=0$ then $\Omega$ is a multiple of the identity, which implies that $\Omega\in\mathscr{A}$, contradicting Lemma \ref{regular CR symbol}. Therefore, $\Omega_{1,1}=0$. Yet if $\Omega_{1,1}=0$, since $\alpha_{1,1}=\alpha_{2,2}=0$, the two equations in  \eqref{ch3 7d case non-reggular 2 item 4 ab}  respectively imply
\[
1=-\epsilon(1-\lambda^2)|\Omega_{2,1}|^2
\quad\mbox{ and }\quad
\lambda^2=\epsilon(1-\lambda^2)|\Omega_{2,1}|^2,
\]
implying $\lambda^2=-1$, contradicting the assumption in \eqref{ch3 7d case H and A formula 7} that $\lambda>1$.
\end{proof}

\begin{lemma}
There exists exactly one equivalence class of ARMS  $\mathfrak{g}^{0,\mathrm{red}}$ satisfying the subalgebra property  (in the sense of Definition \ref{equivalence of ARMS}) corresponding to the case where $(H_\ell, A)$ is as in \eqref{ch3 7d case H and A formula 8}. This equivalence class of reduced modified symbols is represented by any one of the symbols described by \eqref{ch3 7d case H and A formula 8} and 
\begin{align}\label{ch3 7d case maximal non-reggular 3 solution}
\Omega =
e^{i\theta}\left(
\begin{array}{cc}
1 &\tfrac{1}{2}\\
0&0
\end{array}
\right)
\quad\mbox{ and }\quad
\mathscr{A}_0=
0
\quad\mbox{for some }\theta\in\mathbb{R}.
\end{align}

\end{lemma}

\begin{proof}
Item (iii) in \eqref {system} implies
\begin{align}\label{ch3 7d case non-reggular 3 omega simplified a}
\Omega_{1,1}=2\Omega_{1,2}+\Omega_{2,2},
\quad 
\Omega_{2,1}=0,
\quad\mbox{ and }\quad
\mu_{1,1}^1 =2\left(\Omega_{1,2}+\Omega_{2,2}\right).
\end{align}
Labeling the matrix $\left[\overline{H_\ell^{-1}\Omega ^T H_\ell},\Omega  \right]+A\overline{A}-\left(\overline{\mu_{1,1}^1 }\Omega  +\mu_{1,1}^1  \overline{H_\ell^{-1}\Omega ^T H_\ell}\right) $ in item  (iv) of \eqref {system}  $\alpha$ and applying \eqref{ch3 7d case non-reggular 3 omega simplified a} to simplify $\alpha$, we obtain
\begin{align}\label{ch3 7d case non-reggular 3 item 4 a}
\alpha_{1,2}=-2\left(\Omega_{1,2}\left(\overline{\Omega_{1,2}+\Omega_{2,2}}\right)+\overline{\Omega_{1,2}}\left(3\Omega_{1,2}+\Omega_{2,2}\right)-1\right)
\end{align}
and 
\begin{align}\label{ch3 7d case non-reggular 3 item 4 b}
\alpha_{1,1}&=1-2\overline{\Omega_{2,2}}\left(\Omega_{1,2}+\Omega_{2,2}\right)-2\left( \overline{\Omega_{1,2}+\Omega_{2,2}}\right)\left(2\Omega_{1,2}+\Omega_{2,2}\right)\\
&=\left(1-4\left|\Omega_{1,2}\right|^2\right)-2\left(\overline{\Omega_{2,2}}\Omega_{1,2}+\Omega_{2,2}\overline{\Omega_{1,2}}\right)-4\left|\Omega_{2,2}\right|^2-4\overline{\Omega_{2,2}}\Omega_{1,2}.
\end{align}
Since item (iv) of \eqref{system} gives that $\alpha$ belongs to $\mathscr{A}_0$, by \eqref{ch3 7d case non-reggular intersection}, $\alpha=0$. Setting the value in \eqref{ch3 7d case non-reggular 3 item 4 a} equal to zero is equivalent to
\begin{align}\label{ch3 7d case non-reggular 3 item 4 c}
\Omega_{2,2}\overline{\Omega_{1,2}}+\overline{\Omega_{2,2}}\Omega_{1,2}=1-4\left|\Omega_{1,2}\right|^2.
\end{align}
Setting $\alpha_{1,1}=0$ and applying \eqref{ch3 7d case non-reggular 3 item 4 c} to simplify \eqref{ch3 7d case non-reggular 3 item 4 b}, we obtain
\begin{align}\label{ch3 7d case non-reggular 3 item 4 d}
0&=\left(1-4\left|\Omega_{1,2}\right|^2\right)-2\left(\overline{\Omega_{2,2}}\Omega_{1,2}+\Omega_{2,2}\overline{\Omega_{1,2}}\right)-2\left|\Omega_{2,2}\right|^2-4\overline{\Omega_{2,2}}\Omega_{1,2}\\
&=-\left(1-4\left|\Omega_{1,2}\right|^2\right)-2\left|\Omega_{2,2}\right|^2-4\overline{\Omega_{2,2}}\Omega_{1,2}.
\end{align}
Therefore, $\overline{\Omega_{2,2}}\Omega_{1,2}$ is a real number and \eqref{ch3 7d case non-reggular 3 item 4 c} implies
\begin{align}\label{ch3 7d case non-reggular 3 item 4 e}
\Omega_{2,2}=\frac{1-4\left|\Omega_{1,2}\right|^2}{2\overline{\Omega_{1,2}}}.
\end{align}
Together \eqref{ch3 7d case non-reggular 3 item 4 d} and  \eqref{ch3 7d case non-reggular 3 item 4 e} imply
\[
\frac{\left(1-4\left|\Omega_{1,2}\right|^2\right)^2}{2\left|\Omega_{1,2}\right|^2}=2\left|\Omega_{2,2}\right|^2=-\left(1-4\left|\Omega_{1,2}\right|^2\right)-4\overline{\Omega_{2,2}}\Omega_{1,2}=-3\left(1-4\left|\Omega_{1,2}\right|^2\right),
\]
which is equivalent to 
\begin{align}\label{ch3 7d case non-reggular 3 item 4 f}
0=\left(1-4\left|\Omega_{1,2}\right|^2\right)^2+6\left(1-4\left|\Omega_{1,2}\right|^2\right)\left|\Omega_{1,2}\right|^2=(1-4\left|\Omega_{1,2}\right|^2)(2\left|\Omega_{1,2}\right|^2+1).
\end{align}
By  \eqref{ch3 7d case non-reggular 3 item 4 e} and \eqref{ch3 7d case non-reggular 3 item 4 f},
\begin{align}\label{ch3 7d case non-reggular 3 item 4 g}
\left|\Omega_{1,2}\right|^2=\frac{1}{4}
\quad\mbox{ and }\quad
\Omega_{2,2}=0.
\end{align}

Therefore, noting \eqref{ch3 7d case non-reggular 3 item 4 a} and \eqref{ch3 7d case non-reggular 3 item 4 g}, if $(H_\ell, A, \Omega , \mathscr{A}_0)$ satisfies \eqref {system} with $(H_\ell, A)$ as in \eqref{ch3 7d case H and A formula 8} then \eqref{ch3 7d case maximal non-reggular 3 solution} holds. Conversely, if $(H_\ell, A, \Omega , \mathscr{A}_0)$ is as in \eqref{ch3 7d case H and A formula 8} and \eqref{ch3 7d case maximal non-reggular 3 solution} then it is straightforward to check that the system \eqref {system} is consistent.

We finish this proof using the same conclusion as in the proof of Lemma \ref{ch3 7d submaximal regular lemma}. That is,  the $1$-parameter subgroup of $CSp(\mathfrak{g}_{-1})$ given in \eqref{ch3 7d scaling group} acts transitively on the set of reduced modified symbols parameterized by $\theta$ described by \eqref{ch3 7d case H and A formula 8} and \eqref{ch3 7d case maximal non-reggular solution}, providing the Lie algebra isomorphisms that show as $\theta$ varies in  \eqref{ch3 7d case maximal non-reggular 3 solution} the corresponding reduced modified CR symbols belong to the same equivalence class (in the sense of Definition \ref{equivalence of ARMS}).
\end{proof}

\subsection{Local uniqueness and symmetry group dimensions}\label{main theorem proof subsection c}
For structures not of type \Rmnum{3}, we can establish the symmetry bounds and local uniqueness statements in Theorems \ref{ch3 main 7d theorem a} and \ref{ch3 main 7d theorem b} by applying Theorem \ref{maximally symmetric regular models} (a corollary of \cite[Theorem 3.1]{porter2021absolute}) and Theorem \ref{maximally symmetric recoverable models} (a corollary of \cite[Theorem 6.2]{SYKES2023108850}). For structures of type \Rmnum{3}, however, neither of these previous theorems apply. Indeed Theorem \ref{maximally symmetric regular models} does not apply because the local uniqueness results of \cite{porter2021absolute} apply only to structures that are maximally symmetric relative to their CR symbol, which type \Rmnum{3} structures are not. And Theorem \ref{maximally symmetric recoverable models} applies only to recoverable structures, which again type \Rmnum{3} structures are not.

In more detail, for an ARMS $\mathfrak{g}^{0,\mathrm{red}}$ of type  \Rmnum{1}, \Rmnum{2}, \Rmnum{5}, or \Rmnum{6}, its universal Tanaka prolongation has dimension equal to the symmetry group bound indicated in Table \ref{main theorem table}, a fact that is easily checked by direct calculation, as these prolongations can, for example, even be calculated using a computer algebra system such as Maple. By Lemma \ref{recoverability}, structures of each of these types are recoverable, and hence Theorem \ref{maximally symmetric recoverable models} (a corollary of \cite[Theorem 6.2]{SYKES2023108850}) indeed establishes the symmetry bounds and local uniqueness statements in Theorems \ref{ch3 main 7d theorem a} and \ref{ch3 main 7d theorem b} for structures of types  \Rmnum{1}, \Rmnum{2}, \Rmnum{5}, and \Rmnum{6}.

For types \Rmnum{4}, \Rmnum{5}, \Rmnum{6}, and \Rmnum{7} the symmetry bounds of  Theorem \ref{ch3 main 7d theorem a} where calculated in \cite{porter2021absolute}, using the bi-graded Tanaka prolongation method introduced therein. The local uniqueness statements in Theorem \ref{ch3 main 7d theorem b} for structures of types  \Rmnum{1}, \Rmnum{2}, \Rmnum{5}, and \Rmnum{6} follow from the main result in \cite{porter2021absolute} and a weaker version of this main result sufficient for our present application is stated above in Theorem \ref{maximally symmetric regular models}.

Symmetry bounds and local uniqueness of maximally symmetric type \Rmnum{3} structures are addressed by Lemma \ref{g2 structure theorem}, proven in the next section.

\subsection{Proof of Lemma \ref{g2 structure theorem}}\label{main theorem proof subsection d}
Throughout section \ref{main theorem proof subsection d}, let $(M,H)$ be as in Lemma \ref{g2 structure theorem}, so, in particular, $(M,H)$ is a homogeneous structure of type \Rmnum{3} (as described in Corollary \ref{ch3 main 7d corollary}).

Notice that  type \Rmnum{3} and type \Rmnum{4}.B structures have the same CR symbols. They are, however, distinguished by having non-equivalent modified CR symbols, and are therefore not locally equivalent as CR structures. It is shown in \cite{porter2021absolute} that the flat structure of type \Rmnum{4}.B is the unique structure with its CR symbol whose symmetry group is at least $10$-dimensional, and hence the symmetry group of a type \Rmnum{3} structure is at most $9$-dimensional. On the other hand the flat structure of type \Rmnum{3} is generated by a $9$-dimensional ARMS, so its symmetry group is at least $9$-dimensional. Therefore $9$ is the sharp upper bound for the symmetry group dimension of a type \Rmnum{3} structure.

Let $\mathfrak{g}^{0,\mathrm{red}}$ be the maximal ARMS of type \Rmnum{3} given in Table \ref{main theorem table}. We can assign it a basis $(e_0,\ldots,e_8)$ to this  $9$-dimensional Lie algebra with respect to which it has the  matrix representation $\rho:\mathfrak{g}^{0,\mathrm{red}}\to\mathfrak{gl}_6(\mathbb{C})$ given by
\begin{align}\label{7 dimensional regular submaximal example table}
\rho\left(\sum_{j=0}^8t_je_j\right)=
\left(\begin{array}{cccccc}
2t_7 & -t_3 & t_4 & t_1 & -t_2 & -2t_0 \\
0 & t_7+t_8 & \frac{\sqrt{3}}{2}t_6 & t_5 & 0 & -t_1 \\
0 &  \frac{\sqrt{3}}{2}t_5 & t_7+3t_8 & 0 & 0 & -t_2 \\
0 & t_6 & 0 & t_7-t_8 & \frac{\sqrt{3}}{2}t_5 & -t_3 \\
0 & 0 & 0 & \frac{\sqrt{3}}{2}t_6 & t_7-3t_8 & -t_4 \\
0 & 0 & 0 & 0 & 0 & 0 
\end{array}
\right).
\end{align}
To describe a decomposition of this ARMS as in \eqref{modified symbol decomposition} and \eqref{modified symbol decomposition a}, note that $(e_0, \ldots, e_5)$ forms a basis of the ARMS Heisenberg component as in \eqref{Heisenberg basis}, and  one can take $\mathfrak{g}_{0,+}^{\mathrm{red}}$ and $\mathfrak{g}_{0,-}^{\mathrm{red}}$ to be the subspaces spanned by $e_5$ and $e_6$ respectively. Consequently, $\mathfrak{g}_{0,0}^{\mathrm{red}}$ is spanned by $e_7$ and $e_8$.
Notice that $(2e_5,2e_6,e_8)$ is a standard $\mathfrak{sl}_2$ triple, and hence $\mathfrak{g}^{0,\mathrm{red}}$ is not solvable.

Let $\mathfrak{u}$ denote the standard universal Tanaka prolongation of $\mathfrak{g}^{0,\mathrm{red}}$, calculated with respect to the graded decomposition $\mathfrak{g}^{0,\mathrm{red}}=\mathfrak{g}_{-2}\oplus \mathfrak{g}_{-1}\oplus \mathfrak{g}_0^{\mathrm{red}}$ with $\mathfrak{g}_0^{\mathrm{red}}$ regarded as the degree zero component. Calculating $\mathfrak{u}$ explicitly, one finds from its Killing form and Cartan's criterion that $\mathfrak{u}$ is a semisimple, $14$-dimensional, rank $2$, complex Lie algebra, which from the well known classification of semisimple Lie algebras implies that $\mathfrak{u}$ is isomorphic to the Lie algebra of the exceptional complex Lie group $G_2$.

To relate the well know structure of $\mathrm{Lie}(G_2)$ to our considered subalgebra $\mathfrak{g}^{0,\mathrm{red}}$, consider the $\mathbb{Z}$-graded decomposition
\begin{align}\label{prol alg grading}
\mathfrak{u}=\mathfrak{u}_{-2}\oplus\mathfrak{u}_{-1}\oplus\mathfrak{u}_{0}\oplus\mathfrak{u}_{1}\oplus\mathfrak{u}_{2}
\end{align}
with $\mathfrak{u}_{-2}=\mathfrak{g}_{-2}$, $\mathfrak{u}_{-1}=\mathfrak{g}_{-1}$, and $\mathfrak{u}_{0}=\mathfrak{g}_0^{\mathrm{red}}$. In our chosen basis of $\mathfrak{g}_0^{\mathrm{red}}$, the Cartan subalgebra of $\mathfrak{u}$ is $\mathfrak{u}_{0,0}:=\langle e_7,e_8\rangle$. We describe the root space decomposition of $\mathfrak{u}$ by letting $\mathfrak{u}_{j,k}$ denote the intersection of the eigenspace of $\mathrm{ad}( -e_7)$ with eigenvalue $j$ and the eigenspace of $\mathrm{ad}( e_8)$ with eigenvalue $k$, and have the following root space decomposition diagram of $\mathfrak{u}$ labeled with these components. 
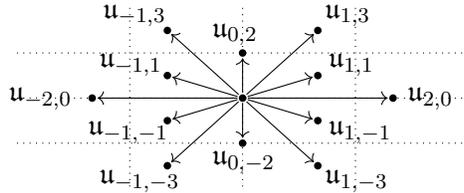
\begin{center}
\hspace{2.2cm}
\begin{tikzpicture}
\node[text width=3cm] at ([shift={(.4,0)}] -2,0) 
    {$\mathfrak{u}_{-2,0}$};
\node[text width=3cm] at ([shift={(.65,.2)}]-1,.9) 
    {$\mathfrak{u}_{-1,3}$};
\node[text width=3cm] at ([shift={(.6,-.2)}]-1,-.9) 
    {$\mathfrak{u}_{-1,-3}$};
\node[text width=3cm] at ([shift={(.6,.15)}]-1,.3) 
    {$\mathfrak{u}_{-1,1}$};
\node[text width=3cm] at ([shift={(.45,-.15)}]-1,-.3) 
    {$\mathfrak{u}_{-1,-1}$};
\node[text width=3cm] at ([shift={(1.1,.25)}] 0,.6) 
    {$\mathfrak{u}_{0,2}$};
\node[text width=3cm] at ([shift={(1.1,-.25)}]0,-.6) 
    {$\mathfrak{u}_{0,-2}$};
\node[text width=3cm] at ([shift={(1.6,.2)}]1,.9) 
    {$\mathfrak{u}_{1,3}$};
\node[text width=3cm] at ([shift={(1.6,-.2)}]1,-.9) 
    {$\mathfrak{u}_{1,-3}$};
\node[text width=3cm] at ([shift={(1.65,.15)}]1,.3) 
    {$\mathfrak{u}_{1,1}$};
\node[text width=3cm] at ([shift={(1.65,-.15)}]1,-.3) 
    {$\mathfrak{u}_{1,-1}$};
\node[text width=3cm] at ([shift={(1.7,0)}] 2,0) 
    {$\mathfrak{u}_{2,0}$};
\draw[dotted] (-3,0) -- (3,0);
\draw[dotted] (-3,.6) -- (3,.6);
\draw[dotted] (-3,-.6) -- (3,-.6);
\draw[dotted] (-1.5,1.2) -- (-1.5,-1.2);
\draw[dotted] (0,1.2) -- (0,-1.2);
\draw[dotted] (1.5,1.2) -- (1.5,-1.2);
\fill[black] (-2,0) circle (0.05 cm);
\fill[black] (-1,.9)  circle (0.05 cm);
\fill[black] (-1,.3) circle (0.05 cm);
\fill[black] (-1,-.3) circle (0.05 cm);
\fill[black] (-1,-.9) circle (0.05 cm);
\fill[black] (0,.6) circle (0.05 cm);
\fill[black] (0,0) circle (0.05 cm);
\fill[black] (0,-.6) circle (0.05 cm);
\fill[black] (1,.9) circle (0.05 cm);
\fill[black] (1,.3) circle (0.05 cm);
\fill[black] (1,-.3) circle (0.05 cm);
\fill[black] (1,-.9) circle (0.05 cm);
\fill[black] (2,0) circle (0.05 cm);
\draw[->,shorten >=2pt, shorten <=2pt](-0,0) -- (1,.9);
\draw[->,shorten >=2pt, shorten <=2pt](-0,0) -- (1,.3);
\draw[->,shorten >=2pt, shorten <=2pt](-0,0) -- (2,0);
\draw[->,shorten >=2pt, shorten <=2pt](-0,0) -- (1,-.3);
\draw[->,shorten >=2pt, shorten <=2pt](-0,0) -- (1,-.9);
\draw[->,shorten >=2pt, shorten <=2pt](-0,0) -- (0,.6);
\draw[->,shorten >=2pt, shorten <=2pt](-0,0) -- (0,-.6);
\draw[->,shorten >=2pt, shorten <=2pt](-0,0) -- (-1,.9);
\draw[->,shorten >=2pt, shorten <=2pt](-0,0) -- (-1,.3);
\draw[->,shorten >=2pt, shorten <=2pt](-0,0) -- (-2,0);
\draw[->,shorten >=2pt, shorten <=2pt](-0,0) -- (-1,-.3);
\draw[->,shorten >=2pt, shorten <=2pt](-0,0) -- (-1,-.9);
\end{tikzpicture}
\end{center}
Here $\mathfrak{u}_{-2,0}=\langle e_0\rangle$, $\mathfrak{u}_{-1,-3}=\langle e_4\rangle$, $\mathfrak{u}_{-1,-1}=\langle e_3\rangle$, $\mathfrak{u}_{-1,1}=\langle e_1\rangle$, $\mathfrak{u}_{-1,3}=\langle e_2\rangle$, $\mathfrak{u}_{0,-2}=\langle e_6\rangle$, and $\mathfrak{u}_{0,2}=\langle e_5\rangle$.

We now need several basic facts about the Tanaka-theoretic prolongation constructions in \cite{SYKES2023108850}. We recall these facts here, and refer the reader to \cite{SYKES2023108850} for a full description of the Tanaka theory.  Geometric Tanaka prolongations $P^1$ and $P^2$ of $P^{0,\mathrm{red}}$ constructed in \cite[Section 9]{SYKES2023108850} are special fiber bundles  fitting into a sequence of fiber bundles
\[
 \mathbb{C} \mathcal{N} \xleftarrow{\pi} \mathbb{C} M \xleftarrow{\mathrm{pr}_0=\mathrm{pr}} P^{0,\mathrm{red}}\xleftarrow{\mathrm{pr}_1} P^1  \xleftarrow{\mathrm{pr}_2} P^2,
\]
where $\mathbb{C} M$ and $\mathbb{C} \mathcal{N}$ denote the complexified CR manifold and Levi leaf space, defined locally in \cite{SYKES2023108850} essentially by just replacing real local coordinates with complex ones. The complex structure on $\mathbb{C} M$ induces an antilinear involution on $\mathfrak{u}$ that extends the involution already defined on $\mathfrak{g}^{0,\mathrm{red}}$. A structure on the Levi leaf space $\mathbb{C} \mathcal{N}$ called a \emph{dynamical Legendrian contact structure (DLC structure)} is introduced in \cite{SYKES2023108850}, and this structure is induced by the CR structure on $\mathbb{C} M$. Symmetries of the DLC structure on $\mathbb{C} \mathcal{N}$ have a naturally induced action on each bundle $P^{0,\mathrm{red}}$, $P^1$, and $P^2$. The prolongation procedure also yields an absolute parallelism on $P^2$ identifying each tangent space in $P^2$ with $\mathfrak{u}$, and the symmetry group of $(\mathbb{C} M,H)$ is embedded in the symmetry group of the DLC structure on $\mathbb{C} \mathcal{N}$ which in turn is embedded in the symmetries of this parallelism via the aforementioned natural action on $P^2$.

Let us fix a set of points $q\in \mathcal{N}$, $\psi_0\in P^{0,\mathrm{red}}$,  and $\psi_1\in P^1$ satisfying
\[
\mathrm{pr}_1(\psi_1)=\psi_0
\quad\mbox{ and }\quad
\mathrm{pr}_1(\psi_0)=q,
\] and let $G$ denote the symmetry group of the complexified CR manifold $(\mathbb{C} M,H)$ with Lie algebra $\mathfrak{g}$. Let $G_q$ be the subgroup of $G$ whose induced action on $\mathbb{C} \mathcal{N}$ fixes the point $q$; let $G_0$ be the subgroup of $G_q$ whose induced action on $P^{0,\mathrm{red}}$ fixes the point $\psi_0$; and let $G_1$ be the subgroup of $G_0$ whose induced action on $P^1$ fixes the point $\psi_1$.  The Tanaka prolongation procedure naturally induces injective linear maps
\[
\mathrm{gr}(\mathfrak{g})_{-}:=\mathfrak{g}/\mathrm{Lie}(G_q)\hookrightarrow\mathfrak{u}_{-2}\oplus \mathfrak{u}_{-1},
\quad
\mathrm{gr}(\mathfrak{g})_{0}:=\mathrm{Lie}(G_q)/\mathrm{Lie}(G_0)\hookrightarrow\mathfrak{u}_{0},
\]
\[
\mathrm{gr}(\mathfrak{g})_{1}:=\mathrm{Lie}(G_0)/\mathrm{Lie}(G_1)\hookrightarrow\mathfrak{u}_{1},
\quad\mbox{ and }\quad
\mathrm{gr}(\mathfrak{g})_{2}:=\mathrm{Lie}(G_1)\hookrightarrow\mathfrak{u}_{2}
\]
such that the graded Lie algebra
\[
\mathrm{gr}(\mathfrak{g}):=\mathrm{gr}(\mathfrak{g})_{-}\oplus \mathrm{gr}(\mathfrak{g})_{0}\oplus \mathrm{gr}(\mathfrak{g})_{1}\oplus \mathrm{gr}(\mathfrak{g})_{2}
\]
obtained from the filtration $\mathfrak{g}\supset \mathrm{Lie}(G_q)\supset \mathrm{Lie}(G_0)\supset \mathrm{Lie}(G_1)$ is mapped homomorphically onto a subalgebra of $\mathfrak{u}$. For notational convenience, let us simply identify $\mathrm{gr}(\mathfrak{g})$ with that subalgebra. The last fact that we will take for granted here, which follows from the construction of the parallelism on $P^2$, is that infinitesimal symmetries on $(\mathbb{C} M, H)$ at a point $o\in \mathbb{C} M$ whose values at $o$ are nonzero are identified (modulo symmetries in $\mathrm{Lie}(G_o)$) in one-to-one correspondence with elements in $\mathfrak{u}_{-2}\oplus \mathfrak{u}_{-1}\oplus \mathfrak{u}_{0,-2}\oplus \mathfrak{u}_{0,2}$, where informally this identification is given by lifting the infinitesimal symmetry to $P^0$ and taking its value at a point in the fiber above $o$.

This last point implies $\mathfrak{u}_{-2}\oplus \mathfrak{u}_{-1}\oplus \mathfrak{u}_{0,-2}\oplus \mathfrak{u}_{0,2}\subset \mathrm{gr}(\mathfrak{g})$. Yet the only $9$-dimensional (complex) subalgebra of $\mathfrak{u}$ having a grading compatible with \eqref{prol alg grading} containing $\mathfrak{u}_{-2}\oplus \mathfrak{u}_{-1}\oplus \mathfrak{u}_{0,-2}\oplus \mathfrak{u}_{0,2}$ is $\mathfrak{u}_{-2}\oplus \mathfrak{u}_{-1}\oplus \mathfrak{u}_{0}$, and hence
\[
\mathrm{gr}(\mathfrak{g})= \mathfrak{u}_{-2}\oplus \mathfrak{u}_{-1}\oplus \mathfrak{u}_{0}= \mathfrak{g}^{0,\mathrm{red}}.
\]
By \cite[Lemma 3]{doubrov1999contact}, $\mathfrak{g}\cong\mathrm{gr}(\mathfrak{g})$, and hence $\mathfrak{g}\cong \mathfrak{g}^{0,\mathrm{red}}$. For any point $o\in P^{0,\mathrm{red}}$, the isotropy subgroup $G_o$ of $G$ fixing $o$ acts freely and transitively on the fiber $P^{0,\mathrm{red}}_o$, and hence the Lie algebra of $G_o$ is exactly the Lie algebra $\mathfrak{g}_{0,0}^{\mathrm{red}}$ of the $P^{0,\mathrm{red}}$ bundle's structure group. The symmetry group $\Re G$ and isotropy subgroup $\Re G_o$  of $(M,H)$ are generated by the real forms $\Re \mathfrak{g}^{0,\mathrm{red}}$ and $\Re\mathfrak{g}_{0,0}^{\mathrm{red}}$  defined as fixed points in $\mathfrak{g}^{0,\mathrm{red}}$ and $\mathfrak{g}_{0,0}^{\mathrm{red}}$ of the involution on $\mathfrak{g}^{0,\mathrm{red}}$ induced by the complex structure on $\mathbb{C} M$.  Since $M= \Re G/ \Re G_0$ and its CR structure descends from the left invariant distribution $\mathfrak{g}_{-1,1}\oplus\mathfrak{u}_{0,0}\oplus \mathfrak{u}_{0,2}$, this shows that $(M,H)$ is locally equivalent to the flat structure generated by $\mathfrak{g}^{0,\mathrm{red}}$.

\section{Extending and linking  ARMS}\label{Extending and linking  ARMS}

In this section we describe two processes, extending ARMS and linking ARMS, by which we can construct new ARMS satisfying the subalgebra property (Definition \ref{subalgebra property}) from others.  We describe all flat structures of dimension at most $11$ that are generated by combinations of ARMS extensions and ARMS links built from the classification in Section \ref{The classification of flat structures in C4 generated by ARMS} of ARMS that generate $7$-dimensional flat structures and the $5$-dimensional flat structure generated by the ARMS encoded by $(H_\ell,A,\Omega,\mathscr{A}_0)$ with $1\times1$ matrices $H_{\ell}=A=1$, $\Omega=0$, and $\mathscr{A}_{0}=\mathbb{C}$ --  which happens to be the only ARMS that generates a model of dimension less than $7$. While this does not account for all ARMS that generate flat structures of dimension at most $11$, it does give a large set of new flat structure examples resulting almost immediately from the ARMS of Theorem \ref{ch3 main 7d theorem a}.

For an ARMS $\left\{H_\ell, A_1,\ldots,A_r, \Omega_1,\ldots,\Omega_r, \mathscr{A}_0 \right\}$ satisfying the subalgebra property (Definition \ref{subalgebra property}), there we two new ARMS called the \emph{positive $2$-dimensional extension} and the \emph{negative $2$-dimensional extension}; these two new ARMS are of the form $\left\{\widetilde{H_\ell}, \widetilde{A_1},\ldots, \widetilde{A_r}, \widetilde{\Omega_1},\ldots, \widetilde{\Omega_r}, \widetilde{\mathscr{A}_0} \right\}$ with
\begin{align}\label{2 dimensional ARMS extension}
\widetilde{H_\ell}=
\left(
\begin{array}{c;{2pt/2pt}c}
H_\ell & 0 \\\hdashline[2pt/2pt]
0& \epsilon
\end{array}
\right),
\,
\widetilde{A_j}=
\left(
\begin{array}{c;{2pt/2pt}c}
A_j & 0 \\\hdashline[2pt/2pt]
0& 0
\end{array}
\right),
\end{align}
\begin{align}\label{2 dimensional ARMS extension a}
\widetilde{\Omega_j}=
\left(
\begin{array}{c;{2pt/2pt}c}
\Omega_j & 0 \\\hdashline[2pt/2pt]
0& 0
\end{array}
\right),\,
\widetilde{\mathscr{A}_0}=
\left\{\left.
\left(
\begin{array}{c;{2pt/2pt}c}
\alpha& 0 \\\hdashline[2pt/2pt]
0& 0
\end{array}
\right)
\right|
\alpha\in\mathscr{A}_0
\right\}
\end{align}
for all $j\in\{1,\ldots, r\}$, where $\epsilon=\pm1$. If $\epsilon=1$ then \eqref{2 dimensional ARMS extension} and \eqref{2 dimensional ARMS extension a}  gives the positive $2$-dimensional extension, whereas $\epsilon=-1$ gives the negative $2$-dimensional extension. It is easily checked that these extensions also satisfy the subalgebra property (Definition \ref{subalgebra property}), thereby generating flat structures themselves, and that, up to a change of sign (positive or negative), these extensions do not depend on the specific matrix representation of the ARMS represented by $\left\{H_\ell, A_1,\ldots,A_r, \Omega_1,\ldots,\Omega_r, \mathscr{A}_0 \right\}$. We call them $2$-dimensional extensions because they generate flat structures whose dimension is two greater than that generated by $\left\{H_\ell, A_1,\ldots,A_r, \Omega_1,\ldots,\Omega_r, \mathscr{A}_0 \right\}$. Repeatedly constructing these $2$-dimensional extensions yields higher dimensional extensions.

\begin{definition}\label{ARMS extension}
For two non-negative integers $p,q\in \mathbb{Z}$ and an ARMS $\mathfrak{g}^{0,\mathrm{red}}$ satisfying the subalgebra property (Definition \ref{subalgebra property}) represented by $\left\{H_\ell, A_1,\right.$ $\left.\ldots,A_r, \Omega_1,\ldots,\Omega_r, \mathscr{A}_0 \right\}$, the \emph{$2(p+q)$-dimensional signature $(p,q)$ extensions of $\mathfrak{g}^{0,\mathrm{red}}$} are the ARMS obtained by applying the positive $2$-dimensional extension to  $\mathfrak{g}^{0,\mathrm{red}}$ in sequence $m$ times for $m\in \{p,q\}$ followed by applying the negative $2$-dimensional extension to  $\mathfrak{g}^{0,\mathrm{red}}$ in sequence $p+q-m$ times.
\end{definition}
The order that the $2$-dimensional extensions are applied in Definition \ref{ARMS extension} does not actually matter, because changing this order yields an equivalent ARMS in the sense of Definition \ref{equivalence of ARMS}. There can be two non-equivalent $2(p+q)$-dimensional signature $(p,q)$ extensions, which correspond to taking $m=p$ or $m=q$ in Definition \ref{ARMS extension}, but it can also happen that either choice $m=p$ or $m=q$ yields the same ARMS.

By applying $2$-dimensional and $4$ dimensional extensions to the ARMS of Theorem \ref{ch3 main 7d theorem a}, we obtain eleven $9$-dimensional flat structures, represented by vertices in the middle ring of the graph in Figure \ref{extending ARMS graph}, and  $19$ $11$-dimensional models, represented vertices in the outer ring of the graph in Figure \ref{extending ARMS graph}. Note also that the two structures of types IV.A and IV.B are the $2$-dimensional ARMS extensions of the unique ARMS that generates a $5$ dimensional flat $2$-nondegenerate structure.

Let us now define \emph{linking of ARMS}  with Levi-kernel dimension $1$. To prepare this we need to introduce compatibility criteria under which two ARMS can be \emph{linked}.

Consider the sesquilinear maps of the form
\[
\Psi\circ \Phi:\mathbb{C}^2\to \mathbb{C}
\]
for which $A$ is the matrix from a tuple
\begin{align}\label{tuple representing ARMS}
\left\{H_\ell, A, \Omega, \mathscr{A}_0 \right\}
\quad\mbox{ (with $A=A_1$ and $\Omega=\Omega_1$)}
\end{align}
representing an ARMS $\mathfrak{g}^{0,\mathrm{red}}$ satisfying the subalgebra property (Definition \ref{subalgebra property}), and with
$
\Psi:  \mathscr{A}_0 \to \mathbb{C}
$
and 
$
\Phi:\mathbb{C}^2\to \mathscr{A}_0 
$
defined as follows. For a  representation as in \eqref{tuple representing ARMS} of an ARMS satisfying the subalgebra property, define 
\begin{align}\label{psi map}
\Psi(\alpha):= \eta_{\alpha,1}^1 
\quad\quad\forall\,\alpha\in\mathscr{A}_0,
\end{align}
where  $\eta_{\alpha,1}^1$ is defined by item $(i)$ in \eqref{system},
and let $\Phi$ be the sesquilinear map given by
\begin{align}\label{phi map}
\Phi(a,b):=\overline{a}b\left(\left[\overline{H_\ell^{-1}\Omega^T H_\ell},\Omega \right]+A\overline{A}-\left(\overline{\mu_{1,1}^1}\Omega +\mu_{1,1}^1 \overline{H_\ell^{-1}\Omega^T H_\ell}\right) \right)
\end{align}
where the $\mu_{1,1}^1$ is the  coefficient defined by item (iii) in \eqref{system}. Notice that $\Phi(1,1)$ is exactly the matrix of item (iv) in \eqref{system}. To see that $\Phi(a,b)$ also belongs to $\mathscr{A}_0$ in general, consider the index $1$ matrix in \eqref{g0minus} rescaled by $\overline{a}$ and the index $1$ matrix in \eqref{g0plus} rescaled by $b$ and compute the Lie bracket of these two matrices. By the subalgebra property (Definition \ref{subalgebra property}), this Lie bracket belongs to $\mathfrak{g}^{\mathrm{red}}_0$, and taking its $\mathfrak{g}_{0,0}^{\mathrm{red}}$ part with respect to \eqref{modified symbol decomposition a}, the upper  left $(n-1)\times (n-1)$ block of this matrix is $\Phi(a,b)$, implying by \eqref{g00} that $\Phi(a,b)\in \mathscr{A}_0$.

The system \eqref{system} can be simplified significantly in our present setting because $r=1$.
\begin{lemma}\label{ARMS normalization}
Let $\mathfrak{g}^{0,\mathrm{red}}$ be an ARMS of Levi kernel dimension $1$ with a representation $\{H_\ell, A,\Omega,\mathscr{A}_0\}$ as in Lemma \ref{matrix representation lemma}. The matrices $A$ and $\Omega$ can be rescaled such that the coefficient $\mu_{1,1}^1$ in \eqref{system} is either $1$ or $0$. Furthermore, if there exists $\alpha\in\mathscr{A}_0 $ for which the $\eta_{\alpha,1}^1$ coefficient in \eqref{system} is nonzero then $\Omega$ can be chosen so that $\mu_{1,1}^1=0$.

Lastly, if $\{H_\ell, A,\Omega,\mathscr{A}_0\}$ can be chosen such that $\mu_{1,1}^1=0$ then $\{H_\ell, A,\Omega,\mathscr{A}_0\}$ can be chosen so that both $\mu_{1,1}^1=0$ and $\Psi\circ\Phi(1,1)\in\{-1,0,1\}$, where, furthermore, this normalized value of $\Psi\circ\Phi(1,1)$ is uniquely determined by $\mathfrak{g}^{0,\mathrm{red}}$.
\end{lemma}
\begin{proof}
Suppose $\mu_{1,1}^1\neq0$. By replacing $\{H_\ell, A,\Omega,\mathscr{A}_0\}$ with $\left\{H_\ell, \tfrac{1}{\mu_{1,1}^1}A,\tfrac{1}{\mu_{1,1}^1}\Omega,\mathscr{A}_0\right\}$, item (iii) in \eqref{system} implies that the new $\mu_{1,1}^1$ coefficient for this new matrix representation of $\mathfrak{g}^{0,\mathrm{red}}$ is equal to $1$.

Suppose instead that  there exists $\alpha\in\mathscr{A}_0 $ for which the $\eta_{\alpha,1}^1$ coefficient in \eqref{system} is nonzero.  Recall that the decomposition in \eqref{modified symbol decomposition a} is not unique. Rather, choosing $\Omega$ and $A$ determines this splitting via \eqref{g0plus}. The different splittings are in $1$-to-$1$ correspondence representations of $\mathfrak{g}^{0,\mathrm{red}}$ of the form $\{H_\ell, A,\Omega+\alpha,\mathscr{A}_0\}$  for $\alpha\in \mathscr{A}_0$. In particular, if $\eta_{\alpha,1}^1\neq0$ then 
\begin{align}\label{new ARMS}
\left\{H_\ell, A,\Omega-\mu_{1,1}^1\left(\overline{\eta_{\alpha,1}^1}\right)^{-1}\overline{H_\ell}\alpha^*\overline{H_\ell}^{-1},\mathscr{A}_0\right\}
\end{align}
represents the same ARMS as $\{H_\ell, A,\Omega,\mathscr{A}_0\}$ because $\overline{H_\ell}\alpha^*\overline{H_\ell}^{-1}$ also belongs $\mathscr{A}_0$. After replacing $\{H_\ell, A,\Omega,\mathscr{A}_0\}$ by \eqref{new ARMS}, the new $\mu_{1,1}^1$ coefficient appearing in item (iii) of \eqref{system} is zero.

Lastly, suppose now that $\{H_\ell, A,\Omega,\mathscr{A}_0\}$ is chosen such that $\mu_{1,1}^1=0$. The map $\Psi\circ\Phi$ depends only on the structure coefficients of the  Lie algebra $\mathfrak{g}^{0,\mathrm{red}}$ and the splitting in \eqref{modified symbol decomposition a}, but not the basis of $\mathfrak{g}_{-}$ with respect to which we obtain its matrix representations. So we need to consider all possible representations $\{\widetilde{H_\ell}, \widetilde{A},\widetilde{\Omega},\widetilde{\mathscr{A}_0}\}$ of $\mathfrak{g}^{0,\mathrm{red}}$ as in Lemma \ref{matrix representation lemma} for which $\mu_{1,1}^1=0$ given with respect to any fixed basis of the form in \eqref{Heisenberg basis}. In general, these representations have the form
\begin{align}\label{new ARMS 2}
\widetilde{H_\ell}=a H_\ell,
\quad
\widetilde{A}=b A,
\quad
\widetilde{\Omega}=b \Omega+\alpha,
\quad
\widetilde{\mathscr{A}_0}=\mathscr{A}_0,
\end{align}
where $a\in \mathbb{R}$, $b\in \mathbb{C}$, and $\alpha\in {\mathscr{A}_0}$ has the coefficient $\eta_{\alpha,1}^1=0$ in the system \eqref{system} calculated with respect to $\{H_\ell, A,\Omega,\mathscr{A}_0\}$; indeed $\alpha$ must satisfy $\eta_{\alpha,1}^1=0$ in \eqref{new ARMS 2} because otherwise the coefficient $\mu_{1,1}^1$ in the system \eqref{system} calculated with respect to $\{\widetilde{H_\ell}, \widetilde{A},\widetilde{\Omega},\widetilde{\mathscr{A}_0}\}$ would be nonzero. 

Letting $\widetilde{\Psi}\circ\widetilde{\Phi}$ be the maps defined by \eqref{psi map} and \eqref{phi map} with respect to $\{\widetilde{H_\ell}, \widetilde{A},\widetilde{\Omega},\widetilde{\mathscr{A}_0}\}$ as in \eqref{new ARMS 2}, we have $\widetilde{\Psi}\circ\widetilde{\Phi}(1,1)=\Psi\circ\Phi(b,b)=|b|^2\Psi\circ\Phi(1,1)$. Therefore, by changing the matrix representation of $\mathfrak{g}^{0,\mathrm{red}}$ while keeping $\mu_{1,1}^1=0$ we can only change $\Psi\circ\Phi(1,1)$ by a positive coefficient, but any positive coefficient can be achieved. Since $\Psi\circ\Phi$ is sesquilinear,  $\Psi\circ\Phi(1,1)\in \mathbb{R}$, and hence the possible transformations of the matrix representations in \eqref{new ARMS 2} allow us to normalize $ \Psi\circ\Phi(1,1)$ to equal one of the values in $\{-1,0, 1\}$ uniquely determined by the original sign of $\Psi\circ\Phi(1,1)$. 
\end{proof}

Noting Lemma \ref{ARMS normalization}, we adopt the following normalization convention.
\begin{definition}[normalization condition]
A tuple $\{H_\ell, A_1,\Omega_1,\mathscr{A}_0\}$ representing an ARMS $\mathfrak{g}^{0,\mathrm{red}}$ satisfying the subalgebra property (Definition \ref{subalgebra property})  with Levi kernel dimension $1$ is \emph{normalized} if either 
\begin{enumerate}[label=(case \arabic*),leftmargin=2cm]
\item $\quad\quad\mu_{1,1}^1=0$ and  $\Psi\circ\Phi(1,1)\in\{-1,0,1\}$, or
\item  $\quad\quad\mu_{1,1}^1=1$,
\end{enumerate}
where in both cases $\mu_{1,1}^1$ is the coefficient appearing in \eqref{system}.
\end{definition} 
By Lemma \ref{ARMS normalization}, the value of $\mu_{1,1}^1$ in this normalization is uniquely determined by $\mathfrak{g}^{0,\mathrm{red}}$, and every ARMS $\mathfrak{g}^{0,\mathrm{red}}$ indeed has a normalized representation $\{H_\ell, A_1,\Omega_1,\mathscr{A}_0\}$. Using the normalization convention, we characterize compatibility of ARMS. 
\begin{definition}[compatibility criteria]\label{compatibility criteria definition}
Two ARMS satisfying the subalgebra property (Definition \ref{subalgebra property}) with Levi kernel dimension $1$ are \emph{compatible} if they have  normalized representations such that 
\begin{itemize}
\item the normalized representations have the same respective values of $\mu_{1,1}^1$ and  $\Psi\circ\Phi(1,1)$, and
\item the image of the linear map $\Psi:\mathscr{A}_0\to \mathbb{C}$ defined in \eqref{psi map} with respect to each of the normalized representations is the same.
\end{itemize}
\end{definition}
We can now define the linking of compatible ARMS with Levi kernel dimension $1$. 

\begin{definition}\label{linking of ARMS}
Let $\mathfrak{g}^{0,\mathrm{red}}$ and $\hat{\mathfrak{g}}^{0,\mathrm{red}}$ be compatible ARMS satisfying the subalgebra property (Definition \ref{subalgebra property}) with Levi kernel dimension $1$ and with respective normalized representations $\{ {H_\ell},  {A}, {\Omega}, {\mathscr{A}_0}\}$ and $\{\hat{H_\ell}, \hat{A},\hat{\Omega},\hat{\mathscr{A}_0}\}$. A \emph{link of $\mathfrak{g}^{0,\mathrm{red}}$ and $\hat{\mathfrak{g}}^{0,\mathrm{red}}$} is an ARMS $\widetilde{\mathfrak{g}}^{0,\mathrm{red}}$ of the form $\{\widetilde{H_\ell}, \widetilde{A},\widetilde{\Omega},\widetilde{\mathscr{A}_0}\}$ with
\begin{align}\label{ARMS links}
\widetilde{H_\ell}=
\left(
\begin{array}{c;{2pt/2pt}c}
H_\ell & 0 \\\hdashline[2pt/2pt]
0& \epsilon \hat{H_\ell}
\end{array}
\right),
\quad
\widetilde{A}=
\left(
\begin{array}{c;{2pt/2pt}c}
A & 0 \\\hdashline[2pt/2pt]
0& \hat{A}
\end{array}
\right),
\quad
\widetilde{\Omega}=
\left(
\begin{array}{c;{2pt/2pt}c}
\Omega & 0 \\\hdashline[2pt/2pt]
0& \hat{\Omega}
\end{array}
\right)
\end{align}
and
\[
\mathscr{A}_0:=
\left\{\left.
\left(
\begin{array}{c;{2pt/2pt}c}
\alpha& 0 \\\hdashline[2pt/2pt]
0& \hat\alpha 
\end{array}
\right)
\right|
\alpha\in\mathscr{A}_0,\,\hat{\alpha}\in\hat{\mathscr{A}_0 },\mbox{ and }\eta_{\alpha,1}^1=\eta_{\hat{\alpha},1}^1
\right\},
\]
where $\epsilon=\pm1$, and $\eta_{\alpha,1}^1$ and $\eta_{\hat{\alpha},1}^1$ are the coefficients of item (i) in \eqref{system} calculated with respect to $\{ {H_\ell},  {A}, {\Omega}, {\mathscr{A}_0}\}$ and $\{\hat{H_\ell}, \hat{A},\hat{\Omega},\hat{\mathscr{A}_0}\}$ respectively.
\end{definition}

By construction, the linked ARMS of Definition \eqref{linking of ARMS} satisfy the subalgebra property, and therefore generate homogeneous $2$-nondegenerate hypersurface-type CR manifolds. As with the extension construction of Definition \eqref{ARMS extension}, linking a pair of compatible ARMS can yield two  non-equivalent links, which correspond to taking $\epsilon=1$ or $\epsilon=-1$ in Definition \ref{linking of ARMS}, but it can also happen that either choice of $\epsilon$ yields the same link.

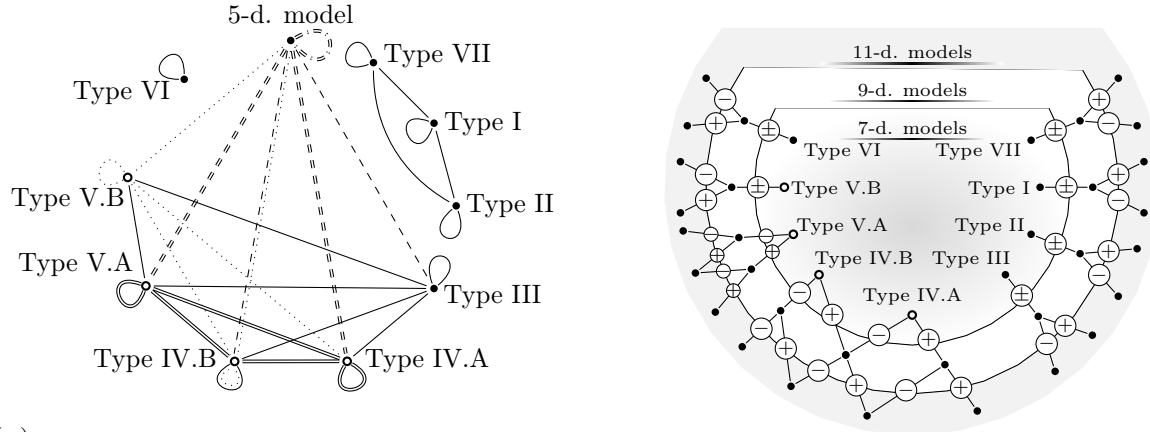
\begin{figure}
\centering
\hfill

\begin{subfigure}[b]{0.46\textwidth}

\scalebox{1}{
\begin{tikzpicture}[
	]
\coordinate (N7) at (11*180/6:2.2 cm);

\coordinate (N8) at (0*180/6:2.2 cm);

\coordinate (N9) at (1*180/6:2.2 cm); 

\coordinate (N1) at (2*180/6:2.2 cm);

\coordinate(N0) at (3*180/6:2.2 cm);

\coordinate (N2) at (4.335*180/6:2.2 cm);

\coordinate (N3) at (5.67*180/6:2.2 cm);

\coordinate (N4) at (6.965*180/6:2.2 cm);

\coordinate (N5) at (8.34*180/6:2.2 cm);

\coordinate (N6) at (9.665*180/6:2.2 cm);

\draw[dash dot,double, shorten >=2pt, shorten <=2pt] (N0) to [out=30,in=-50,looseness=8,min distance=9mm] (N0);

\draw[shorten >=2pt, shorten <=2pt] (N1) to [out=200,in=110,looseness=8,min distance=7mm] (N1);

\draw[shorten >=2pt, shorten <=2pt] (N9) to [out=250,in=160,looseness=8,min distance=7mm] (N9);

\draw[shorten >=2pt, shorten <=2pt] (N8) to [out=220,in=290,looseness=8,min distance=7mm] (N8);

\draw[shorten >=2pt, shorten <=2pt] (N7) to [out=40,in=110,looseness=8,min distance=7mm] (N7);

\draw[double, shorten >=2pt, shorten <=2pt] (N6) to [out=240,in=330,looseness=8,min distance=7mm] (N6);

\draw[shorten >=2pt, shorten <=2pt] (N5) to [out=-60,in=-150,looseness=8,min distance=7mm] (N5);
\draw[dotted, shorten >=2pt, shorten <=2pt] ([shift={(-.015,-.07)}]N5) to [out=-60,in=-150,looseness=8,min distance=4.5mm] ([shift={(-.015,-.07)}]N5);

\draw[double, shorten >=2pt, shorten <=2pt] (N4) to [out=250,in=160,looseness=8,min distance=7mm] (N4);

\draw[dotted, shorten >=2pt, shorten <=2pt] (N3) to [out=210,in=120,looseness=8,min distance=7mm] (N3);

\draw[shorten >=2pt, shorten <=2pt] (N2) to [out=180,in=90,looseness=8,min distance=7mm] (N2);

\draw[thin, double, shorten >=3pt, shorten <=3pt] (N4) -- (N6);


\draw[thin, dotted, shorten >=3pt, shorten <=4pt] (N0) -- (N3);

\draw[thin, double, dashed, shorten >=3pt, shorten <=4pt] (N0) -- (N4);

\draw[thin, dashed, shorten >=3pt, shorten <=4pt] ([shift={(-.02,0)}] N0) -- ([shift={(-.025,0)}]N5);
\draw[thin, dotted, shorten >=3pt, shorten <=4pt] ([shift={(.02,0)}]N0) -- ([shift={(.025,0)}]N5);

\draw[thin, double, dashed, shorten >=3pt, shorten <=4pt] (N0) -- (N6);

\draw[thin, dashed, shorten >=3pt, shorten <=4pt] (N0) -- (N7);


\draw[thin, shorten >=3pt, shorten <=3pt] (N3) -- (N7);

\draw[thin, shorten >=3pt, shorten <=3pt] (N4) -- (N7);

\draw[thin, shorten >=3pt, shorten <=3pt] (N5) -- (N7);

\draw[thin, shorten >=3pt, shorten <=3pt] (N6) -- (N7);



\draw[thin, shorten >=3pt, shorten <=3pt] (N3) -- (N4);

\draw[thin, double, shorten >=3pt, shorten <=3pt] (N4) -- (N5);

\draw[thin, double, shorten >=3pt, shorten <=3pt] (N5) -- (N6);


\draw[thin, dotted, shorten >=3pt, shorten <=3pt] (N3) -- (N5);


\draw[thin, dotted, shorten >=3pt, shorten <=3pt] (N3) -- (N6);

\draw[thin, shorten >=3pt, shorten <=3pt] (N8) -- (N9);

\draw[thin, shorten >=3pt, shorten <=3pt] (N1) -- (N9);

\draw[thin, shorten >=3pt, shorten <=3pt] (N1) to [out=270,in=150] (N8);

\fill[black] (N0) circle (0.05 cm);
	\node[anchor=south] at ([shift={(0,.1)}] N0)  (5d) {\small $5$-d. model};

\fill[black] (N1) circle (0.05 cm);
	\node[anchor=west] at ([shift={(0,.1)}] N1)  (3.7) {\small Type \Rmnum{7}};

\fill[black] (N2) circle (0.05 cm);
	\node[anchor=east] at ([shift={(-.05,-.18)}] N2)  (3.6) {\small Type \Rmnum{6}};
	
\draw[black,thick] (N3) circle (0.05 cm);
	\node[anchor=east] at ([shift={(0.1,-.3)}] N3)  (3.5.b) {\small Type \Rmnum{5}.B};

\draw[black,thick] (N4) circle (0.05 cm);
	\node[anchor=east] at ([shift={(-0.03,.28)}] N4)  (3.5.a) {\small Type \Rmnum{5}.A};

\draw[black,thick] (N5) circle (0.05 cm);
	\node[anchor=north east] at ([shift={(-0.1,.3)}] N5)  (3.4.b) {\small Type \Rmnum{4}.B};

\draw[black,thick] (N6) circle (0.05 cm);
	\node[anchor=north west] at ([shift={(0.1,.3)}] N6)  (3.4.a) {\small Type \Rmnum{4}.A};

\fill[black] (N7) circle (0.05 cm);
	\node[anchor=west] at ([shift={(-0.01,-.1)}] N7)  (3.3) {\small Type \Rmnum{3}};

\fill[black] (N8) circle (0.05 cm) node[anchor=west]{\small Type \Rmnum{2}};

\fill[black] (N9) circle (0.05 cm) node[anchor=west]{\small Type \Rmnum{1}};

\end{tikzpicture}
}

\caption{\scriptsize Vertices in this graph correspond to ARMS generating the flat structures enumerated in Table \ref{main theorem table} as well as the unique ARMS that generates a 5-d. $2$-nondegenerate structure. Vertices are connected by edges if their ARMS are compatible. Each edge corresponds to a linking of ARMS. Vertex pairs linked by two edges represent ARMS with two non-equivalent links. Solid edges denote links that generate $11$-d. flat structures. The two dashed-and-dotted edges linking the 5d. model to itself correspond to the  $7$-d. flat structures of types \Rmnum{5}.A and \Rmnum{5}.B. Purely dashed edges  represent $9$-d. structures. Purely dotted edges represent flat structures that are equivalent to structures obtained from other links in the diagram.}
\label{linking ARMS graph}

\end{subfigure}
\hspace{.5cm}
\begin{subfigure}[b]{0.46\textwidth}

\hspace{.3cm}\scalebox{1}{
\begin{tikzpicture}[
	]

\coordinate (M7) at (8.8*360/10:1.7 cm); 

\coordinate (M8) at (9.4*360/10:1.7 cm); 

\coordinate (M9) at (0*360/10:1.7 cm); 

\coordinate (M1) at (0.6*360/10:1.7 cm); 

\coordinate (M2) at (4.4*360/10:1.7 cm); 

\coordinate (M3) at (5*360/10:1.7 cm); 

\coordinate (M4) at (5.6*360/10:1.7 cm); 

\coordinate (M5) at (6.2*360/10:1.7 cm); 

\coordinate (M6) at (7.5*360/10:1.7 cm); 

%
%
\fill [gray!10,domain=140:400] plot ({8+3.3*cos(\x)}, {0.4+3.3*sin(\x)});
\filldraw [black, inner color=white,domain=144.5:394.5] plot ({8+2.75*cos(\x)}, {0.4+2.75*sin(\x)});
\filldraw [black, inner color=gray!30,domain=150:390] plot ({8+2.1*cos(\x)}, {0.4+2.1*sin(\x)});


\fill[black] ([shift={(8,0.4)}] M1) circle (0.05 cm);
	\node[anchor=east] at ([shift={(8,0.25)}] M1)  (3.7) {\tiny Type \Rmnum{7}};

\fill[black] ([shift={(8,0.4)}] M2) circle (0.05 cm);
	\node[anchor=west] at ([shift={(8,0.25)}] M2)  (3.6) {\tiny Type \Rmnum{6}};

\draw[black,thick] ([shift={(8,0.4)}] M3) circle (0.05 cm);
	\node[shift={(-0.05,0.05)},,anchor=west] at ([shift={(8,0.32)}] M3)  (3.5b) {\tiny Type \Rmnum{5}.B};

\draw[black,thick] ([shift={(8,0.4)}] M4) circle (0.05 cm) node[shift={(-.1,0.1)},anchor=west]{\tiny Type \Rmnum{5}.A};

\draw[black,thick] ([shift={(8,0.4)}] M5) circle (0.05 cm) node[shift={(-.2,0.2)},anchor=west]{\tiny Type \Rmnum{4}.B};

\draw[black,thick] ([shift={(8,0.4)}] M6) circle (0.05 cm) node[anchor=south]{\tiny  Type \Rmnum{4}.A};

\fill[black] ([shift={(8,0.4)}] M7) circle (0.05 cm) node[shift={(.2,0.2)},anchor=east]{\tiny Type \Rmnum{3}};

\fill[black] ([shift={(8,0.4)}] M8) circle (0.05 cm) node[shift={(.05,0.1)},anchor=east]{\tiny Type \Rmnum{2}};

\fill[black] ([shift={(8,0.4)}] M9) circle (0.05 cm);
\node[shift={(0,0.05)},anchor=east] at ([shift={(8,0.32)}] M9)  (3.1) {\tiny Type \Rmnum{1}};

\node[inner sep=2pt] at (8,1.2)  (7d) {\tiny $7$-d. models};
	\path[left color=black,right color=white](7d.south)+(0,0pt) rectangle +(1cm,.5pt); 
	\path[left color=white,right color=black]([shift={(-1,0)}] 7d.south)+(0,0pt) rectangle +(1cm,.5pt); 
\node[inner sep=2pt] at (8,1.7)  (9d) {\tiny $\,9$-d. models$\,$};
	\path[left color=black,right color=white](9d.south)+(0,0pt) rectangle +(1.1cm,.5pt); 
	\path[left color=white,right color=black]([shift={(-1.1,0)}] 9d.south)+(0,0pt) rectangle +(1.1cm,.5pt); 
\node[inner sep=2pt] at (8,2.2)  (11d) {\tiny $\,11$-d. models$\,$};
	\path[left color=black,right color=white](11d.south)+(0,0pt) rectangle +(1.2cm,.5pt); 
	\path[left color=white,right color=black]([shift={(-1.2,0)}] 11d.south)+(0,0pt) rectangle +(1.2cm,.5pt); 

\coordinate (MP1) at (0.6*360/10: 2.4 cm); 
\fill[black] ([shift={(8,0.4)}] MP1) circle (0.05 cm) node[anchor=east]{};
\draw[thin, shorten >=2pt, shorten <=2pt] ([shift={(8,0.4)}] M1) -- ([shift={(8,0.4)}] MP1) node [circle,draw,inner sep=0pt,midway,fill=white] {\tiny $\pm$};
\coordinate (MPP1) at (0.775*360/10: 3.1 cm); 
\fill[black] ([shift={(8,0.4)}] MPP1) circle (0.05 cm) node[anchor=east]{};
\draw[thin, shorten >=2pt, shorten <=2pt] ([shift={(8,0.4)}] MP1) -- ([shift={(8,0.4)}] MPP1) node [circle,draw,inner sep=0pt,midway,fill=white] {\tiny $+$};
\coordinate (MPM1) at (0.425*360/10: 3.1 cm); 
\fill[black] ([shift={(8,0.4)}] MPM1) circle (0.05 cm) node[anchor=east]{};
\draw[thin, shorten >=2pt, shorten <=2pt] ([shift={(8,0.4)}] MP1) -- ([shift={(8,0.4)}] MPM1) node [circle,draw,inner sep=0pt,midway,fill=white] {\tiny $-$};

\coordinate (MP7) at (8.73*360/10: 2.4 cm); 
\fill[black] ([shift={(8,0.4)}] MP7) circle (0.05 cm) node[anchor=east]{};
\draw[thin, shorten >=2pt, shorten <=2pt] ([shift={(8,0.4)}] M7) -- ([shift={(8,0.4)}] MP7) node [circle,draw,inner sep=0pt,midway,fill=white] {\tiny $\pm$};
\coordinate (MPP7) at (8.905*360/10: 3.1 cm); 
\fill[black] ([shift={(8,0.4)}] MPP7) circle (0.05 cm) node[anchor=east]{};
\draw[thin, shorten >=2pt, shorten <=2pt] ([shift={(8,0.4)}] MP7) -- ([shift={(8,0.4)}] MPP7) node [circle,draw,inner sep=0pt,midway,fill=white] {\tiny $+$};
\coordinate (MPM7) at (8.55*360/10: 3.1 cm); 
\fill[black] ([shift={(8,0.4)}] MPM7) circle (0.05 cm) node[anchor=east]{};
\draw[thin, shorten >=2pt, shorten <=2pt] ([shift={(8,0.4)}] MP7) -- ([shift={(8,0.4)}] MPM7) node [circle,draw,inner sep=0pt,midway,fill=white] {\tiny $-$};

\coordinate (MP8) at (9.4*360/10:2.4 cm); 
\fill[black] ([shift={(8,0.4)}] MP8) circle (0.05 cm) node[anchor=east]{};
\draw[thin, shorten >=2pt, shorten <=2pt] ([shift={(8,0.4)}] M8) -- ([shift={(8,0.4)}] MP8) node [circle,draw,inner sep=0pt,midway,fill=white] {\tiny $\pm$};
\coordinate (MPP8) at (9.575*360/10: 3.1 cm); 
\fill[black] ([shift={(8,0.4)}] MPP8) circle (0.05 cm) node[anchor=east]{};
\draw[thin, shorten >=2pt, shorten <=2pt] ([shift={(8,0.4)}] MP8) -- ([shift={(8,0.4)}] MPP8) node [circle,draw,inner sep=0pt,midway,fill=white] {\tiny $+$};
\coordinate (MPM8) at (9.225*360/10: 3.1 cm); 
\fill[black] ([shift={(8,0.4)}] MPM8) circle (0.05 cm) node[anchor=east]{};
\draw[thin, shorten >=2pt, shorten <=2pt] ([shift={(8,0.4)}] MP8) -- ([shift={(8,0.4)}] MPM8) node [circle,draw,inner sep=0pt,midway,fill=white] {\tiny $-$};

\coordinate (MP9) at (0*360/10:2.4 cm); 
\fill[black] ([shift={(8,0.4)}] MP9) circle (0.05 cm) node[anchor=east]{};
\draw[thin, shorten >=2pt, shorten <=2pt] ([shift={(8,0.4)}] M9) -- ([shift={(8,0.4)}] MP9) node [circle,draw,inner sep=0pt,midway,fill=white] {\tiny $\pm$};
\coordinate (MPP9) at (0.175*360/10: 3.1 cm); 
\fill[black] ([shift={(8,0.4)}] MPP9) circle (0.05 cm) node[anchor=east]{};
\draw[thin, shorten >=2pt, shorten <=2pt] ([shift={(8,0.4)}] MP9) -- ([shift={(8,0.4)}] MPP9) node [circle,draw,inner sep=0pt,midway,fill=white] {\tiny $+$};
\coordinate (MPM9) at (-0.175*360/10: 3.1 cm); 
\fill[black] ([shift={(8,0.4)}] MPM9) circle (0.05 cm) node[anchor=east]{};
\draw[thin, shorten >=2pt, shorten <=2pt] ([shift={(8,0.4)}] MP9) -- ([shift={(8,0.4)}] MPM9) node [circle,draw,inner sep=0pt,midway,fill=white] {\tiny $-$};

\coordinate (MP2) at (4.4*360/10:2.4 cm); 
\fill[black] ([shift={(8,0.4)}] MP2) circle (0.05 cm) node[anchor=east]{};
\draw[thin, shorten >=2pt, shorten <=2pt] ([shift={(8,0.4)}] M2) -- ([shift={(8,0.4)}] MP2) node [circle,draw,inner sep=0pt,midway,fill=white] {\tiny $\pm$};
\coordinate (MPM2) at (4.225*360/10: 3.1 cm); 
\fill[black] ([shift={(8,0.4)}] MPM2) circle (0.05 cm) node[anchor=east]{};
\draw[thin, shorten >=2pt, shorten <=2pt] ([shift={(8,0.4)}] MP2) -- ([shift={(8,0.4)}] MPM2) node [circle,draw,inner sep=0pt,midway,fill=white] {\tiny $-$};
\coordinate (MPP2) at (4.575*360/10: 3.1 cm); 
\fill[black] ([shift={(8,0.4)}] MPP2) circle (0.05 cm) node[anchor=east]{};
\draw[thin, shorten >=2pt, shorten <=2pt] ([shift={(8,0.4)}] MP2) -- ([shift={(8,0.4)}] MPP2) node [circle,draw,inner sep=0pt,midway,fill=white] {\tiny $+$};

\coordinate (MP3) at (5*360/10:2.4 cm); 
\fill[black] ([shift={(8,0.4)}] MP3) circle (0.05 cm) node[anchor=east]{};
\draw[thin, shorten >=2pt, shorten <=2pt] ([shift={(8,0.4)}] M3) -- ([shift={(8,0.4)}] MP3) node [circle,draw,inner sep=0pt,midway,fill=white] {\tiny $\pm$};
\coordinate (MPM3) at (4.825*360/10: 3.1 cm); 
\fill[black] ([shift={(8,0.4)}] MPM3) circle (0.05 cm) node[anchor=east]{};
\draw[thin, shorten >=2pt, shorten <=2pt] ([shift={(8,0.4)}] MP3) -- ([shift={(8,0.4)}] MPM3) node [circle,draw,inner sep=0pt,midway,fill=white] {\tiny $-$};
\coordinate (MPP3) at (5.175*360/10: 3.1 cm); 
\fill[black] ([shift={(8,0.4)}] MPP3) circle (0.05 cm) node[anchor=east]{};
\draw[thin, shorten >=2pt, shorten <=2pt] ([shift={(8,0.4)}] MP3) -- ([shift={(8,0.4)}] MPP3) node [circle,draw,inner sep=0pt,midway,fill=white] {\tiny $+$};

\coordinate (MM4) at (5.45*360/10:2.4 cm); 
\fill[black] ([shift={(8,0.4)}] MM4) circle (0.05 cm) node[anchor=east]{};
\draw[thin, shorten >=2pt, shorten <=2pt] ([shift={(8,0.4)}] M4) -- ([shift={(8,0.4)}] MM4) node [circle,draw,inner sep=-1pt,midway,fill=white] {\tiny $-$};
\coordinate (MP4) at (5.75*360/10: 2.4 cm); 
\fill[black] ([shift={(8,0.4)}] MP4) circle (0.05 cm) node[anchor=east]{};
\draw[thin, shorten >=2pt, shorten <=2pt] ([shift={(8,0.4)}] M4) -- ([shift={(8,0.4)}] MP4) node [circle,draw,inner sep=-1pt,midway,fill=white] {\tiny $+$};
\coordinate (MMM4) at (5.3*360/10: 3.1 cm); 
\fill[black] ([shift={(8,0.4)}] MMM4) circle (0.05 cm) node[anchor=east]{};
\draw[thin, shorten >=2pt, shorten <=2pt] ([shift={(8,0.4)}] MM4) -- ([shift={(8,0.4)}] MMM4) node [circle,draw,inner sep=-1pt,midway,fill=white] {\tiny $-$};
\coordinate (MPP4) at (5.9*360/10: 3.1 cm); 
\fill[black] ([shift={(8,0.4)}] MPP4) circle (0.05 cm) node[anchor=east]{};
\draw[thin, shorten >=2pt, shorten <=2pt] ([shift={(8,0.4)}] MP4) -- ([shift={(8,0.4)}] MPP4) node [circle,draw,inner sep=-1pt,midway,fill=white] {\tiny $+$};
\coordinate (MMP4) at (5.6*360/10: 3.1 cm); 
\fill[black] ([shift={(8,0.4)}] MMP4) circle (0.05 cm) node[anchor=east]{};
\draw[thin, shorten >=2pt, shorten <=2pt] ([shift={(8,0.4)}] MM4) -- ([shift={(8,0.4)}] MMP4) node [circle,draw,inner sep=-1pt,midway,fill=white] {\tiny $+$};
\draw[thin, shorten >=2pt, shorten <=2pt] ([shift={(8,0.4)}] MP4) -- ([shift={(8,0.4)}] MMP4) node [circle,draw,inner sep=-1pt,midway,fill=white] {\tiny $-$};

\coordinate (MM5) at (6.2*360/10:2.4 cm); 
\fill[black] ([shift={(8,0.4)}] MM5) circle (0.05 cm) node[anchor=east]{};
\draw[thin, shorten >=2pt, shorten <=2pt] ([shift={(8,0.4)}] M5) -- ([shift={(8,0.4)}] MM5) node [circle,draw,inner sep=0pt,midway,fill=white] {\tiny $-$};
\coordinate (MP5) at (6.9*360/10:2.4 cm); 
\fill[black] ([shift={(8,0.4)}] MP5) circle (0.05 cm) node[anchor=east]{};
\draw[thin, shorten >=2pt, shorten <=2pt] ([shift={(8,0.4)}] M5) -- ([shift={(8,0.4)}] MP5) node [circle,draw,inner sep=0pt,midway,fill=white] {\tiny $+$};
\coordinate (MMM5) at (6.2*360/10:3.1 cm); 
\fill[black] ([shift={(8,0.4)}] MMM5) circle (0.05 cm) node[anchor=east]{};
\draw[thin, shorten >=2pt, shorten <=2pt] ([shift={(8,0.4)}] MM5) -- ([shift={(8,0.4)}] MMM5) node [circle,draw,inner sep=0pt,midway,fill=white] {\tiny $-$};
\coordinate (MPM5) at (6.63*360/10:3.1 cm); 
\fill[black] ([shift={(8,0.4)}] MPM5) circle (0.05 cm) node[anchor=east]{};
\draw[thin, shorten >=2pt, shorten <=2pt] ([shift={(8,0.4)}] MP5) -- ([shift={(8,0.4)}] MPM5) node [circle,draw,inner sep=0pt,midway,fill=white] {\tiny $-$};
\draw[thin, shorten >=2pt, shorten <=2pt] ([shift={(8,0.4)}] MM5) -- ([shift={(8,0.4)}] MPM5) node [circle,draw,inner sep=0pt,midway,fill=white] {\tiny $+$};

\coordinate (MP6) at (7.79*360/10:2.4 cm); 
\fill[black] ([shift={(8,0.4)}] MP6) circle (0.05 cm) node[anchor=east]{};
\draw[thin, shorten >=2pt, shorten <=2pt] ([shift={(8,0.4)}] M6) -- ([shift={(8,0.4)}] MP6) node [circle,draw,inner sep=0pt,midway,fill=white] {\tiny $+$};
\coordinate (MM6) at (6.9*360/10:2.4 cm); 
\fill[black] ([shift={(8,0.4)}] MM6) circle (0.05 cm) node[anchor=east]{};
\draw[thin, shorten >=2pt, shorten <=2pt] ([shift={(8,0.4)}] M6) -- ([shift={(8,0.4)}] MM6) node [circle,draw,inner sep=0pt,midway,fill=white] {\tiny $-$};
\coordinate (MPP6) at (7.95*360/10:3.1 cm); 
\fill[black] ([shift={(8,0.4)}] MPP6) circle (0.05 cm) node[anchor=east]{};
\draw[thin, shorten >=2pt, shorten <=2pt] ([shift={(8,0.4)}] MP6) -- ([shift={(8,0.4)}] MPP6) node [circle,draw,inner sep=0pt,midway,fill=white] {\tiny $+$};
\coordinate (MMP6) at (7.19*360/10:3.1 cm); 
\fill[black] ([shift={(8,0.4)}] MMP6) circle (0.05 cm) node[anchor=east]{};
\draw[thin, shorten >=2pt, shorten <=2pt] ([shift={(8,0.4)}] MM6) -- ([shift={(8,0.4)}] MMP6) node [circle,draw,inner sep=0pt,midway,fill=white] {\tiny $+$};
\draw[thin, shorten >=2pt, shorten <=2pt] ([shift={(8,0.4)}] MP6) -- ([shift={(8,0.4)}] MMP6) node [circle,draw,inner sep=0pt,midway,fill=white] {\tiny $-$};

\end{tikzpicture}
}

\caption{\scriptsize Vertices in this graph correspond to  flat structures generated by ARMS and are arranged in 3 concentric rings. The inner ring has the nine $7$-d. structures enumerated in Table \ref{main theorem table}. Edges connect vertices representing structures generated by ARMS related by a positive or negative $2$-dimensional extension. Vertices in the middle and outer rings correspond to $9$-d. and $11$-d. flat structures respectively. Edges labeled by $+$ (respectively $-$) correspond to structures obtained by applying a positive (resp. negative) extension with respect to the matrix representations of the nine $7$-d. flat structures' ARMS given in Table \ref{main theorem table}. The  $\pm$ edges stem from ARMS having equivalent positive and negative extensions.}
\label{extending ARMS graph}

\end{subfigure}
\hfill
\caption{\small These graphs show extensions and links of ARMS  generating  $9$ and $11$-dimensional flat structures. The $\TikCircle$ vertex label marks the ARMS of $7$-d. structures obtained as links or extensions of the ARMS that generates the unique flat $5$-d. structure.}
\label{linking and extending ARMS graphs}
\end{figure}
By linking pairs of ARMS from among the ARMS of Theorem \ref{ch3 main 7d theorem a} as well as the unique ARMS that generates a $5$-dimensional homogeneous hypersurface, we obtain six $9$-dimensional models, represented by the dashed edges in Figure \ref{linking ARMS graph}, and  $24$ $11$-dimensional models, represented by the solid edges in Figure \ref{linking ARMS graph}.

In total, using the classification of ARMS in Theorem \eqref{ch3 main 7d theorem a} and simple constructions of linking and extending ARMS, we obtain $14$ homogeneous locally  non-equivalent structures in $\mathbb{C}^5$  and $38$ homogeneous locally  non-equivalent structures in $\mathbb{C}^6$. Notice that $9$ and $11$-dimensional structures obtained by linking pairs of ARMS among those of types \Rmnum{4}.A and \Rmnum{4}.B and the $5$-dimensional flat structure's ARMS are equivalent to those obtained from $2$ and $4$-dimensional extensions of the ARMS of types \Rmnum{5}.A and \Rmnum{5}.B, which accounts for three $9$-dimensional and five $11$-dimensional structures appearing in both graphs of Figure \ref{linking and extending ARMS graphs}. All other $9$ and $11$-dimensional structures described by  Figure \ref{linking ARMS graph} do not also appear in the graph of Figure \ref{extending ARMS graph}.

These totals are obtained purely by linking pairs of ARMS and applying extensions without combining these two operations in sequence.  It is natural therefore to check if other structures arise from linking more than two ARMS or by combining the extending and linking constructions, but, in fact, no additional structures arise from such combinations. This last observation readily follows from noticing that all constructions of extending and linking ARMS commute, and the structures of types \Rmnum{4} and \Rmnum{5} are themselves obtained as extensions and linkings of the ARMS generating the $5$-dimensional structure. Specifically, all of these combinations of linking and extending the ARMS of $5$ and $7$-dimensional structures can be reduced to combinations of linking and extending the ARMS in Figure \ref{linking ARMS graph} labeled by $\TikCircleFill$ (i.e., without using the ARMS  labeled by $\TikCircle$).  All possible combinations written in this reduced form are indeed accounted for in the graphs of Figure \ref{linking and extending ARMS graphs}.

\section{Examples in $\mathbb{C}^5$}\label{9-dimensional examples with rank-$1$ Levi kernel}

Having introduced links and extensions of ARMS, the following definition is very natural.
\begin{definition}\label{indecomposable ARMS}
An ARMS is \emph{indecomposable} if it is not equivalent to a combination of links and extensions of other ARMS, and it is  \emph{decomposable} otherwise.
\end{definition}
In addition to the $14$ ARMS found in Section \ref{Extending and linking  ARMS} that generate flat structures in $\mathbb{C}^5$, six other such non-equivalent ARMS with Levi-kernel dimension $1$ are known to the author and they are listed here, given by their matrix representations $\{H_{\ell},A,\Omega,\mathscr{A}_0\}$. All six are indecomposable, as they do not appear in Section \ref{Extending and linking  ARMS}. 

\subsection{Example \ref{9 dimensional non-regular strongly non-nilpotent example}} \label{9 dimensional non-regular strongly non-nilpotent example}
Let  
\begin{align}\label{9 dimensional non-regular strongly non-nilpotent example formula a}
\bgroup
\renewcommand\arraystretch{1}  
A=
 \left(
\begin{array}{ccc}
1 & 1 & 0 \\
0 & 1 & 0 \\
0 & 0 & 1 
\end{array}
\right),
\quad
H_\ell=
\left(
\begin{array}{ccc}
0 & 1 & 0 \\
1 & 0 & 0 \\
0 & 0 & 1
\end{array}
\right),
\egroup
\end{align}
\begin{align}\label{9 dimensional non-regular strongly non-nilpotent example formula b}
\Omega=
\left(
\begin{array}{ccc}
 \frac{1}{2} & 1 & 0 \\
 0 & -\frac{3}{2} & 0 \\
 0 & 0 & -\frac{1}{2}
\end{array}
\right)
\quad\mbox{ and }\quad
\bgroup
\renewcommand\arraystretch{1}  
\mathscr{A}_0=
\left\{
\left.\left(
\begin{array}{ccc}
 0 & 0 & a \\
 0 & 0 & 0 \\
 0 & - a & 0
\end{array}
\right)
\,\right|\,
{a\in\mathbb{C}}
\right\}.
\egroup
\end{align}
By calculating the complete ARMS represented by $(H_\ell,A,\Omega, \mathscr{A}_0)$ and applying \cite[Theorem 6.2]{SYKES2023108850} and \cite[Theorem 3.8]{sykes2021maximal}, we find that this structure's symmetry algebra is $11$-dimensional.

\subsection{Example \ref{9 dimensional other non-regular strongly non-nilpotent example}} \label{9 dimensional other non-regular strongly non-nilpotent example}
Let  
\begin{align}\label{9 dimensional other non-regular strongly non-nilpotent example formula a}
A=
 \left(
\begin{array}{ccc}
0 & -1 & 0 \\
1 & 0 & 0 \\
0 & 0 & \sqrt{3} 
\end{array}
\right),
\quad
\bgroup
\renewcommand\arraystretch{1}  
H_\ell=
\left(
\begin{array}{ccc}
0 & 1 & 0 \\
1 & 0 & 0 \\
0 & 0 & 1
\end{array}
\right)
\egroup
,
\end{align}
\begin{align}\label{9 dimensional other non-regular strongly non-nilpotent example formula b}
\Omega=
\frac{1}{\sqrt{3}}\left(
\begin{array}{ccc}
 \frac{1}{2} & -1 & -\sqrt{3}  \\
 -1 & \frac{1}{2} & \sqrt{3}   \\
1   & 1 & \frac{1}{2}
\end{array}
\right)
\quad\mbox{ and }\quad
\mathscr{A}_0=0.
\end{align}
By calculating the complete ARMS represented by $(H_\ell,A,\Omega, \mathscr{A}_0)$ and applying \cite[Theorem 6.2]{SYKES2023108850} and \cite[Theorem 3.8]{sykes2021maximal}, we find that this structure's symmetry algebra is $10$-dimensional.

\subsection{Example \ref{9 dimensional 3rd non-regular strongly non-nilpotent example}} \label{9 dimensional 3rd non-regular strongly non-nilpotent example}
Let  
\begin{align}\label{9 dimensional 3rd non-regular strongly non-nilpotent example formula a}
\bgroup
\renewcommand\arraystretch{1}  
A=
 \left(
\begin{array}{ccc}
1 & 1 & 0 \\
0 & 1 & 1 \\
0 & 0 & 1 
\end{array}
\right),
\quad
H_\ell=
\left(
\begin{array}{ccc}
0 & 0 & 1 \\
0 & 1 & 0 \\
1 & 0 & 0
\end{array}
\right)
\egroup
,
\quad
\Omega=
\left(
\begin{array}{ccc}
- \frac{3}{2} & -1 & 0 \\
 0 & -\frac{1}{2} & 0 \\
 0 & 0 & \frac{1}{2}
\end{array}
\right)
\end{align}
and 
$
\mathscr{A}_0=0.
$
By calculating the complete ARMS represented by $(H_\ell,A,\Omega, \mathscr{A}_0)$ and applying \cite[Theorem 6.2]{SYKES2023108850} and \cite[Theorem 3.8]{sykes2021maximal}, we find that this structure's symmetry algebra is $10$-dimensional.

\subsection{Example \ref{9 dimensional weakly non-nilpotent example 1}} \label{9 dimensional weakly non-nilpotent example 1}
Let  
\begin{align}\label{9 dimensional 4th non-regular strongly non-nilpotent example formula a}
A=
\bgroup
\renewcommand\arraystretch{1}  
 \left(
\begin{array}{ccc}
0 & 1 & 0 \\
0 & 0 & 0 \\
0 & 0 & 1 
\end{array}
\right),
\quad
H_\ell=
\left(
\begin{array}{ccc}
0 & 1 & 0 \\
1 & 0 & 0 \\
0 & 0 & 1
\end{array}
\right),
\egroup
\end{align}
\begin{align}\label{9 dimensional 4th non-regular strongly non-nilpotent example formula b}
\Omega=
\left(
\begin{array}{ccc}
- \frac{3}{2} & -1 & 0 \\
 0 & -\frac{1}{2} & 0 \\
 0 & 0 & \frac{1}{2}
\end{array}
\right)
\quad\mbox{ and }\quad
\bgroup
\renewcommand\arraystretch{1}  
\mathscr{A}_0=
\left\{
\left.\left(
\begin{array}{ccc}
 0 & a & 0 \\
 0 & 0 & 0 \\
 0 & 0 & 0
\end{array}
\right)
\,\right|\,
a\in\mathbb{C}
\right\}.
\egroup
\end{align}
By calculating the complete ARMS represented by $(H_\ell,A,\Omega, \mathscr{A}_0)$ and applying \cite[Theorem 6.2]{SYKES2023108850} and \cite[Theorem 3.8]{sykes2021maximal}, we find that this structure's symmetry algebra is $11$-dimensional.

\subsection{Example \ref{9 dimensional submaximal nilpotent example 2}.} \label{9 dimensional submaximal nilpotent example 2}
Let  
\[
\bgroup
\renewcommand\arraystretch{1}  
A=
 \left(
\begin{array}{ccc}
0 & 1 & 0 \\
0 & 0 & 1 \\
0 & 0 & 0 
\end{array}
\right)
\quad\mbox{ and }\quad
H_\ell=
\left(
\begin{array}{ccc}
0 & 0 & 1 \\
0 & 1 & 0 \\
1 & 0 & 0
\end{array}
\right).
\egroup
\]
The ARMS represented by this pair $(H_\ell, A)$  and
\begin{align}\label{nilpotent 9d symbol}
\bgroup
\renewcommand\arraystretch{1}  
\Omega=
\left(
\begin{array}{ccc}
 0 & 0 & 0 \\
 0 & 0 & 0 \\
 0 & 0 & 0
\end{array}
\right)
\quad\mbox{ and }\quad
\mathscr{A}_0=
\left\{
\left.\left(
\begin{array}{ccc}
 0 & 0 & a \\
 0 & 0 & 0 \\
 0 & 0 & 0
\end{array}
\right)
\,\right|\,
a\in\mathbb{C}
\right\},
\egroup
\end{align}
generates a structure that was studied in \cite{porter2021absolute}, whereas the ARMS given by  $(H_\ell, A)$  and
\begin{align}\label{nilpotent 9d symbol}
\bgroup
\renewcommand\arraystretch{1}  
\Omega=
\left(
\begin{array}{ccc}
 0 & 1 & 0 \\
 0 & 0 & 0 \\
 0 & 0 & 0
\end{array}
\right)
\quad\mbox{ and }\quad
\mathscr{A}_0=
\left\{
\left.\left(
\begin{array}{ccc}
 a & 0 & b \\
 0 & 0 & 0 \\
 0 & 0 & c
\end{array}
\right)
\,\right|\,
a,b,c\in\mathbb{C}
\right\}.
\egroup
\end{align}
generates a structure that was studied in \cite{SYKES2023108850}. These two structures are compared in \cite[Examples 8.2 and 8.3]{SYKES2023108850}. Their respective symmetry algebras have dimension  $16$ and $14$.

\section{Sequences of flat structures in higher dimensions}\label{Sequences of flat structures in higher dimensions}

In \cite{porter2021absolute}, for every integer $n\geq 3$, several $(2n+1)$-dimensional homogenous  $2$-nondegenerate  hypersurf\-ace\--type CR manifolds are explicitly described as flat structures generated by regular CR symbols. In terms of extensions and linkings of ARMS, for the structures with $n>4$ and Levi kernel rank $1$, all examples in \cite{porter2021absolute} are flat structures generated by combinations of linking and extending the ARMS equal to regular CR symbols whose flat structures are hypersurfaces in $\mathbb{C}^3$, $\mathbb{C}^4$, or $\mathbb{C}^5$. In this section, we introduce two sequences of examples of indecomposable ARMS, showing empirically that indecomposable ARMS generating hypersurfaces in $\mathbb{C}^{n+1}$ exist for all $n\geq 2$. 

Lemmas \ref{nilpotent model sequence}  and  \ref{nilpotent symbols that do not have nilpotent ARMS lemma} each give a sequence of indecomposable ARMS whose CR symbols are the same. The ARMS of \ref{nilpotent model sequence} have nilpotent bases, whereas the ARMS of \ref{nilpotent symbols that do not have nilpotent ARMS lemma} do not have completions with a nilpotent base, which shows that Lemmas \ref{nilpotent model sequence}  and  \ref{nilpotent symbols that do not have nilpotent ARMS lemma} indeed describe non-equivalent structures.

In what follows we let $T_k$ denote the $k\times k$ nilpotent matrix in Jordan normal form whose eigenspace is $1$-dimensional, and let $\mathrm{Diag}(a_1,\ldots, a_k)$ denote the diagonal matrix whose $(j,j)$ entry is $a_k$, i.e.,
\[
\bgroup
\renewcommand\arraystretch{1}  
T_k=
\left(
\begin{array}{ccccc}
0 & 1 & 0 & \cdots & 0 \\
\vdots & \ddots & \ddots &\ddots & \vdots\\
\vdots&&\ddots&\ddots& 0\\
\vdots&&&\ddots&1\\
0&0&\cdots&0&0
\end{array}
\right)
\quad\mbox{ and }\quad
\mathrm{Diag}(a_1,\ldots, a_k)=
\left(
\begin{array}{cccc}
a_1 & 0 & \cdots & 0\\
0 & a_2 & \ddots&\vdots \\
\vdots & \ddots& \ddots &0 \\
0 & \cdots &0 & a_k
\end{array}
\right).
\egroup
\]

The following two lemmas give non-equivalent examples of $(2n+1)$-dimensional flat structures generated by ARMS of the form $(H_\ell, A, \Omega,\mathscr{A}_0)$ with
\begin{align}\label{nilpotent symbol}
(H_\ell)_{i,j}=\delta_{i+j,n}
\quad\mbox{ and }\quad
A=T_{n-1}.
\end{align}

\begin{lemma}\label{nilpotent model sequence}
For each $n\geq 3$, there exists a sequence $(\omega_1,\ldots,\omega_{n-1})$ such that the ARMS encoded by \eqref{nilpotent symbol} and
\begin{align}\label{nilpotent general d symbol}
\Omega=\mathrm{Diag}(\omega_1,\ldots, \omega_{n-1})T_{n-1}
\quad\mbox{ and }\quad
\mathscr{A}_0=0
\end{align}
is a nilpotent subalgebra of $\mathfrak{g}_{-}\rtimes\mathfrak{csp}(\mathfrak{g}_{-1})$, and hence the flat structure generated by this ARMS has the CR symbol encoded by  \eqref{nilpotent symbol}. A defining formula for $(\omega_1,\ldots,\omega_{n-1})$ is given in \eqref{nilpotent examples Sol}.
\end{lemma}
\begin{proof}
For $n=3$, the lemma is satisfied by the type \Rmnum{7} structure of Table \ref{main theorem table}. For $n=4$, the lemma is satisfied by Example \ref{9 dimensional submaximal nilpotent example 2}. For $n>4$, we  define such sequences $(\omega_1,\ldots,\omega_{n-1})$ as follows. 

Let $(\kappa_i)$ be the sequence indexed by $\mathbb{Z}$ given by 
\[
\kappa_1=\frac{1}{\sqrt{3}},
\quad
\kappa_2=\frac{2}{\sqrt{3}},
\quad
\kappa_3=\frac{1}{\sqrt{3}},
\quad
\quad\mbox{ and }\quad
\kappa_i=-\kappa_{i+3}
\quad\quad\forall \,i\in\mathbb{Z}.
\]
Consider now the six vectors $v_i^r$ of some length $n-1$ given by
\begin{align}
v_i^r=(\kappa_{r+i},\ldots, \kappa_{r+i+n-1}) \quad\quad\forall\, i\in \{0,\ldots, 5\}
\end{align}
for some integer $r$, and notice that
\begin{align}\label{commutator property for nilpotent examples}
v_1^r\odot v_3^r-v_2^r\odot v_2^r=-(1,\ldots, 1),
\end{align}
independently of $r$, where $v_i\odot v_j$ denotes the Hadamard product. Define 
\begin{align}\label{nilpotent examples Sol}
(\omega_1,\ldots, \omega_{n-1}):=
\begin{cases}
v_{\left[\tfrac{7-n}{2}\right]_6}^0 & \mbox{ if $n$ is odd} \\
\left(\tfrac{n}{2}-1, \tfrac{n}{2}-2,\ldots, 1-\tfrac{n}{2}\right) &\mbox{ if $n$ is even}
\end{cases}
\end{align}
where $\left[\tfrac{7-n}{2}\right]_6$ denotes the equivalence class of $\tfrac{7-n}{2}$ reduced modulo $6$.
It is straightforward now to check that the ARMS given by  \eqref{nilpotent general d symbol} and \eqref{nilpotent examples Sol} indeed defines a subalgebra of  $\mathfrak{g}_{-}\rtimes\mathfrak{csp}(\mathfrak{g}_{-1})$.

We will show now that indeed this choice of $\omega_i$ works. Letting $k$ denote the length of the vectors $v_i^r$, define the permutation $R :\{v_0^r,\ldots, v_5^r\}\to \{v_0^r,\ldots, v_5^r\}$ by 
\[
R (\kappa_{r},\ldots, \kappa_{r+n-1}):=(\kappa_{r+n-2},\kappa_{r+n-3},\ldots,\kappa_{r-1} ).
\]
In other words,
\[
R (v^r_i)=v^r_{[5-n-i-2r]_6}\quad\quad\forall\, k\in\mathbb{N}, \,r\in\mathbb{Z},\,i\in \{0,\ldots, 5\}
\]
where the bracket notation denotes the integer $3-k-i-2r$ reduced modulo $6$.

Applying this last formula, we get
\begin{align}
&-\left[
\mathrm{Diag}(v_i^r) T_{n-1},\mathrm{Diag}\big(R(v_{i}^{r})\big) T_{n-1}
\right]=\\
&\quad\quad\quad=
\mathrm{Diag}\left(v_i^r\odot v_{[5-n-2r-i]_6}^{r+4}-v_{[5-n-2r-i]_6}^{r+3}\odot v_{i}^{r+1}\right)T_{n-1}^2 \\
&\quad\quad\quad=\mathrm{Diag}\left(v_i^r\odot v_{[9-n-2r-i]_6}^{r}-v_{[8-n-2r-i]_6}^{r}\odot v_{i+1}^{r}\right)T_{n-1}^2 ,
\end{align}
and hence, by \eqref{commutator property for nilpotent examples}, if $n$ is odd then 
\begin{align}\label{commutator property Sol for nilpotent examples odd case}
-\left[
\mathrm{Diag}\left(v_{\left[\tfrac{7-n}{2}\right]_6}^0\right) T_{n-1},\mathrm{Diag}\left(R\left(v_{\left[\tfrac{7-n}{2}\right]_6}^0\right)\right) T_{n-1}
\right]=-T_{n-1}^2.
\end{align}
Similarly, if $n$ is even then it is straightforward to check that
\begin{align}\label{commutator property Sol for nilpotent examples even case}
\left[
\mathrm{Diag}\left(\tfrac{n}{2}-1, \tfrac{n}{2}-2,\ldots, 1-\tfrac{n}{2}\right)  T_{n-1},\mathrm{Diag}\left(\tfrac{n}{2}-2, \tfrac{n}{2}-3,\ldots, -\tfrac{n}{2}\right)  T_{n-1}
\right]=-T_{n-1}^2.
\end{align}
With $\Omega$, $H_\ell$, and $A$ as in \eqref{nilpotent symbol}, \eqref{nilpotent general d symbol}, and \eqref{nilpotent examples Sol} we get that $[\Omega,-\overline{H_\ell}^{-1}\Omega^*\overline{H_\ell}]$ equals the left side of \eqref{commutator property Sol for nilpotent examples odd case} or \eqref{commutator property Sol for nilpotent examples even case} depending on the parity of $n$. It then follows readily from \eqref{commutator property Sol for nilpotent examples odd case} or \eqref{commutator property Sol for nilpotent examples even case} that
\begin{align}\label{compatibility criteria}
\left[
\left(\begin{array}{cc}
\Omega & A  \\
0 & -{H_\ell}^{-1}\Omega^{T} H_\ell
\end{array}\right)
,
\left(\begin{array}{cc}
 -\overline{H_\ell}^{-1}\Omega^{*} \overline{H_\ell}& 0 \\
\overline{A } & \overline{\Omega } 
\end{array}\right)
\right]=0.
\end{align}
Noting the matrix representation of this ARMS given in Section \ref{Matrix representations of ARMS}, it follows that $[\mathfrak{g}_{0}^{\mathrm{red}},\mathfrak{g}_{0}^{\mathrm{red}}]=0$, so $\mathfrak{g}_{0}^{\mathrm{red}}$ is indeed a subalgebra of $\mathfrak{csp}(\mathfrak{g}_{-1})$. Lastly, nilpotency of $\mathfrak{g}^{0,\mathrm{red}}$ follows from $[\mathfrak{g}_{0}^{\mathrm{red}},\mathfrak{g}_{0}^{\mathrm{red}}]=0$.
\end{proof}

\begin{lemma}\label{nilpotent symbols that do not have nilpotent ARMS lemma}
There exists a unique sequence $(\omega_1,\omega_2,\ldots)$ such that for any odd integer $n>3$ the ARMS encoded by the $(n-1)\times(n-1)$ matrices \eqref{nilpotent symbol},
\begin{align}\label{nilpotent symbol with non-nilpotent ARMS a}
\Omega&=\mathrm{Diag}(n-2,n-4,\ldots,2-n) +\\
& \indent+\sum_{k=1}^{(n-1)/2} \mathrm{Diag}\left(\tfrac{n-2k-1}{2}\omega_k,\left(\tfrac{n-2k-1}{2}-1\right)\omega_k,\ldots, \left(\tfrac{n-2k-1}{2}-n-1\right)\omega_k\right)T_{n-1}^{2k},
\end{align}
and
\begin{align}\label{nilpotent symbol with non-nilpotent ARMS b}
\mathscr{A}_0= \mathrm{span}\left\{\left. \mathrm{Diag}(a,-a,a,-a,\ldots, (-1)^{n}a) T^{2k+1}_{n-1}\,\right|\,a\in\mathbb{C}\,,k\in\mathbb{N}\right\}
\end{align}
is a subalgebra of $\mathfrak{g}_{-}\rtimes\mathfrak{csp}(\mathfrak{g}_{-1})$, and hence the flat structure generated by this ARMS has the CR symbol encoded by  \eqref{nilpotent symbol}. The sequence $\{\omega_s\}$ begins
\[
\omega_1=\frac{1}{2},
\quad
\omega_2=-\frac{1}{12},
\quad
\omega_3=\frac{1}{40},
\quad
\omega_4=-\frac{23}{2520},
\quad
\omega_5=\frac{67}{18144},
\quad\ldots
\] 
and is, in general, defined by the recursive formula
\begin{align}\label{nilpotent-nonnilpotent example defining formula}
\omega_s:=\frac{s }{2-4s} \sum_{k=1}^{s-1} \omega_{s-k}\omega_{k} 
\quad\quad\forall \,s>1.
\end{align}
\end{lemma}
\begin{proof}
This proof is a matter of direct calculation, which we outline here in sufficient detail to reproduce. Some explicit formulas are, however, omitted.

Apply induction on $n$. For the base case of induction, first assume $n=5$, and set $\omega_{1}=\frac{1}{2} $.
In this case, by direct calculation with  \eqref{nilpotent symbol} and \eqref{nilpotent symbol with non-nilpotent ARMS a} we get
\begin{equation}\label{02 and 0-2 brackets}
\small
\begin{split}
\left[
\left(\begin{array}{cc}
\Omega & A  \\
0 & -{H_\ell}^{-1}\Omega^{T} H_\ell
\end{array}\right)
,
\left(\begin{array}{cc}
 -\overline{H_\ell}^{-1}\Omega^{*} \overline{H_\ell}& 0 \\
\overline{A} & \overline{\Omega} 
\end{array}\right)
\right]&=
-2 \left(\begin{array}{cc}
\Omega & A  \\
0 & -{H_\ell}^{-1}\Omega^{T} H_\ell
\end{array}\right)+\\
&\indent
+2\left(\begin{array}{cc}
 -\overline{H_\ell}^{-1}\Omega^{*} \overline{H_\ell}& 0 \\
\overline{A} & \overline{\Omega} 
\end{array}\right)
\end{split}
\end{equation}
and
\begin{align}\label{02 and 00 brackets}
\left[
\left(\begin{array}{cc}
\alpha& 0 \\
0 & -{H_\ell}^{-1}\alpha^{T} H_\ell
\end{array}\right),
\left(\begin{array}{cc}
\Omega & A  \\
0 & -{H_\ell}^{-1}\Omega^{T} H_\ell
\end{array}\right)
\right]
=
\left(\begin{array}{cc}
\beta& 0 \\
0 & -{H_\ell}^{-1}\beta^{T} H_\ell
\end{array}\right),
\end{align}
where
\begin{align}
\alpha=
\left(
\begin{array}{cccc}
0 & a_1 & 0 & a_2 \\
0 & 0 & -a_1 & 0 \\
0 & 0 & 0 & a_1 \\
0 & 0 & 0 & 0
\end{array}
\right)
\,\,\mbox{ and }\,\,
\beta=
\left(
\begin{array}{cccc}
0 & -2 a_1 & 0 & 6 a_2 + a_1/2\\
0 & 0 & 2 a_1 & 0 \\
0 & 0 & 0 & -2 a_1 \\
0 & 0 & 0 & 0
\end{array}
\right).
\end{align}
It follows that $[\mathfrak{g}_{0,-}^{\mathrm{red}},\mathfrak{g}_{0,0}^{\mathrm{red}}]\subset\mathfrak{g}_{0,0}^{\mathrm{red}}$, and hence $\mathfrak{g}_{0}^{\mathrm{red}}$ is indeed a subalgebra of $\mathfrak{g}_{-}\rtimes\mathfrak{csp}(\mathfrak{g}_{-1})$ when $n=5$.

Proceeding, suppose that for some odd integer $k> 5$, we have chosen $\omega_1,\ldots, \omega_{(k-5)/2}$ such that the lemma holds for all odd $n<k$, and now let us assume $n=k$. We will see that this assumption implies  \eqref{02 and 0-2 brackets}, but, to begin, let us furthermore assume that our choice satisfies  \eqref{02 and 0-2 brackets} for all $n<k$. We will choose $\omega_{(n-3)/2}$ so that  \eqref{02 and 0-2 brackets} holds with $n=k$, and need to show that such a choice exists. 

By direct calculation, one finds that the upper right and lower left $n\times n$ blocks in the matrix equation \eqref{02 and 0-2 brackets} holds for any choice of $\omega_{(n-3)/2}$. Using the assumption that \eqref{02 and 0-2 brackets} holds for all $n<k$ with our choice of $\omega_1,\ldots, \omega_{(k-5)/2}$, it is also straightforward to calculate that with $n=k$ the $(i,j)$ scalar component of \eqref{02 and 0-2 brackets} holds for each $(i,j)\not\in \{(1,n-2),(2,n-1),(n,n-2),(n,n-1)\}$, that is for every scalar component in which $\omega_{(n-3)/2}$ does not appear on either side of \eqref{02 and 0-2 brackets}. Lastly, each $(i,j)$ scalar component of \eqref{02 and 0-2 brackets}  for $(i,j)\in \{(1,n-2),(2,n-1),(n,n-2),(n,n-1)\}$ gives the same defining equation for $\omega_{(n-3)/2}$, namely, \eqref{nilpotent-nonnilpotent example defining formula} with $s=(n-3)/2$. In this way, $\omega_{(n-3)/2}$ is uniquely determined. With this $\omega_{(n-3)/2}$ set, another direct calculation of the Lie bracket in \eqref{02 and 00 brackets} with $\alpha$ taken as an arbitrary matrix in $ \mathscr{A}_0$ shows that the ARMS is indeed closed under Lie brackets.
\end{proof}

By linking and extending the ARMS of Lemma \ref{nilpotent model sequence}, we obtain the following theorem.
\begin{theorem}\label{nilpotent CR symbols have hmodels}
Every CR symbol encoded by $(H_\ell, A)$ with $A$ nilpotent can be obtained from a homogeneous $2$-nondegenerate hypersurface.
\end{theorem}
\begin{proof}
If $A$ has a $1$-dimensional eigenspace, then Lemma \ref{nilpotent model sequence} gives an example of an ARMS whose flat structures has the corresponding CR symbol $(H_\ell, A)$. For  an arbitrary nilpotent $A$, by \cite[Theorem 2.2]{sykes2020canonical}, we  can assume without loss of generality that $A$ is in Jordan normal form, and thus represents the CR symbol of an ARMS obtained by linking and extending ARMS of the type in Lemma \ref{nilpotent model sequence}. For this last point, one needs that indeed all  ARMS in Lemma \ref{nilpotent model sequence} are compatible (in the sense of Definition \ref{linking of ARMS}), and indeed the needed compatibility criteria of Definition \ref{compatibility criteria definition} follow readily from \eqref{compatibility criteria}.
\end{proof}

\begin{remark}
In the very recent work \cite{kolar2023new}, for every CR symbol of the form in Theorem \ref{nilpotent CR symbols have hmodels}, we obtain local coordinate descriptions given as defining equations of large classes of homogeneous $2$-nondegenerate CR hypersurfaces with the given symbol.
\end{remark}

\section{Acknowledgments}
The author would like to thank Igor Zelenko for many helpful discussions on this topic. The author was supported by the GACR grant GA21-09220S.

 \newcommand{\noop}[1]{}

\end{document}